\documentclass[english,oneside,reqno]{amsart}
\usepackage{hyperref}
\usepackage[utf8]{inputenc} 
\usepackage[english]{babel}
\usepackage[T1]{fontenc}
\usepackage{geometry} 
\usepackage{xcolor}
\usepackage{cancel}
\usepackage{todonotes}

\usepackage{graphicx} 				
\usepackage{booktabs} 				
\usepackage{array}    				
\usepackage[parfill]{parskip}		

\usepackage[normalem]{ulem}
\usepackage{enumitem} 				
\usepackage{nicefrac}
\usepackage{bbm}
\usepackage{tikz}
\usepackage{tikz-cd} 
\usetikzlibrary{decorations.markings,positioning}
\usetikzlibrary{matrix,decorations.pathmorphing,decorations.pathreplacing,calligraphy,}
\usetikzlibrary{calc, intersections,hobby}
\usetikzlibrary{arrows,shapes.geometric,calc}	
\usetikzlibrary{bending}		
\usetikzlibrary{backgrounds, fit}
\usetikzlibrary{spath3}
\usepackage{rank-2-roots}
\usepackage{dynkin-diagrams}
\usepackage{mathrsfs}
\usepackage{dutchcal} 
\usepackage{amsfonts}
\usepackage{amssymb}
\usepackage{amsmath}
\usepackage{ytableau}
\usepackage{amsthm}
\usepackage{thmtools}
\usepackage{textgreek}
\usepackage{newunicodechar}
\usepackage[alphabetic, initials]{amsrefs}
\usepackage{hyperref}
\usepackage[capitalize, nameinlink]{cleveref}
\crefalias{lemma}{lem}
\theoremstyle{plain}
\newtheorem{lem}{Lemma}[section]
\newtheorem{lemma}[lem]{Lemma}
\newtheorem{prop}[lem]{Proposition}
\newtheorem{cor}[lem]{Corollary}
\newtheorem{thm}[lem]{Theorem}
\newtheorem{thmintro}{Theorem}

\theoremstyle{definition}
\newtheorem{definition}[lem]{Definition}
\newtheorem{remark}[lem]{Remark}

\newtheorem{notation}[lem]{Notation}

\newtheorem{eg}[lem]{Example}
\newtheorem{convention}[lem]{Convention}
\newtheorem{conjecture}[lem]{Conjecture}

\crefname{align}{}{}
\newcommand{\ytabcenter}[1]{
\,\ytableausetup{smalltableaux}
	\vcenter{\hbox{$\begin{ytableau}#1\end{ytableau}$}}\,
}

\newcommand{\tempGraph}{\tilde{\linearGraph}}

\newcommand*\circled[1]{\tikz[baseline=(char.base)]{
            \node[shape=circle,draw, inner sep=1pt] (char) {#1};}}

\newcommand{\nocontentsline}[3]{}
\let\origcontentsline\addcontentsline	
\newcommand\stoptoc{\let\addcontentsline\nocontentsline}
\newcommand\resumetoc{\let\addcontentsline\origcontentsline}

\newcommand{\g}{\mathfrak{g}}
\newcommand{\classicalFixPointSubalg}[1][\linearGraph]{\g^\theta_{#1}}

\newcommand{\gr}{\operatorname{gr}}

\newcommand{\bo}{B_0}

\newcommand{\dCoideal}{\hat{D}}

\newcommand{\bbZ}{\mathbb{Z}}
\newcommand{\bbN}{\mathbb{N}}

\newcommand{\coideal}{U_q'}

\newcommand{\inv}{^{-1}}

\newcommand{\specht}{\mathbb{S}}

\newcommand{\basis}{v}

\newcommand{\kCoideal}{\hat{K}}
	
\newcommand{\content}{\operatorname{cont}}

\newcommand{\kmatrix}{\mathbb{K}}

\newcommand{\id}{\operatorname{id}}

\newcommand{\negidx}[1]{\overline{#1}}
\newcommand{\half}{{\diamondsuit}}
\newcommand{\heckeSpecialized}[1]{\mathscr{H}_{1,q}(B_{#1})}

\newcommand{\graphGeneral}{\mathtt{\Gamma}}
\newcommand{\linearGraph}{\mathtt{\Gamma}}
\newcommand{\source}[1]{\sigma(#1)}
\newcommand{\target}[1]{\tau(#1)}

\newcommand{\graphVertices}{\mathcal{P}}
\newcommand{\graphEdges}{\mathcal{I}}
\newcommand{\cont}{\operatorname{cont}}

\newcommand{\vecrep}[1][]{\mathbb{V}_{#1}}

\newcommand{\C}{\mathbb{C}}
\newcommand{\Z}{\mathbb{Z}}

\newcommand{\N}{\mathbb{Z}_{\geq 0}}

\newcommand{\gl}{\mathfrak{gl}}

\renewcommand{\sl}{\mathfrak{sl}}

\newcommand{\Uq}{U_q}

\newcommand{\dottedArrowRight}{
\geq
}

\newcommand{\catFont}[1]{\mathit{#1}}

\newcommand{\sgn}{\epsilon}

\newcommand{\qbinom}{ 
	\genfrac{[}{]}{0pt}{}
}

\newcommand{\Qint}[1]{[#1]_{\textit{\tiny Q}}}
\newcommand{\Qbinom}[2]{\qbinom{#1}{#2}_Q}

\newcommand{\undercrossing}[4]{
	\draw[in=-90,out=90]  (#1) to (#3) ;
	\draw	[overline,in=-90,out=90](#2) to (#4);
}
\newcommand{\overcrossing}[4]{
		\draw[in=-90,out=90] (#2) to (#4);
		\draw[overline,in=-90,out=90] (#1) to (#3);
}

\newcommand{\webQpar}{Q}

\newcommand{\Kar}{\operatorname{Kar}}

\newcommand{\Ijbar}[2]{\flip{#2}{#1}}

\newcommand{\dotNotationText}[1][k]{~\begin{tikzpicture}[baseline=0.5ex, every path/.style={thick},scale=0.35]
\draw (0,0) node[right]{\tiny$#1$} to (0,1);
\draw[fill] (0,0.5) circle (0.1);
\end{tikzpicture}}
\newcommand{\dimV}{n}
\newcommand{\foldVertex}[1]{\overline{#1}}
\newcommand{\bnumber}{b}

\newcommand{\finker}{f}

\newcommand{\qwedge}{\Lambda}
\newcommand{\monWedge}{\qwedge}

\newcommand{\webCategoryB}{\mathit{Web}^B}

\newcommand{\webCenter}[3]{
	\coordinate[inner sep=0] (WebCenter) at (barycentric cs:#1=0.8,#2=0.3,#3=0.3);
}

\newcommand{\webCategoryA}{\mathit{Web}^A}

\renewcommand{\hom}{\operatorname{Hom}}

\newcommand{\Aring}{{\mathcal{R}}}
\newcommand{\webker}{\mathit{M}}

\newcommand{\groundring}{\Aring}

\newcommand{\Fund}[1][\linearGraph]{\mathit{Fund}_{#1}}
\newcommand{\polyRepCat}[1][{\classicalFixPointSubalg[\linearGraph]}]{\catFont{Poly}({#1})}

\newcommand{\funcField}{{\C(q)}}
\newcommand{\funcA}{\Phi^A}

\newcommand{\End}{\operatorname{End}}
\newcommand{\symFuncElmtry}{{e}}

\newcommand{\matrixElmtFont}[1]{\mathsf{#1}}

\newcommand{\splitNotation}{{\mathsf{Y}}}

\newcommand{\mergeNotation}{\text{\rotatebox[origin=c]{180}{$\mathsf{Y}$}}}

\newcommand{\splitWebTextStyle}[3]{
		\begin{tikzpicture}[scale=0.35,anchorbase]
				
				\node (ab) at (0,-2) {$#1$}; 
				\node (a) at (-1,1) {$#2$};
				\node (b) at (1,1) {$#3$};
				\webCenter{ab}{a}{b}
				
				\path[very thick,out=-90,in=90,]	(WebCenter)	   edge  (ab)
				[out=130,in=-90,]		edge (a)
				[out=50,in=-90,]		edge (b);

			\end{tikzpicture}
}

\newcommand{\mergeWebTextStyle}[3]{
	\begin{tikzpicture}[scale=0.35,anchorbase]
						
						\node (ab) at (0,1) {$#1$}; 
						\node (a) at (-1,-2) {$#2$};
						\node (b) at (1,-2) {$#3$};
						\webCenter{ab}{a}{b};

						\path[very thick,out=90,in=-90,] (WebCenter)  edge  (ab);
						\path[draw,in=-130,out=90, 
						] (a) to	(WebCenter) ;
						\path[draw,in=-50,out=90, ]	(b) to (WebCenter) ;

					\end{tikzpicture}
}

\newcommand{\intUq}{U'_{\groundring}}

\newcommand{\rank}{\operatorname{rank}}

\newcommand{\img}{\operatorname{Im}}
\newcommand{\coker}{\operatorname{coker}}

\newcommand{\boxdiagramSet}{\mathcal{D}_{\square}}

\newcommand{\unfoldSupScript}{\mathrm{un}}

\newcommand{\hecke}[1][d]{\mathscr{H}_q(S_{#1})}

\newcommand{\FundCoideal}{\Fund'}

\newcommand{\smiley}{\smile}

\usepackage{wasysym}
\usepackage{marvosym}

\let\oldbigwedge\bigwedge
\renewcommand{\bigwedge}{{\textstyle\oldbigwedge}}

\newcommand{\functorWebB}{\Phi^B}

\newcommand{\elmntSn}{w}
\newcommand{\abshuffleSet}{\mathrm{Sh}}
\newcommand{\dSh}{{\abshuffleSet^\bullet}}
\newcommand{\len}{\operatorname{len}}

\newcommand{\interval}{\mathfrak{I}}
\newcommand{\webnotation}{\mathfrak{w}}

\newcommand{\matUnit}{\mathsf{E}}
\newcommand{\supp}{\operatorname{supp}}
\newcommand{\lgl}{\matrixElmtFont{d}}

\newcommand{\Aalg}{{A}}
\newcommand{\kleinCSA}{\mathrm{H}}
\newcommand{\delustig}{\partial^L}
\newcommand{\Upoly}{U^{poly}}
\newcommand{\letzterIso}{\hat{\psi}}

\newcommand{\irrquot}{\mathbb{L}}

\newcommand{\hwfont}[1]{\underline{#1}}

\newcommand{\UmPlus}{U^{-}_{\succ}}
\newcommand{\UmMinus}{U^{-}_{\prec}}
\newcommand{\UmZero}{U^-_{\circ}}

\newcommand{\rtPos}[2]{\alpha_{[#1,#2]}}

\newcommand{\PBWset}{\mathsf{PBW}}

\newcommand{\posCoideal}{U'_{>}}
\newcommand{\negCoideal}{U'_{<}}
\newcommand{\csaCoideal}{U'_{\circ}}

\newcommand{\verma}{\mathbb{M}}

\newcommand{\SR}{\operatorname{Serre}}
\newcommand{\flip}[2]{#2\updownarrow{#1}}

\crefformat{equation}{#2#1#3}
\Crefformat{equation}{#2#1#3}

\newcommand{\delL}{\partial^{\mathrm L}}
\newcommand{\delR}{\partial^{\mathrm R}}

\newcommand{\posposCoideal}{U'_{\succ,+}}
\newcommand{\negnegCoideal}{U'_{\prec,-}}

\newcommand{\wtZero}{U^\kleinCSA}
\newcommand{\pr}{\operatorname{pr}}

\newcommand{\ersatzUoUplus}{\widehat{U}'_>}

\newcommand{\hwboxdiag}[1][k]{{}_{#1}\boxdiagramSet^{\mathrm{hw}}}
\allowdisplaybreaks

\newcommand*{\StrikeThruDistance}{0.1cm}%
\tikzset{
	strike thru/.style={
		postaction={
			decorate,
			decoration={
				markings,
				mark=at position 0.5 with {
					\draw
						++ (-\StrikeThruDistance,-\StrikeThruDistance)
						-- (\StrikeThruDistance,\StrikeThruDistance);
				}
			}
		}
	},
	anchorbase/.style={baseline={([yshift=-0.5ex]current bounding box.center)}},
	every path/.style={draw,very thick},
	every picture/.style={anchorbase,auto,font=\tiny},
	overline/.style={preaction={draw,white,line width=4pt}},
	>=To,
}

\tikzset{
	webdotted/.style={
    postaction={decorate},
    decoration={markings,
      mark=at position #1 with {\fill (0,0) circle (1.2pt);}
    }
  	},
	webdotted/.default=0.08,
	every node/.style={anchor=center},
}

\crefformat{equation}{(#2#1#3)}
\crefrangeformat{equation}{(#3#1#4) to (#5#2#6)}
\crefmultiformat{equation}{(#2#1#3)}{ and (#2#1#3)}{, (#2#1#3)}{ and (#2#1#3)}

\ytableausetup{mathmode,boxsize=.9em}
\title[Webs and finite dimensional Representation theory for quantum symmetric pairs]{Webs and finite dimensional Representation theory\\ for quantum symmetric pairs}
\author{Liao Wang}
\address{LW: Mathematisches Institut, University of Bonn, Endenicher Allee 60, 53115 Bonn, Germany, s6liwang@math.uni-bonn.de}
\begin{document}
\maketitle

\begin{abstract}
	We develop a highest weight framework and construct explicit finite dimensional irreducible modules for the type AIII (quasi split) quantum symmetric pair subalgebras $\coideal(\gl_n^\theta)$. 
	In Part~1, we provide a uniform construction of Cartan subalgebras and Verma modules for both subfamilies. Our construction recovers Letzter's Cartan subalgebra in one subfamily and gives a new Cartan subalgebra in the other. 
	We construct a weak triangular decomposition of $\coideal(\gl_n^\theta)$ via the Letzter map. Then we define Verma modules as induced modules and prove that finite dimensional simple $\coideal(\gl_n^\theta)$-modules are quotients of our Verma modules. 
	In Part~2, we define a diagrammatic category $\webCategoryB$ that controls polynomial representations of $\coideal(\gl_n^\theta)$. We prove a multiplicity-free decomposition of the quantum exterior powers $\bigwedge^k\vecrep$ as a $\coideal(\gl_n^\theta)$-module, using a explicit eigenspace decomposition of certain  dot morphisms in $\webCategoryB$. 
	As applications, we deduce the multiplicity-freeness of the involved anti-spherical Hecke module. For $\bigwedge^k\vecrep$, we explicitly determine the highest weight vectors and their weights, thereby determining the isomorphism classes of their irreducible summands.
	Finally, we determine the kernel of the $\webCategoryB$-action on certain finite dimensional representation category of $\coideal(\gl_n^\theta)$.
\end{abstract}

\tableofcontents

\renewcommand{\dimV}{r}

\section{Introduction}
Let $\g$ be a complex semisimple Lie algebra.
The associated Drinfeld--Jimbo quantum group $\Uq(\g)$ is a certain Hopf algebra deformation of the universal enveloping algebra $U(\g)$ of $\g$. 
The quantum group $\Uq(\g)$ comes with a distinguished presentation which lifts the Serre presentation of $U(\g)$. 
Given further a Lie algebra involution $\theta$ of $\g$, there is an attached a symmetric pair $(\g,\g^\theta)$ where $\g^\theta\subset \g$ is the fixed point subalgebra under $\theta$. 
The quantum analogue of the embedding $U(\g^\theta) \hookrightarrow U(\g)$ is 
the \emph{quantum symmetric pair} $(\Uq(\g), \coideal(\g^\theta))$, where $\coideal(\g^\theta)$ is a certain \emph{coideal} subalgebra of $\Uq(\g)$ 
first introduced by Letzter in \cite{Letzter-symmetric-pairs-quantized-enveloping} and \cite{Letzter-coideal-subalgebras}. We refer to \cite{kolbQKacMoodyPairs} for a detailed treatment and classification of quantum symmetric pairs. A uniform Serre-type presentation in finite types was achieved in \cite{Letzter-symmetric-pairs-quantized-enveloping}. Since then, there have been rapid developments in the structure theory of quantum symmetric pairs by \cite{BKolb},  \cite{BaoWang}, \cite{EhSt-nw-algebras-howe}, \cite{Le-cartan-coideal}, \cite{BaKo-universal-k}, \cite{Kolb-Yakimov-defining-rel-QSP}, \cite{VlAp-kac-moody-univeral-K}, \cite{Ko-plain-sight}, and \cite{kolb2026short-star-products}.
Recently works  \cite{ARK-branching-rules}, \cite{Watanabe-crystal-basis-quantum-symmetric-pair}, \cite{Wenzl}, \cite{Watanabe2}, \cite{Watanabe1}, \cite{Watanabe0}, \cite{StWoj-coidealDiagrammatics}, \cite{Kolb-Stephens-very-non-standard}, \cite{StWweight-mod-qsp}, developed various aspects of the finite dimensional representation theory of $\coideal(\g^\theta)$.

In this paper we consider the quasi-split type AIII case. In small examples, \cite{ARK-branching-rules}, \cite{StWoj-coidealDiagrammatics} and \cite{StWweight-mod-qsp} classified good irreducible finite dimensional representations. In the second subfamily in \cref{satakes}, \cite{Watanabe-crystal-basis-quantum-symmetric-pair} gave a highest-weight classification of irreducible polynomial representations also in higher ranks.
The major difficulty here is that Letzter's Cartan subalgebras introduced in \cite{Le-cartan-coideal} are not generated by grouplike elements and thus do not allow the usual weight space decompositions. 
Moreover, a good notion of a triangular decomposition is also unavailable in general, except a construction of (weak) triangular decomposition and Verma modules for many types in \cite{Watanabe2}. 
Towards a uniform theory of \emph{weight modules}, our first main idea is a novel application of the Letzter map approach developed (following the spirit of the projection method in \cite{Le-cartan-coideal}) in \cite{Kolb-Yakimov-defining-rel-QSP} and \cite{kolb2026short-star-products}. 
We obtain a uniform construction of Cartan subalgebras and Verma modules in both subfamilies (see \cref{satakes}) of quasi-split type AIII quantum symmetric pairs.

Complementary to the Letzter map approach which does not yet provide explicit descriptions of irreducible modules, we also take a diagrammatic point of view analogous to \cite{CaKaMo-webs-skew-howe} to construct explicit irreducible finite dimensional representations. 
A diagrammatic monoidal category was introduced in \cite{CaKaMo-webs-skew-howe} to control quantum $\sl_n$ representations.
A slight modification of this category, denoted $\webCategoryA$, contains both the type $A$ Hecke algebra and the corresponding $q$-Schur algebra, see \cite{Br-schur}. 
Via the Schur--Weyl duality and skew Howe duality from \cite{CaKaMo-webs-skew-howe}, $\webCategoryA$ governs all $\Uq(\gl_n)$-module maps between tensor powers of fundamental representations.
From this point of view, for example, the question of decomposition of tensor products of fundamental representations into simple modules is equivalent to the computation of certain idempotents in $\webCategoryA$. Although this is typically hard, it sometimes can be done inductively thanks to the monoidal structure. The famous Jones--Wenzl projectors provide such examples, see e.g. \cite{CarScSai-6j} and the references therein. analogous projectors in type $B$ and $D$ also exist in \cite{StWoj-coidealDiagrammatics}. 
The corresponding version of skew Howe duality from \cite{EhSt-nw-algebras-howe} and the Schur--Weyl duality from \cite{BaoWang}, \cite{EhSt-nw-algebras-howe} is the starting point of our work. As a result, we obtain interesting interaction between representation theory and diagrammatical categories.

To understand our results, we explain first their classical counterparts. For quasi-split type AIII symmetric pairs, we have $\g=\gl_n$ and  $\g^\theta=\gl_n^\theta\cong \gl_r\times \gl_{n-r}$ where $r = \lfloor n/2\rfloor$. 
Consider the matrix $\kmatrix= \left(\begin{smallmatrix}
	&&1\\
	&\text{\rotatebox{90}{$\ddots$}}&\\
	1
\end{smallmatrix} \right)\in \End_\C(\C^n)$, we have a decomposition $\C^n=W^+\oplus W^-$ where $W^+ ,W^-$ are the $\pm1$ eigenspaces of $\kmatrix$. Define $\kmatrix_k \in \End_{\C} (\bigwedge^k\C^n)$ to be the following composition of maps: 
\begin{equation}
	\label{K-mat-on-k}
	\begin{tikzcd}[font=\normalsize]
		\bigwedge^k \C^n \arrow[r,"\splitNotation"] &
		 \C^n\otimes  \bigwedge^{k-1} \C^n \arrow[r,"\kmatrix\otimes\id"] & \C^n\otimes  \bigwedge^{k-1} \C^n  \arrow[r,"\mergeNotation"]  & 
		 \bigwedge^{k} \C^n
	\end{tikzcd}
\end{equation}
where $\mergeNotation$ is given by the exterior product. Here, $\splitNotation$ is given by $v_1\wedge\cdots \wedge v_k \mapsto
\sum_j (-1)^{j-1} v_j\otimes v_1\wedge\cdots \hat{v}_j \cdots \wedge v_k$ (where $\hat{v}_j$ denotes omission of the factor $v_j$). These two maps are specialisations of the generating morphisms in $\webCategoryA$, and $\kmatrix$ is $\gl_n^\theta$-linear. Consequently, $\kmatrix_k$ is $\gl_n^\theta$-linear. The following innocent looking observation is the key motivation for us.
\begin{prop}\label{intro-prop}
	The following hold:
	\begin{enumerate}
		\item The action of $\kmatrix_k$ on $\bigwedge^k \C^n $ is diagonalisable.
		\item $\bigwedge^{k-i} W^+ \otimes \bigwedge^{i}W^-$ is the $k-2i$ eigenspace of $\kmatrix_k$. 
		\item The subspaces $\bigwedge^{k-i} W^+ \otimes \bigwedge^{i}W^-\, (k-\dim W^+\leq i\leq \dim W^-)$ are pairwise non-isomorphic irreducible $\gl_n^\theta$-submodules.
		\item The eigenspace decomposition $\bigwedge^k \C^n = \bigoplus \bigwedge^{k-i}W^+ \otimes \bigwedge^{i}W^-$ of $\kmatrix_k$ is also the multiplicity free decomposition into irreducible $\gl_n^\theta$-modules.
	\end{enumerate}
\end{prop}
The isomorphism classes of the irreducible summands $\bigwedge^{k-i} W^+ \otimes \bigwedge^{i}W^-$ are determined by computing the highest weights with respect to $\gl_n^\theta$, see \cref{classical-hw} and \cref{classical-hw-odd}. Toward a quantum analogue of classical highest weight theory, we first develop our Letzter map construction of Verma modules in Part~1 of this paper. We will establish the following results in \cref{verma-properties} and \cref{fd-quotients}.
\begin{thmintro}[Verma modules]\label{verma-intro}
	The following hold:
	\begin{enumerate}
	\item For any highest weight $\hwfont{\lambda}$, $\verma(\hwfont{\lambda})\cong\negCoideal$ as vector spaces.
	\item The PBW basis of any Verma module is a weight basis with respect to $\kleinCSA$.
	\item Let $V$ be any $\coideal(\gl_n^\theta)$-module and $x\in V$ a highest weight vector with weight $\hwfont{\lambda}$. Then there is a unique $\coideal(\gl_n^\theta)$-module homomorphism $\verma(\hwfont{\lambda})\to V$ with $1\otimes 1\mapsto x$.
	\item $\verma(\hwfont{\lambda}) \cong \verma(\hwfont{\mu})$ if and only if $\hwfont{\lambda} =\hwfont{\mu}$. 
	\item $\verma(\hwfont{\lambda})$ has a unique irreducible quotient $\irrquot(\hwfont{\lambda})$.
	\item $\irrquot(\hwfont{\lambda}) \cong \irrquot(\hwfont{\mu})$ if and only if $\hwfont{\lambda} =\hwfont{\mu}$. 
	\item Every finite dimensional irreducible weight representation of $\coideal(\gl_n^\theta)$ appears as the irreducible quotient of some Verma module.
	\end{enumerate}
\end{thmintro}
Within the two subfamilies (see \cref{satakes}) of quasi-split type AIII case,
we dramatically simplify the explicit case-by-case calculation in \cite{Watanabe2}*{\S 4.4, \S4.5} to a short and uniform computation in \cref{Ri-Rj-commute}. 
We expect similar construction to be possible in many other types.

In Part~2 and Part~3, we prove the quantum analogue of \cref{intro-prop}, where a certain morphism $\dotNotationText\in \webCategoryB$ is the $q$ analogue of $\kmatrix_k$ from \cref{K-mat-on-k}.
We do this by constructing an explicit eigenbasis indexed by a certain set $ {}_{k}\boxdiagramSet(\linearGraph)$ of box diagrams from \cite{EhSt-nw-algebras-howe}. 

\begin{thmintro}\label{dotspec}
	The following hold with $Q=-q$:
	\begin{enumerate}
		\item An eigenbasis of $\bigwedge^k\vecrep$ is given by $ \{ \Omega_\beta \mid \beta \in {}_{k}\boxdiagramSet(\linearGraph) \}$, and we have
		\begin{align*}
			 \dotNotationText  (\Omega_\beta)  = \Qint{m_\beta} \Omega_\beta
			,\text{ where } 
			m_\beta = |\beta\inv(+)| + |\beta\inv(0)| - |\beta\inv(-)|
			.
		\end{align*}
		\item For $\coideal(\gl_{2r+2}^\theta)$, the spectrum of~ $\dotNotationText $ on $\bigwedge^k \vecrep$ is 
		\[
			\{\Qint{k-2i}\mid \max(0,k-r-1)\leq i\leq \min(k,r+1)\}, 
		\]
		where $\Qint{k-2i}$ has multiplicity $\binom{r+1}{k-i}\binom{r+1}{i}$. 
		\item For $\coideal(\gl_{2r+1}^\theta)$, the spectrum of~ $\dotNotationText $ on $\bigwedge^k \vecrep$ is 
		\[
			 \{\Qint{k-2i}\mid\max(0,k-r-1)\leq i\leq \min(k,r)\},
		\]
		where $\Qint{k-2i}$ has multiplicity $\binom{r+1}{k-i}\binom{r}{i}$.
		\item The eigenspaces of $\dotNotationText$ in $\bigwedge^k\vecrep$ are exactly its irreducible $\coideal(\gl_n^\theta)$-submodules. 
		\item An explicit subset of ${}_k\boxdiagramSet$ (see \cref{Zaction-picture}) gives rise to all highest weight vectors in $\bigwedge^k\vecrep$. Furthermore, their weights are explicitly given by \cref{hweven} and \cref{hwodd}.
	\end{enumerate}
\end{thmintro}
Together with \cref{verma-intro}, part e.) of \cref{dotspec} determine the isomorphism classes of irreducible summands of $\bigwedge^k\vecrep$. 
The eigenbasis of $\bigwedge^k\vecrep$ we construct is adapted to the study of simple modules, and is therefore substantially different from both the $\iota$-canonical basis in \cite{BaoWang} and the canonical basis given by decategorified classes of  indecomposable projective modules in \cite{EhSt-nw-algebras-howe}. Therefore, our results do not follow from these canonical basis construction in the literature. 

As byproducts, we obtain a multiplicity-free decomposition of an anti-spherical Hecke module in \cref{beef-liver-thm} and a diagrammatical relation in \cref{SURPRISE-SURPRISE} which we were not able to prove purely diagrammatically.
We also determine the kernel $\webker$ of the Schur--Weyl duality/skew Howe duality functor $\functorWebB: \webCategoryB\to \mathit{Fund}'_n$, where $\mathit{Fund}'_n$ is the category  $\mathit{Fund}_n$ of tensor products of fundamental representations of $\Uq(\gl_n)$  restricted to $\coideal(\gl_n^\theta)$. Furthermore, this is compatible with the diagrammatic equivalence $\funcA: \webCategoryA/ \id_{n+1} \cong \mathit{Fund}_n$. 
\begin{thmintro}
	There is an equivalence of categories $\bar\Phi:\webCategoryB_{\funcField}/\webker \cong \FundCoideal$ such that the following diagram commutes: \[
		\begin{tikzcd}[font=\normalsize]
			\webCategoryB_{\funcField}/\webker \times \webCategoryA_{\funcField}/\id_{n+1}  
			\arrow[r]\arrow[d,"\bar\Phi^B\times \bar\Phi^A"] 
			 & \webCategoryB_{\funcField}/\webker
			  \arrow[d,"\bar\Phi^B"] 
			\\
			\mathit{Fund}'_n\times \mathit{Fund}_n \arrow[r,"\otimes"] & \mathit{Fund}'_n
		\end{tikzcd}
	\]
\end{thmintro}

\subsection*{Disclosure of AI use}
We used ChatGPT 5.6 Sol to assist in finding the Cartan elements defined in \cref{Wi-def} in small examples (see \cref{ai}), and to improve the wording of selected passages.

\subsection*{Acknowledgements}
We thank Mao Hoshino and Hideya Watanabe for their patience in explaining some technical details of their work, and Catharina Stroppel, Stefan Kolb, Michael Ehrig, Daniel Tubbenhauer, and Erik Koelink for helpful discussions. We gratefully acknowledge financial support from the Gottfried Wilhelm Leibniz Prize of the German Research Foundation and from the Hausdorff Center for Mathematics in Bonn. The results of this article form part of the author’s PhD thesis, written under the supervision of Catharina Stroppel.

\part{Quantum groups and quantum symmetric pairs}

\section{Preliminaries}

Let $\groundring\subset \C(q)$ be the subring of rational functions in $q$ regular at $q=1$. View $\C$ as an $\groundring$-module via the evaluation map $q\mapsto 1$.
Fix $\webQpar \in \groundring$. Define for  $n\in \Z$ the \emph{quantum integers} and for $m,k\in \N$ the \emph{quantum factorial} and \emph{quantum binomial coefficients} as follows:
\[
	\Qint{n} \coloneqq \frac{\webQpar^n-\webQpar^{-n}}{\webQpar-\webQpar^{-1}}
	, \quad 
	\Qint{m}! \coloneqq \Qint{m} \cdots \Qint{1}
	,\quad 
	\Qbinom{m}{k} \coloneqq \frac{\Qint{m}!}{\Qint{k}!\Qint{m-k}!}
	\in\bbZ[q,q^{-1}] \subset\groundring.
\]	
In the case $\webQpar = q^{\pm1} $ we omit the subscript $\webQpar$.

\subsection{Linear graphs}\label{sec-linear-graph}
\begin{definition}
	A \emph{linear graph} is one of the following four types of graphs, where $\half =\nicefrac{1}{2}$:
\begin{equation}\label{linear-graphs}
	\begin{gathered}
	\graphGeneral_{\dimV} \; = \; \begin{tikzpicture}[baseline={([yshift=-1ex]current bounding box.center)}]
			\newcommand{\leng}{1.5}
			\foreach \i/\j/\k in {7/4/{\dimV},0/{-4}/{\negidx{\dimV}},2/-2/{\negidx{1}},3/-1/{\negidx{0}},4/1/{0},5/2/{1}}{
				\node[] (\j) at (\i*\leng,0) {$\k$};
			}
			\node [] (-3) at (1*\leng,0) {$\cdots$};
			\node [] (3) at (6*\leng,0) {$\cdots$};
			\foreach \i/\j/\la in {1/2/1,3/4/\dimV,2/3/2}{
				\draw [thin,->] (\i) to node[above]{$-\la$} (\j);
				\draw [thin,->] (-\j) to node[above]{$\la$} (-\i);
			}
			\draw [thin,->] (-1) to node[above]{$0$} (1);
		\end{tikzpicture},\\
	\graphGeneral'_\dimV \; = \; 
	\begin{tikzpicture}[baseline={([yshift=-1ex]current bounding box.center)}]
		\node (L)  at (0,0)   {$\bar r$};
		\node (Ld) at (2.2,0) {$\cdots$};
		\node (L1) at (4.4,0) {$\bar 1$};
		\node (C)  at (6.2,0) {$0$};
		\node (R1) at (8.0,0) {$1$};
		\node (Rd) at (10.2,0) {$\cdots$};
		\node (R)  at (12.4,0) {$r$};

		\draw[thin,->] (L)  -- node[above] {$\half+r$} (Ld);
		\draw[thin,->] (Ld) -- node[above] {$\half+1$}        (L1);
		\draw[thin,->] (L1) -- node[above] {$\half$}        (C);
		\draw[thin,->] (C)  -- node[above] {$-\half$}       (R1);
		\draw[thin,->] (R1) -- node[above] {$-\half-1$}       (Rd);
		\draw[thin,->] (Rd) -- node[above] {$-\half-r$} (R);
	\end{tikzpicture},\\
	\graphGeneral_{\dimV}^+ \; = \; \begin{tikzpicture}[baseline={([yshift=-1ex]current bounding box.center)}]
			\newcommand{\leng}{1.5}
			\foreach \i/\j/\k in {7/4/{\dimV},4/1/{0},5/2/{1}}{
				\node[] (\j) at (\i*\leng,0) {$\k$};
			}
			\node [] (3) at (6*\leng,0) {$\cdots$};
			\foreach \i/\j/\la in {1/2/1,3/4/\dimV,2/3/2}{
				\draw [thin,->] (\i) to node[above]{$-\la$} (\j);
			}
		\end{tikzpicture},\quad
	(\graphGeneral'_\dimV )^+\; = \; \begin{tikzpicture}[baseline={([yshift=-1ex]current bounding box.center)}]
			\newcommand{\leng}{1.85}
			\foreach \i/\la in {0/0,1/1,2/{\cdots},3/\dimV}{
				\node[] (\i) at (\i*\leng,0) {$\la$};
			}

			\foreach \i/\la  [evaluate=\i as \j using int(\i+1)]  in {0/{-\half},1/{-\half-1},2/{-\half-r}}{
				\draw [thin,->] (\i) to node[above]{$\la$} (\j);
			}
		\end{tikzpicture}.
		\end{gathered}
	\end{equation}
We allow $r=\infty$, which means there are infinitely many vertices in both directions. We view the graphs $\linearGraph_r^+$ and $(\linearGraph_r')^+$ as subgraphs of $\linearGraph_r$ and $\linearGraph_r'$ respectively. Denote by $\graphVertices_{\linearGraph}$ respectively $\graphEdges_{\linearGraph}$ the set of vertices and edges of the linear graph $\linearGraph$.
\end{definition}

We denote the source and target of an edge $i$ by $\source{i}$ and $\target{i}$ respectively. 
Define a linear order on $\graphVertices_{\linearGraph}$ by $p\dottedArrowRight p' $ if there exists a sequence of arrows $p\to p_1 \to \cdots \to p_r \to p' \in \linearGraph$.  
For $p\in \graphVertices_{\linearGraph}$, denote by $p_{\rightarrow} <p $ and $ p_{\leftarrow} > p$ be the successor and predecessor of $p$ respectively.

For $\linearGraph\in\{\linearGraph_r,\linearGraph_r'\}$ there is an involution on $\graphVertices_{\linearGraph}$ by interchanging $\foldVertex{p}$ and $p$ for any $p>0-\{0\} $ and fixing $0$. For any $p\in \graphVertices_{\linearGraph}$ we also denote by $\foldVertex{p}$ the involution applied to $p$. Define \[
	\text{ for any } p\in \graphVertices_{\linearGraph} , \quad |p| \coloneqq
	\begin{cases}
		p, & \text{ if } p>0,\\
		\foldVertex{p}, & \text{ otherwise}.
	\end{cases}
\]

\renewcommand{\dimV}{n}

\subsection{Quantum groups}
\begin{definition}\label{convention: quantum group}
	Let $\linearGraph\in\{\linearGraph_r,\linearGraph_r'\}$, where $r\in \N\cup \{\infty\}$.
	The  \emph{quantum group} $\Uq(\linearGraph)\cong \Uq(\gl_{|\graphVertices_{\linearGraph}|})$ associated to $\linearGraph$ is the $\C(q)$-algebra generated by $E_i,F_i$ for $i\in \graphEdges_{\linearGraph}$ and $ D_p ^{\pm1}$ for $p \in \graphVertices_{\linearGraph}$, subject to the following relations 
	\begin{equation}\label{glrel1}
	\begin{gathered}
				  D_p   D_ p\inv = 1=    D_ p\inv   D_ p, \quad    D_ pE_i  D_ p\inv = q^{\delta_{\source{i},p}-\delta_{\target{i},p}} E_i, \\
				  D_ pF_i  D_ p\inv = q^{-\delta_{\source{i},p}+\delta_{\target{i},p}} F_i ,\quad
			E_iF_{i'} -F_{i'}E_i = \delta_{ii'} \frac{K_i-K_i\inv}{q-q\inv },\\
		\end{gathered}
	\end{equation}
	\begin{equation}\label{glrel2}
		\begin{gathered}
		E_{i}E_{i'} =E_{i'}E_i,	\quad 	 F_iF_{i'} =F_{i'} F_i , 
		\quad  \mbox{ if } i, i' \text{ are not adjacent,}
		\\
		\SR (E_i, E_{i'} )\coloneqq E_i ^2E_{i'} - {[2]} E_i E_{i'}E_i +E_{i'}E_i^2 =0 ,\quad 
		\SR(F_i, F_{i'}) = 0 \, ,
		 \quad \text{otherwise.}
		\end{gathered}
		\tag{Serre}
	\end{equation}
	where $\delta$ is the Kronecker symbol and $\displaystyle K_i =  D_ {\source{i}}   D_ {\target{i}}\inv$. 
	Recall that $\Uq(\linearGraph)$ is a Hopf algebra with comultiplication $\Delta$ given by \[
	\Delta(E_i) = E_i\otimes1 +  K_{i} \otimes E_i,\quad \Delta(F_i) = F_i\otimes  K_{i}\inv  + 1\otimes F_i,\quad \Delta(D_j) = D_j\otimes D_j.
\]
\end{definition}

Given $\linearGraph\in\{\linearGraph_r,\linearGraph_r'\}$ and $\mathbbm{k} \in \{\funcField,\groundring,\C\}$, let $\vecrep[\mathbbm{k}]$ be the free $\mathbbm{k}$-module with basis $\{\basis_p\mid p\in \graphVertices_{\linearGraph}\}$. We abbreviate $\vecrep = \vecrep[\funcField]$. 
Then $\vecrep$ becomes the natural representation of $\Uq(\linearGraph)$ via $F_i\basis_p=\delta_{p,\source{i}}\basis_{\target{i}}$, $E_i\basis_p=\delta_{p,\target{i}}\basis_{\source{i}}$ and $D_p \basis_{p'}= q^{\delta_{p,p'}}\basis_{p'}$. Thus, $\linearGraph$ is in fact the crystal graph of $\vecrep$. 

\subsection{Quantum wedge products}\label{wedge-label-section}
Let $\bigwedge^k \vecrep$ be the (type I) finite dimensional irreducible representation with highest weight being the $k$-th fundamental weight $({1,\ldots ,1},0,\ldots ,0)$ where the first $k$ entries are $1$. The quantum wedge product is given by \[
	v_{p'} \otimes v_{p}  \mapsto -q v_p\wedge v_{p'}  ,\qquad v_{p} \otimes v_{p'}  \mapsto  v_p\wedge v_{p'}, \qquad \text{ where } p>p' \in\graphVertices_{\linearGraph}.
\]
By convention, the standard basis of $\bigwedge^k \vecrep$ is $\basis_{p_1} \wedge \cdots \wedge \basis_{p_k} $ for $p_1> \cdots > p_k$. Clearly, this basis is indexed by $\binom{\graphVertices_{\linearGraph}}{k}$, the set of $k$-element subsets of $\graphVertices_{\linearGraph}$. 
Let $\Fund$ be the full subcategory of tensor products of fundamental representations in $\Uq(\linearGraph)$-modules.

\subsection{Box diagrams}\label{box-diag-section}
Let $\linearGraph\in\{\linearGraph_r,\linearGraph_r'\}$.
We will use a modification of the box diagrams from \cite{EhSt-nw-algebras-howe} to index standard wedges in $\bigwedge^k\vecrep$.
For this let $+,-,\pm,\bullet$ be formal symbols. 

\begin{definition}
    An \emph{unfolded box diagram} associated to $\linearGraph\in\{\linearGraph_r,\linearGraph_r'\}$ is a partially defined map $\gamma: \graphVertices_{\linearGraph} \to \{\bullet\}$. 
	Given an unfolded box diagram $\gamma$, define its \emph{support} $\supp(\gamma)$ to be its domain, i.e. the subset of $\graphVertices_{\linearGraph} $ where $\gamma$ is defined. 
\end{definition} 

We will mostly use the following folded version of box diagrams which suits better the study of quantum symmetric pairs. We will show that they are equivalent to the unfolded ones.
\begin{definition}
	A \emph{box diagram} associated to the linear graph $\linearGraph\in\{\linearGraph_r,\linearGraph_r'\}$ is a partially defined map $\beta: \graphVertices_{\linearGraph^+} \to\{+,-,\pm,\bullet\} $, such that $\beta(0)=\bullet$ if $\linearGraph=\linearGraph_r'$ and $\beta $ is defined at $0$.
\end{definition}

\begin{definition}
	Let  $\beta$  be  a box diagram.
	The \emph{straightening} $\beta^\unfoldSupScript$ of $\beta$ is the unfolded box diagram supported on $\{p\in \graphVertices_{\linearGraph}\mid \beta (|p|) \text{ is defined}\}$ given by $
	\beta^\unfoldSupScript (p)= \beta(|p|)$.
	Define the \emph{support} of $\beta$ to be $\supp(\beta) = \supp(\beta^\unfoldSupScript)$.
\end{definition}
Let ${}_{k}\boxdiagramSet^{\unfoldSupScript}(\linearGraph),{}_{k}\boxdiagramSet(\linearGraph)$ be the set of unfolded respectively folded box diagrams $\beta$ with $|\supp(\beta)|=k$. Then we have $\beta\in {}_{k}\boxdiagramSet(\linearGraph)$  if and only if $ \beta^\unfoldSupScript\in {}_{k}\boxdiagramSet^{\unfoldSupScript}(\linearGraph)$.

A box diagram $\beta\in {}_{k}\boxdiagramSet(\linearGraph)$ can be visualised as follows: we draw a column of $r+1$ boxes labeled with the vertices in $\graphVertices_{\linearGraph^+}$ from top to bottom, and then fill the $p$-th box by $\beta(p)$. For example, 
\begin{align}
	\beta = \begin{tabular}{cccccc}
		\hline\hline
		$0$ & $1$ &$2 $&$3$ &$4 $&$5$\\
		\hline
		$+$ &$\pm$ &  & $-$ & $+$\\
		\hline
		\hline
	\end{tabular}\;
	\in {}_{5}\boxdiagramSet(\linearGraph_5)
		\quad \rightsquigarrow\quad 
	\ytabcenter{
	+ \\
	\pm \\
	~ \\
	- \\
	+ \\
	~\\
	} \;\;,
	\quad  
	\text{ with } \supp(\beta)= \{0,\foldVertex{1},1,\foldVertex{3},4\}
	\label{box-diag-example}
\end{align}
We can also visualise an unfolded box diagram $\beta^\unfoldSupScript$ by taking a column of $|\graphVertices_{\linearGraph}|$ boxes and label them by the vertices of the graph from top to bottom, then fill $\beta^\unfoldSupScript(p)$ in the $p$-th box.

\begin{lem}\label{boxdiag-count}
	Let $\linearGraph\in\{\linearGraph_r,\linearGraph_r'\}$. Unfolding and taking support give bijections
	\[
		\begin{tikzcd}[row sep=tiny,font=\normalsize]
			{}_{k}\boxdiagramSet(\linearGraph)   \arrow[r,"\cong"] &
			{}_{k}\boxdiagramSet^{\unfoldSupScript}(\linearGraph)
			 \arrow[r,"\cong"] &
			\binom{\graphVertices_{\linearGraph}}{k} 
			\\
			\beta \arrow[r,mapsto] &
			\beta^\unfoldSupScript \arrow[r,mapsto] &
			\supp(\beta)
		\end{tikzcd}
	\]
\end{lem}
\begin{proof}
	Clearly, an unfolded box diagram and its unfolded support determine each other. 
	Given an unfolded box diagram $\gamma$, define its folding to be the following  box diagram:\[
		\beta(p)= \begin{cases}
			+, & \gamma (p) = \bullet \text{ and }  \gamma (\foldVertex{p}) \text{ is not defined},
			\\
			-, & \gamma (\foldVertex{p}) = \bullet  \text{ and }  \gamma ({p}) \text{ is not defined},
			\\
			\pm, & \gamma (p) = \bullet = \gamma (\foldVertex{p}) = \bullet \text{ and } p\ne 0,
			\\
			\bullet, & p=0 \text{ and } \gamma(p)=\bullet.
		\end{cases}
	\]
	One easily checks that this is the inverse map of straightening.
\end{proof}

\begin{convention}
For $\beta\in {}_{k}\boxdiagramSet(\linearGraph)$ with $\supp(\beta)= \{p_1> \cdots > p_k\}$ we denote $\monWedge_\beta\coloneqq\basis_{p_1} \wedge \cdots \wedge \basis_{p_k} $.
\end{convention}

\begin{definition}\label{boxdiag-opeations}
	Let $\beta\in {}_{k}\boxdiagramSet(\linearGraph)$ be a box diagram. 
	\begin{enumerate}
		\item For $p\in \graphVertices_{\linearGraph}- \supp(\beta)$, let $\beta\cup{p}\in {}_{k+1}\boxdiagramSet(\linearGraph)$ be the box diagram with support $ \supp(\beta) \cup \{p\}$. 
		\item For $p\in \supp(\beta)$, let $\beta\setminus{p}\in {}_{k-1}\boxdiagramSet(\linearGraph)$ be the box diagram  with support $\supp(\beta)\setminus\{p\}$. 
		\item For $p \in \supp(\beta)$ with $\foldVertex{p_1}\notin \supp(\beta)$, let $\flip{{p_1}}\beta \in {}_{k}\boxdiagramSet(\linearGraph)$ be the box diagram  with support $ \supp(\beta) \cup \{\foldVertex{p_1}\}- \{p\}$. 
	\end{enumerate}
\end{definition}

\begin{definition}
    Given a  box diagram $\beta\in {}_{k}\boxdiagramSet(\linearGraph)$, for $p>0$, let $\bnumber_p^\beta$ be the number of `$-$'s minus the number of `$+$'s \emph{below} the $p$-th box  (excluding itself).
\end{definition}
For example, let $\beta$ be the box diagram in \eqref{box-diag-example}, we have $\bnumber_0^\beta=\bnumber_1^\beta=\bnumber_2^\beta=\bnumber_4^\beta=\bnumber_5^\beta=0$, $\bnumber_3^\beta=-1$.

\subsection{Quantum symmetric pairs}
Associated to the  linear graphs $\graphGeneral_r$, respectively $\graphGeneral_r'$, there is a quantum symmetric pair coideal subalgebra, \cite{Letzter-symmetric-pairs-quantized-enveloping} \cite{Letzter-coideal-subalgebras},  corresponding to the following type $\operatorname{AIII}$ Satake diagrams without black nodes. In the framework of \cite{kolbQKacMoodyPairs}, these algebras depend on a choice of certain parameters. In this paper we use the specific choice in \cite{EhSt-nw-algebras-howe}.
\begin{equation}\label{satakes}
	\dynkin[%
	edge length=0.75cm,
	labels={r,r{-}1,1,0,-1,1{-}r,-r},
	label directions={above,above,above,above,above,above,above,above},
	involution/.style={blue!50,stealth-stealth,thick},
	involutions={1{7};26;35}
	]{A}{oo.ooo.oo}
	\;,\quad\quad
		\dynkin[%
		edge length=0.85cm,
		labels={{r+\half},{1+\half},\half,-\half,{-1-\half},{-r-\half}},
		label directions={above,above,above,above,above,above},
		involution/.style={blue!50,stealth-stealth,thick},
		involutions={1{6};25;34}
		]{A}{o.oooo.o}
		\;.
\end{equation}

\begin{definition}
Let $\linearGraph\in\{\linearGraph_r,\linearGraph_r'\}$, where $r\in \N\cup \{\infty\}$. Define the following elements in $\Uq(\linearGraph)$:
\begin{gather}
\begin{aligned}
	\text{ for } \linearGraph_r: \begin{cases}
		B_{\pm i} \coloneqq F_{\pm{i}}+  E_{\mp i}K_{\pm i}\inv  ,  & \text{ if } i\ne0,\\
		B_0\coloneqq F_0+q\inv E_0 K_0\inv,
	\end{cases}
	&\qquad 
	\text{ for } \linearGraph_r': \begin{cases}
		B_{\pm i} \coloneqq F_{\pm{i}}+  E_{\mp i}K_{\pm i}\inv  ,  & \text{ if } i\ne -\half,\\
		B_{-\half} = F_{-\half} + q E_{\half} K_{-\half}\inv,
	\end{cases}
	\\
\dCoideal_p^{\pm1}\coloneqq & (D_{p}D_{\negidx{p}} )^{\pm1}\; (p\geq 0)
.
\end{aligned}
\label{Bi-def}
\end{gather}
The \emph{quantum symmetric pair coideal subalgebra} $\Uq'(\graphGeneral)$ is the subalgebra of $\Uq(\graphGeneral)$ generated by $\{B_i,\dCoideal_p^{\pm 1}\mid i\in \graphEdges_{\linearGraph}, p\geq 0 \}$. We will denote $\kCoideal_i =K_i\inv K_{-i}$ for $i\in  \graphEdges_{\linearGraph^+}$.  
\end{definition}

The defining relations for a presentation in these generators are as follows. For $\linearGraph=\linearGraph_r$,
the $\{B_i\}_{i\ne 0} $ and $\{\dCoideal_p\}$ satisfy the \emph{$ \Uq(\graphGeneral_{r}^+)$-relations}\footnote{That is the usual quantum group relations with $E_i, F_i, D _p$ replaced by the $\{B_{-i}\}_{i>0}$, $\{B_{i}\}_{i>0}$, $\dCoideal_p$ respectively.}  from \cref{glrel1}, \cref{glrel2}, and the following relations:
\begin{align}
	[\bo,\dCoideal_p]=0
	,\qquad
	\SR(\bo, B _{\pm1}) =   B _{\pm1}
	,\qquad
	\SR(B_{\pm1}, \bo )  = 0
	.
	\label{serre relations 1}
\end{align}
And for $\linearGraph=\linearGraph_r'$, the $\{B_i\}_{i\ne \pm\half} $ and $\{\dCoideal_p\}$ satisfy the $\Uq(\gl_{r})$-relations, and \begin{align}
	\begin{gathered}
	\dCoideal_1 B_{-\half} \dCoideal_1 \inv= q B_{-\half} 
	,\; 
	\dCoideal_1 B_{\half} \dCoideal_1 \inv= q\inv B_{\half} 
	,\; 
	\dCoideal_0 B_{-\half} \dCoideal_0 \inv= q^{-2} B_{-\half} 
	,\; 
	\dCoideal_0 B_{\half} \dCoideal_0 \inv=  q^2 B_{\half} ,
	\\
	[B_{\half+1},B_{-\half}]=0 
	,\qquad
	[B_{-\half-1},B_{\half}]=0 
	,\qquad
	D_{p} B_{\pm\half} =B_{\pm\half}  \;(p\notin\{0,\pm1\})
	,
	\end{gathered}\\
	\begin{gathered}
	\SR(B_\half,B_{-\half}) = -[2] B_{\half} (q\dCoideal_{0}\dCoideal_1\inv +q^{-1}\dCoideal_{0}\inv \dCoideal_1)
	,\qquad
	\SR( B_{\pm(\half+1)},B_{\pm\half}) =0
	,\\
	\SR(B_{-\half},B_\half) = -[2] B_{-\half} (q^2\dCoideal_{0}\inv \dCoideal_1  +q^{-2}\dCoideal_{0}\dCoideal_1\inv )
	,\qquad
	\SR(B_{\pm\half}, B_{\pm(\half+1)}) =0.
	\end{gathered}\label{SR-Gamma-prime}
\end{align}
We refer to \cite{EhSt-nw-algebras-howe}*{Prop. 7.17, Prop. 7.18} for the proof.

One easily verifies that $\coideal(\linearGraph)$ is a right coideal subalgebra and hence acts on any representation in $\Fund$ via comultiplication. 
Let $\FundCoideal$ be the full subcategory of ${\Uq'(\linearGraph)}$-modules whose objects are tensor products of fundamental representations of $\Uq(\linearGraph) $ restricted to $\coideal(\linearGraph)$.

\subsection{Schur--Weyl duality with the  type $\operatorname{B}$ Hecke algebra  \texorpdfstring{$\heckeSpecialized{d}$}{}}\label{hecke algebra section} Let $d\in \bbN$. 
\begin{definition}
	The Hecke algebra $\heckeSpecialized{d}$ is generated by $H_0,H_1,\ldots,H_{d-1}$ subject to 
	\begin{align}
		\begin{gathered}
		H_0^2=1, \qquad  H_{0}H_1H_0H_1=H_{1}H_0H_{1}H_0
		,\qquad 
		H_i^2 = 1+(q\inv-q)H_i\; (i>0)
			,\\
		H_i H_{i+1} H_i = H_{i+1} H_i H_{i+1}
			,\qquad 
		H_i H_j = H_j H_i \text{ for } |i-j|>1.
		\end{gathered}
		\label{DefHecke}
	\end{align}
\end{definition}

\begin{definition}\label{Kmatrix}
The \emph{$K$-matrix} is $\kmatrix=(\delta_{p_1,\foldVertex{p_2}})_{p_1,p_2\in \graphVertices_{\linearGraph}} \in \End_{\funcField}(\vecrep)$.  
\end{definition}

There is an action of $\heckeSpecialized{d}$ on $\vecrep^{\otimes d}$ given as follows:  $H_0$ acts on the first tensor factor $\vecrep$ via $\kmatrix$
and $H_i$ acts on the $i$-th and $(i+1)$-th tensor factor via 
$H \colon \vecrep^{\otimes 2}\to \vecrep^{\otimes 2}$ defined by 
\begin{align}
	H (\basis_{p}\otimes \basis_{p'} )&=
	\begin{cases}
		\basis_{p'}\otimes \basis_{p} +(q \inv-q)\cdot \basis_{p}\otimes \basis_{p'} & \text{if } p < p'  \text{ in } \graphVertices_{\graphGeneral}, 
		\\
			\basis_{p'}\otimes \basis_{p} & \text{if } p' < p \text{ in } \graphVertices_{\graphGeneral},
		\\
		q\inv \cdot \basis_{p}\otimes \basis_{p} & \text{if } p=p',
	\end{cases}
	\label{eqn:R-matrix generic form}
\end{align}
The following result was established in \cite[Thm. 5.4, Thm. 6.27]{newapproach} and \cite[Thm. 10.4]{EhSt-nw-algebras-howe}.
\begin{thm}[Schur--Weyl duality]\label{Schur-Weyl-quantum} 
Let $\linearGraph\in\{\linearGraph_r,\linearGraph_r'\}$. 
The action of $\coideal(\linearGraph)$ on $\vecrep^{\otimes d }$ and $\heckeSpecialized{d}$ generate each other's centraliser. Furthermore, if $\dim\vecrep \geq d$, then the action of $\heckeSpecialized{d}$ is faithful.
\end{thm}
\begin{cor}\label{Cor:ss}
	The $\coideal(\linearGraph)$-action on $\vecrep^{\otimes d}$ is semisimple. Furthermore, $\FundCoideal$ is semisimple.
\end{cor}
\begin{proof}
	By \cref{Schur-Weyl-quantum} and the double centraliser theorem, the action of $\coideal(\linearGraph)$ on $\vecrep^{\otimes d}$ is semisimple, since $\heckeSpecialized{d}$ is semisimple, \cite[Rk. 5.4, Thm. 5.5]{DipperJames}. Since objects in $\FundCoideal$  are summands of direct sums of $\vecrep^{\otimes d}$, it follows that any object in $\FundCoideal$ is semisimple. 
\end{proof} 

\begin{remark}\label{specht-module}
	The algebra $\heckeSpecialized{d}$ contains the recursively defined \emph{Jucys--Murphy elements} \[
			J_1 \coloneqq H_0,\qquad J_i = H_{i-1} J_{i-1} H_{i-1}, \text{ for } 2\leq i\leq d,
		\] 
	which pairwise commute, see \cite{ArKo-hecke-algebra}. They play the analogous role of Cartan subalgebras in the study of representation theory of $\heckeSpecialized{d}$.
	Parallel to $W_d$-modules (see \cite[5.5.4]{GeckPfeiffer} and \cite[Prop. 3]{MaSt-complex-reflection-groups}), simple modules over $\heckeSpecialized{d}$ are labeled by bipartitions of $d$ and are also known as \emph{Specht modules} $\specht({\lambda,\mu})$, see \cite[Thm. 3.10]{ArKo-hecke-algebra}. We denote the Specht module associated to the bipartition $(\lambda,\mu)$ by $\specht(\lambda,\mu)$. It has a basis labeled by \emph{standard bitableaux} of shape $(\lambda,\mu)$, which are tableaux of shape $\lambda$ and $\mu$ with total entries $\{1,2,\ldots,d\}$ such that, for each part, the entries in each row and column are strictly increasing. Given a standard bitableau $T$ of shape $(\lambda,\mu)$, the content, i.e. its column number minus its row number, of the box filled with $i$ in $T$ is denoted by $\content_T(i)$. The action of the Jucys--Murphy element $J_j$ on the basis element corresponding to $T$ is then $q^{-2\content_T(i)}$ if $i$ is in the first partition of $T$ and $-q^{-2\content_T(i)}$ if $i$ is in the second partition of $T$. This determines uniquely the simple module up to isomorphism, see \cite{ArKo-hecke-algebra} and \cite[\S5]{AO-cyclotomic-JM} for details.
\end{remark}

\subsection{Classical symmetric pairs}\label{classical-sym-pair}
For any linear graph $\linearGraph$, let $\gl(\linearGraph)$ be the general linear Lie algebra with matrix units $\{\matUnit_{p,p'}\mid p,p'\in \graphVertices_{\linearGraph} \}$ defined by $\matUnit_{p,p'} \basis_l = \delta_{p',l}\basis_p$. Denote the standard Lie algebra generators of $\gl(\linearGraph)$ by $\lgl_p = \matUnit_{p,p}\, (p\in \graphVertices_{\linearGraph})$ and $
	\matrixElmtFont{e}_i = \matUnit_{\source{i},\target{i}}
	,\;
	\matrixElmtFont{f}_i = \matUnit_{\target{i},\source{i}} ,\; (i \in\graphEdges_{\linearGraph})
	.
$ 

For $\linearGraph\in\{\linearGraph_r,\linearGraph_r'\}$, the classical counterpart of $\Uq'(\linearGraph)$ is the universal enveloping algebra of the fixed point subalgebra $\classicalFixPointSubalg\subset \gl(\linearGraph)$
under the involution $\theta$ given by rotating the matrix $180^\circ$. The $\Uq'(\linearGraph)$ generators in \cref{Bi-def} correspond to the following generators of $\classicalFixPointSubalg$: 
\begin{align}\label{gens-g-fixed}
	\matrixElmtFont{b}_{i}= \matrixElmtFont{f}_i + \matrixElmtFont{e}_{-i} =\matUnit_{\target{i},\source{i}} + \matUnit_{\source{-i},\target{-i}} 
	\;\; (i\in \graphEdges_{\linearGraph})
	,\;\; 
	\hat{\matrixElmtFont{d}}_p= \lgl_p +\lgl_{\foldVertex{p}} \;\; (p\in \graphVertices_{\linearGraph})
	.
\end{align}

Let $S_d,W_d$ be the symmetric group and the type $B$ Weyl group respectively. 
\begin{prop}[Classical type $B$ Schur-Weyl duality]\label{SW-duality-q=1}
	Let $\linearGraph\in\{\linearGraph_r,\linearGraph_r'\}$.
	The action of $\C[W_d]$ and the action of $\classicalFixPointSubalg$ on $\vecrep[\C]^{\otimes d}$ 
	generate each other's centralisers. Moreover, the $\C[W_d]$ action is faithful if $\dim \vecrep[\C]\geq d$. 
\end{prop}
\begin{proof}
	See \cite[Thm. 9]{MaSt-complex-reflection-groups} for the double centraliser property. Recall that $W_d\cong S_d\ltimes (\Z/2\Z)^d$. Denote by $m_1,\ldots ,m_d\in W_d$ the standard generators of the subgroup $(\Z/2\Z)^d$.
	If $\dim \vecrep[\C]\geq d$, then the subalgebras $\C[m_1,\ldots,m_d]$ and $\C[S_d]\subset \C[W_d]$ acts faithfully on $\vecrep[\C]^{\otimes d}$. Since $\C[W_d]\cong \C[W(S_d)]\otimes\C[m_1,\ldots, m_d] $ as vector spaces, $\C[W_d]$ also acts faithfully.
\end{proof}

\begin{eg}
In the case of $\linearGraph_1$ we have the following generators of $\classicalFixPointSubalg$: 
\begin{align*}
	\matrixElmtFont{b}_{1}=\left(\begin{smallmatrix}
		0 & 1 \\
		& 0 \\
		&& 0 \\
		&&1& 0
	\end{smallmatrix}\right)
	,\;\;
	\matrixElmtFont{b}_{-1}=\left(\begin{smallmatrix}
		0 &  \\
		1 & 0 \\
		&& 0 &1 \\
		&&& 0
	\end{smallmatrix}\right)
	,\;\;
	\matrixElmtFont{b}_0=\left(\begin{smallmatrix}
		&&&0 \\
		&&1 \\
		&1 \\
		0
	\end{smallmatrix}\right)
	,\;\; 
	\hat{\matrixElmtFont{d}}_0 = \left(\begin{smallmatrix}
		0   \\
		& 1 & \\
		&& 1 \\
		&&&0
	\end{smallmatrix}\right)
	,\;\; 
	\hat{\matrixElmtFont{d}}_1= \left(\begin{smallmatrix}
		1 &  \\
		& 0 \\
		&& 0 \\
		&&&1
	\end{smallmatrix}\right)
	.
\end{align*}
In the case of $\linearGraph_1'$ we have the following generators of $\classicalFixPointSubalg$: \
\begin{align*}
	\matrixElmtFont{b}_{-\half}=\left(\begin{smallmatrix}
		0 & 1 \\
		 & 0 \\
		 & 1& 0 \\
	\end{smallmatrix}\right)
	,\;\;
	\matrixElmtFont{b}_{\half}=\left(\begin{smallmatrix}
		0 & & \\
		1 &0&1  \\
		& & 0
	\end{smallmatrix}\right)
	,\;\;
	\hat{\matrixElmtFont{d}}_0=\left(\begin{smallmatrix}
		0 \\
		&2 \\
		&&0 
	\end{smallmatrix}\right)
	,\;\; 
	\hat{\matrixElmtFont{d}}_1 = \left(\begin{smallmatrix}
		1& & \\
		&0\\
		&&1  \\
	\end{smallmatrix}\right)
	.
\end{align*}
\end{eg}

We next identify the structure of $\classicalFixPointSubalg$. First consider the case $\linearGraph=\linearGraph_r$.
There is a Lie algebra isomorphism $ \varphi: \gl_{r+1}\times \gl_{r+1}\cong\classicalFixPointSubalg$ given by 
\begin{equation}
	(\matrixElmtFont{x},0) \mapsto \frac{1}{2}\left( \begin{smallmatrix}
		\matrixElmtFont{JxJ} &  \matrixElmtFont{Jx}\\
		\matrixElmtFont{xJ} & \matrixElmtFont{x}
	\end{smallmatrix} \right)
	,\quad 
	(0,\matrixElmtFont{x}) \mapsto \frac{1}{2}\left( \begin{smallmatrix}
		\matrixElmtFont{JxJ} &  -\matrixElmtFont{Jx}\\
		-\matrixElmtFont{xJ} &  \matrixElmtFont{x}
	\end{smallmatrix} \right),
	\text{ where } \matrixElmtFont{J} =  \left(\begin{smallmatrix}
	&&1\\
	&\text{\rotatebox{90}{$\ddots$}}&\\
	1
\end{smallmatrix} \right).\label{iso gln times gln}
\end{equation}
Identify $\gl_{r+1}\cong \gl(\linearGraph_r^+)$.
Then the $\classicalFixPointSubalg$ generators correspond to the following elements via $\varphi$: \begin{align*}
	\matrixElmtFont{b}_0 = \varphi(\lgl_\half, -\lgl_\half),\quad 
	\matrixElmtFont{b}_{i} = \varphi(\matrixElmtFont{e}_i,\matrixElmtFont{e}_i),\quad 
	\matrixElmtFont{b}_{-i} = \varphi(\matrixElmtFont{f}_i,\matrixElmtFont{f}_i)\;(i>0) , \quad 
	\hat{\matrixElmtFont{d}}_p =\varphi (\lgl_p,\lgl_p)
	,\;(p\in \graphVertices_{\linearGraph_r^+})
	.
\end{align*}
Thus, $ \{\matrixElmtFont{b}_{\pm i},\hat{\matrixElmtFont{d}}_p\mid i\in \graphEdges_{\linearGraph_r^+}, p	\in \graphVertices_{\linearGraph_r^+}  \}$ generate the diagonally embedded copy of $\gl_{r+1}\subset \classicalFixPointSubalg$.
To describe the inverse of $\varphi$, define inductively 
\[
	\matrixElmtFont{y}_i \coloneqq [\matrixElmtFont{z}_{i-1},\matrixElmtFont{b}_{i}]
	,\qquad
	\matrixElmtFont{z}_i \coloneqq	\begin{cases}
		[\matrixElmtFont{b}_{-i}, \matrixElmtFont{y}_i] & i>0,\\
		\matrixElmtFont{b}_0 &  i =0.
	\end{cases}
	\quad 
	\text{And define } 
	\matrixElmtFont{z}_p' = \sum_{\target{j}\dottedArrowRight p} \matrixElmtFont{z}_j
	\;\text{ for } p\in \graphVertices_{\linearGraph_r^+}
	.
\]
Then the inverse of $\varphi$ is given by \begin{equation}
\begin{aligned}
    (\lgl_p,0) &\longmapsto \frac{1}{2} (\hat{\matrixElmtFont{d}}_p+ \matrixElmtFont{z}_p'), &
    (0,\lgl_p) &\longmapsto \frac{1}{2} ( \hat{\matrixElmtFont{d}}_p- \matrixElmtFont{z}_p') , &
    (\matrixElmtFont{e}_i,0) &\longmapsto \frac{1}{2} ( \matrixElmtFont{b}_{-i}- [\matrixElmtFont{z}_p',\matrixElmtFont{b}_{-i}]) , \\[1ex]
    (0,\matrixElmtFont{e}_i) &\longmapsto \frac{1}{2} ( \matrixElmtFont{b}_{-i}+ [\matrixElmtFont{z}_p',\matrixElmtFont{b}_{-i}]) , &
    (\matrixElmtFont{f}_i,0) &\longmapsto \frac{1}{2} ( \matrixElmtFont{b}_{i}+ [\matrixElmtFont{z}_p',\matrixElmtFont{b}_{i}]) , &
    (0,\matrixElmtFont{f}_i) &\longmapsto \frac{1}{2} ( \matrixElmtFont{b}_{i}- [\matrixElmtFont{z}_p',\matrixElmtFont{b}_{i}]) .
\end{aligned}
\end{equation}

Note that $\phi\inv (\matrixElmtFont{z}_p')=(\matrixElmtFont{d}_p,-\matrixElmtFont{d}_p)$ and $\phi\inv(\hat{\matrixElmtFont{d}}_p)=(\matrixElmtFont{d}_p,\matrixElmtFont{d}_p)$. It follows that  $\phi\inv (\matrixElmtFont{z}_i)=(\matrixElmtFont{h}_i,-\matrixElmtFont{h}_i)$ where $\matrixElmtFont{h}_i=\matrixElmtFont{d}_{\source{i}}-\matrixElmtFont{d}_{\target{i}}$.

\begin{lem}\label{classical CSA}
	The standard Cartan subalgebra $\mathfrak{h}$ of $\classicalFixPointSubalg$ has a basis given by the diagonal Cartan elements $\hat{\matrixElmtFont{d}}_p	$ and the anti-diagonal Cartan elements $\matrixElmtFont{z}_i \;\;(0\leq i\leq r)$. 
\end{lem}

In \cref{intro-prop}, the highest weight of the irreducible summand $\bigwedge^iW^+\boxtimes \bigwedge^{k-i}W^-$ 
with respect to the basis of $\mathfrak{h}$ from \cref{classical CSA} 
is given by 
\begin{align}\label{classical-hw}
	\matrixElmtFont{z}_j\text{-weight} = \begin{cases}
		1 , & \text{ if } i=j,\\
		-1 , & \text{ if } k-i=j,\\
		0 ,& \text{ otherwise},
	\end{cases}
	\qquad
	\matrixElmtFont{d}_p\text{-weight} = \begin{cases}
		2 , & \text{ if } j\leq \min(i,k-i),\\
		1 , & \text{ if }  k-i < j \leq  i,\\
		-1 , & \text{ if }  i < j \leq  k-i,\\
		0 ,& \text{ otherwise}.
	\end{cases}
\end{align}

In the case $\linearGraph=\linearGraph'_r$, there is a Lie algebra isomorphism $\varphi:\classicalFixPointSubalg\cong \gl_r\times \gl_{r+1}$ given by  \[
	(\matrixElmtFont{x},0)
	\mapsto \frac{1}{2}\left(\begin{smallmatrix}
		\matrixElmtFont{JxJ} &  & -\matrixElmtFont{Jx} \\
		&0& \\
		-\matrixElmtFont{xJ} & & \matrixElmtFont{x}
	\end{smallmatrix}\right)
	,\quad
	(0,\left(\begin{smallmatrix}
		c & \matrixElmtFont{v} \\
		\matrixElmtFont{u} & \matrixElmtFont{x}
	\end{smallmatrix}\right))
	\mapsto \frac{1}{2}\left(\begin{smallmatrix}
		\matrixElmtFont{JxJ} &  \matrixElmtFont{Ju} & \matrixElmtFont{Jx} \\
		 \matrixElmtFont{vJ} &c&  \matrixElmtFont{v}\\
		\matrixElmtFont{xJ} & \matrixElmtFont{u} & \matrixElmtFont{x}
	\end{smallmatrix}\right),
	\text{ where } \matrixElmtFont{x} \text{ is an } r\times r \text{ matrix}
	.
\]
Identify $\gl_{r+1}\cong \gl(\linearGraph^+)$ and $\gl_r\cong \gl(\tempGraph)$ where $\tempGraph$ is the linear graph obtained from  $\gl(\linearGraph^+)$ by removing the vertex $0$. Then we have \begin{align}
	\begin{gathered}
	\varphi\inv (\hat{\matrixElmtFont{d}}_p)= (\matrixElmtFont{d}_p,\matrixElmtFont{d}_p)\; (p\in \tempGraph)
	,\quad 
	\varphi\inv (\matrixElmtFont{b}_i) = (\matrixElmtFont{e}_i,\matrixElmtFont{e}_i)
	,\quad 
	\varphi\inv (\matrixElmtFont{b}_{-i}) = (\matrixElmtFont{f}_i,\matrixElmtFont{f}_i)
	\;(i\in\graphEdges_{\tempGraph})
	,\\
	\varphi\inv(\hat{\matrixElmtFont{d}}_0) = (0,\matrixElmtFont{{d}}_0)
	,\quad 
	\varphi\inv (\matrixElmtFont{b}_\half) = (0,\matrixElmtFont{e}_\half)
	,\quad 
	\varphi\inv (\matrixElmtFont{b}_{-\half}) = (0,\matrixElmtFont{f}_\half)
	.
	\end{gathered}\label{csa-odd-classical}
\end{align}
Let $\matrixElmtFont{w}_0= [\matrixElmtFont{b}_\half,\matrixElmtFont{b}_{-\half}]$ and  define inductively  $\matrixElmtFont{w}_n= [\matrixElmtFont{b}_{-n},[\matrixElmtFont{w}_{n-1},\matrixElmtFont{b}_{n}]]$. From \cref{csa-odd-classical}, these elements form a basis of the Cartan subalgebra of the factor $\gl(\linearGraph^+) $, whence we have:
\begin{lemma}
	Let $\linearGraph=\linearGraph_r'$, then
	the elements $\{\matrixElmtFont{w}_p, \hat{\matrixElmtFont{d}}_p\mid p\in \graphVertices(\linearGraph^+)\}$ form a basis of the standard Cartan subalgebra of $\classicalFixPointSubalg$.
\end{lemma}

In \cref{intro-prop}, the highest weight of the irreducible summand $\bigwedge^iW^+\boxtimes \bigwedge^{k-i}W^-$ 
with respect to the basis of $\mathfrak{h}$ from \cref{classical CSA} 
is given by 
\begin{align}\label{classical-hw-odd}
	\matrixElmtFont{w}_j\text{-weight} = \begin{cases}
		2 , & \text{ if } i=j+1,\\
		0 ,& \text{ otherwise},
	\end{cases}
	\qquad
	\matrixElmtFont{d}_p\text{-weight} = \begin{cases}
		2 , & \text{ if } j\leq \min(i,k-i),\\
		1 , & \text{ if }  k-i < j \leq  i,\\
		-1 , & \text{ if }  i < j \leq  k-i,\\
		0 ,& \text{ otherwise}.
	\end{cases}
\end{align}

For any $x,y\in \Uq(\linearGraph) $ and $a\in \funcField$ denote $[x,y]_a= xy-ayx$.
\newcommand{\posrootSet}{\Phi}
\subsection{Root vectors}
Recall that the positive roots of $\Uq(\linearGraph)$ are $ \{\rtPos{i}{j}\mid i\geq j\in \graphEdges_{\linearGraph}\}$, where $\rtPos{i}{i}$ is the simple root $\alpha_i$ and $\rtPos{i}{j} = \alpha_i+\alpha_{i-1} +\cdots +\alpha_{j}$. 
Recall that the convex order $<$ on the set of positive roots is the lexicographical order on the pair $(j,i)$ ($i,j$ are intentionally swapped). 

We arrange the positive roots in the following pattern (for $\linearGraph=\linearGraph_r$):
\begin{align}\label{pos-roots-pic}
	\begin{tikzpicture}[xscale=1.7, yscale=0.9]
		\colorlet{redregion}{purple!40!pink}
		\colorlet{redborder}{purple!50!pink}
		\colorlet{yellowregion}{yellow!60!orange!40}
		\colorlet{yellowborder}{yellow!80!orange}
		\colorlet{blueregion}{cyan!50!blue!40}
		\colorlet{blueborder}{cyan!60!blue}
		\filldraw[fill=redregion, draw=redborder, thick, opacity=0.6, line join=round]
			(-4, 3.4)  -- (3, 3.4) -- (-0.524,-0.124)-- cycle;
		\filldraw[fill=blueregion, draw=blueborder, thick, opacity=0.6, line join=round]
			(3.4, 3) -- (3.4, -4) -- (-0.1, -0.5) -- cycle;
		\filldraw[fill=yellowregion, draw=yellowborder, thick, opacity=0.6, line join=round]
			(3, 3.4) -- (3.4, 3.4) -- (3.4,3) --(-0.1, -0.5) --  (-0.524,-0.124) -- cycle;
		\foreach \u in {3, 2, 1, 0, -1, -2, -3} {
			\foreach \v in {3, 2, 1, 0, -1, -2, -3} {
				\ifnum\u<\v
				\else
					\node at (-\v, \u) {$[\u, \v]$};
				\fi
			}
		}
		\foreach \u in {3, 2, 1, 0, -1, -2, -3} {
			\node at (-\u, \u) {$[\u, \u]$};
		}
	\end{tikzpicture}
\end{align}
Graphically, the order $>$ is given by scanning each column from top to bottom and proceeding from left to right across the columns in \cref{pos-roots-pic}.
For $\linearGraph_r'$ one has a completely analogous picture.

For both $\linearGraph_r$ and $\linearGraph_r'$ we subdivide the positive roots into the following three subsets: \begin{equation}\label{order-root-regions}
		\posrootSet_\succ =\{\rtPos{i}{j}\mid i> |j| \text{ or } i=j \},
		\posrootSet_\circ = \{\rtPos{i}{-i}\mid i> 0\},
		\posrootSet_\prec =\{\rtPos{-j}{-i}\mid \rtPos{i}{j}\in \posrootSet_\succ\}.
	\end{equation}
Note that in the picture \cref{pos-roots-pic}, $\posrootSet_\succ,\posrootSet_\circ, \posrootSet_\prec$ are just the red, yellow, blue regions respectively. This will be useful in the triangular decomposition of $\coideal(\linearGraph)$.

Associated to the convex order $<$ is the usual PBW root vectors of $\Uq(\linearGraph)$ are given by 
\begin{align}\label{root-vec-Uq}
	F_{[i,j]} = [F_j , [F_{j+1},\cdots [F_{i-1},F_i]_{q }\cdots ]_{q }]_{q } \qquad \text{ for } i>j.
\end{align}

\begin{prop}\label{rectangle-commutator-rootvec--Uq}
	Let $\rtPos{i}{j}< \rtPos{k}{l}$ be positive roots. 
	Then for some $a_1,a_2 \in \Z[q,q\inv]$ we have
	\begin{align}
		\label{levandoeskii-soilbelman-better}
		[F_{[i,j]}, F_{[k,l]} ]_{q^{-(\rtPos{i}{j}, \rtPos{k}{l})}} \in \begin{cases}
			a_1 F_{[i,l]} F_{[k,j]} , & \text{ if } k>i\geq l >j,\\
			a_2 F_{[k,j]} , & \text{ if } l=i+1 ,\\
			\{0\},  &  \text{ otherwise.}
		\end{cases}
	\end{align}
\end{prop}
\begin{remark}
	In fact one can show that in \cref{levandoeskii-soilbelman-better} we have $a_1=q\inv-q, a_2=1$. Moreover,
	under the arrangement in \cref{pos-roots-pic}, \cref{rectangle-commutator-rootvec--Uq} gives a way to read off the suitable $q$-commutator of $F_{[i,j]},F_{[k,l]}$: draw a rectangle formed by the rows and columns of $i,j,k,l$, then the result is given by the other vertices of this rectangle above the diagonal.
\end{remark}
\begin{proof}
	First notice that $(\rtPos{i}{j}, \rtPos{k}{l})= \delta_{ik} +\delta_{jl}-\delta_{il}-\delta_{jk}$. 
	We claim that $[F_{[i,j]}, F_{[k,l]} ]_{q^{-(\rtPos{i}{j}, \rtPos{k}{l})}}\ne 0 $ only if $k>i$.
	To see this, by the Levendorskii--Soibelman formula \cite[Prop 5.5.2]{Leven-Soibelman-formula} we have \begin{align}\label{Levendorskii--Soibelman}
		[F_{[i,j]}, F_{[k,l]} ]_{q^{-(\rtPos{i}{j}, \rtPos{k}{l})}} \in \smashoperator{\bigoplus_{\rtPos{k}{l}>\rtPos{i_1}{j_1}\cdots>\rtPos{i_r}{j_r} > \rtPos{i}{j}}} \Z[q^{\pm1}]\cdot  F_{[i_1,j_1]}\cdots F_{[i_r,j_r]}
		.
	\end{align}
	Suppose $\rtPos{k}{l}>\rtPos{i_1}{j_1}\cdots>\rtPos{i_r}{j_r} > \rtPos{i}{j}$ appears on the right-hand side. We then have $\sum \rtPos{i_s}{j_s} = \rtPos{i}{j}+\rtPos{k}{l}$ in view of the weights of the root vectors. 
	By the definition of the convex order we also have $j\leq j_r\leq \cdots \leq j_1\leq l $.
	Since $\rtPos{i}{j} = \alpha_i+\alpha_{i-1} +\cdots +\alpha_{j}$, there exist $t$ such that $j_t=j$. The condition that $\rtPos{i_t}{j_t} > \rtPos{i}{j}$ implies that $i_t>i$. If $k\leq i$, this is not possible, and we must have $[F_{[i,j]}, F_{[k,l]} ]_{q^{-(\rtPos{i}{j}, \rtPos{k}{l})}} =0$.

	Now there are three cases to consider: $
		j\leq i<l-1<l\leq k ,\quad \text{ or } k\geq l = i+1 ,\quad \text{ or } k>i\geq l.
	$
	The first case is trivial, because in view of \cref{root-vec-Uq}, each simple root vector involved in the expression of $F_{[i,j]}$ commutes with each simple root vector involved in the expression of $F_{[k,l]}$, and we have $(\rtPos{i}{j}, \rtPos{k}{l})=0$. 
	For the last two cases, suppose $j_t=j$, we show that $i_t=k$. Suppose $i_t<k$, from the previous paragraph we know that $i_t\geq l$.
	Then any $\rtPos{i_u}{j_u}$ must be a sum of simple roots $\alpha_v$ with $v>l$. Since the order is convex, we must have $\rtPos{i_u}{j_u}\geq \alpha_{l+1}>\rtPos{k}{l}$, a contradiction. This forces $i_t=k$. 
	Consequently, in the second case the right-hand side of \cref{Levendorskii--Soibelman} can only be $\Z[q^{\pm1}]\cdot F_{[k,j]} $. 
	In the last case, the sum of the remaining $\rtPos{i_v}{j_v}$'s must be equal to $\sum_{l\leq v\leq i}\alpha_v$. Then there must be an $i_u=i$. Note that $\rtPos{i_u}{w}> \rtPos{k}{l}$ for all $w>l$. This forces $w=l$. In other words, the only combination of positive roots we can have is $\rtPos{k}{l}>\rtPos{i}{l}>\rtPos{k}{j}>\rtPos{i}{j}$, and the right-hand side of \cref{Levendorskii--Soibelman} can only be $ \Z[q^{\pm1}]\cdot F_{[i,l]} F_{[k,j]} $.
\end{proof}

\subsection{The Letzter map} \label{letztermap-section}
We next recall the Letzter map from \cite{Kolb-Yakimov-defining-rel-QSP}, \cite{kolb2026short-star-products} in our specific setting. 
Let $\kleinCSA \subset \Uq(\linearGraph)$ be the subalgebras generated by $\{ \dCoideal_p\mid  p\in\graphVertices_{\linearGraph^+}\}$.

\begin{definition}
Let $\Aalg \subset \Uq(\linearGraph)$ be the subalgebras generated by $ \kleinCSA\cup \{ F_i\mid  i\in \graphEdges_{\linearGraph}\}$. 
\end{definition}

The algebra $\Aalg$ is called the ``horospherical subalgebra'' in \cite{kolb2026short-star-products}.
Let also $\Upoly \subset  \Uq(\linearGraph) $ be the subalgebra generated by $\Aalg$ and $\langle E_i K_{i}\inv , K_i\inv \mid i\in \graphEdges_{\linearGraph}\rangle $. 
Note that $\coideal(\linearGraph)\subset \Upoly$.
The triangular decomposition of $\Uq(\linearGraph)$ (given by the PBW basis with respect to the convex order) implies $\Upoly\cong \Aalg \oplus \Upoly S $ as vector spaces, where $S= \operatorname{span}\{E_i K_{i}\inv , K_i\inv \}$.

\begin{definition}
The \emph{Letzter map} is the linear projection $\psi: \Upoly\to \Aalg$.
\end{definition}

\begin{definition}\label{grading-filtrations}
There is a graded algebra structure on  $\Aalg$ given by $\deg F_i =1 $ for all $i$ and $\deg \dCoideal_j =0$ for all $j$. 
Similarly, there is a filtration on $\coideal(\linearGraph)$ defined by $\deg B_i=1,\deg \dCoideal_j=0$. 
\end{definition}
The grading on $\Aalg $ is induced by the usual grading on the quantum group $\Uq(\linearGraph)$.

Recall Lusztig's left respectively right skew derivations $\delustig,\partial^R: \Uq^-\to \Uq^-$ defined by 
\begin{align}\label{skew-der-def}
	[E_j,u] = \frac{K_j\delustig_j( u) - \partial_j^R (u) K_j\inv}{q-q\inv} ,\qquad (u\in \Uq^-).
\end{align}
The following twisted Leibniz rules hold for $x,y\in \Uq^-$ with weights $-\mu$ and $-\nu$ respectively, we refer to \cite[3.1.3]{Lu-intro-quantum-groups} for details.
\begin{equation}
	\delL_s(xy)=\delL_s(x)\,y+q^{(\alpha_s,\mu)}x\,\delL_s(y),
	\qquad
	\delR_s(xy)=q^{(\alpha_s,\nu)}\delR_s(x)\,y+x\,\delR_s(y),
	\label{leibniz}
\end{equation}
The following results are the main tools to work with the Letzter map.

\begin{thm}[Properties of the Letzter map]\label{letzter-map}
	The Letzter map has the following properties:
	\begin{enumerate}
		\item $\psi $ is $\kleinCSA$-bilinear.
		\item $\psi(ab) = \psi(a\psi(b)),\; \forall a,b\in \Upoly$.
		\item The restriction $\letzterIso=\psi|_{\coideal(\linearGraph)}:\coideal(\linearGraph) \to \Aalg$ is an isomorphism of vector spaces.
		\item $\psi$ preserves the filtrations on $\coideal$ and $\Aalg$. Furthermore,  the associated graded map of $\letzterIso$ is an algebra isomorphism $\gr\letzterIso:\gr\coideal \cong \gr \Aalg$. 
		\item Let $u\in \Uq^-$ and $a=\letzterIso\inv(u)$, then we have \begin{align}
			\begin{gathered}
			\psi(B_i a) = F_i u + \frac{c_i}{q-q\inv} K_i\inv K_{-i} \delustig_{-i}(u),
			\\
			\psi(a B_i) = u F_i + \frac{c_i}{q-q\inv}  \partial^R_{-i}(u) K_i K_{-i}\inv
			\end{gathered},
			\quad \text{ where }
			c_i = \begin{cases}
				q\inv , & \text{ if } i=\half,\\
				q, & \text{ if } i=0,\\
				1, & \text{ otherwise},
			\end{cases}
			\label{letzter-lusztig-derivative}
		\end{align}
	\end{enumerate}
\end{thm}
\begin{proof}
	See \cite[Lem. 2.10, 2.11]{Kolb-Yakimov-defining-rel-QSP} for (a)--(d). For the first identity of \cref{letzter-lusztig-derivative} we compute  \begin{align*}
		\psi(B_i u) &= \psi(F_i + c_i K_i\inv E_{-i} u)
		= \psi(F_i u) +  c_i K_i\inv\frac{K_{-i}\delustig_{-i}(u) - \partial_{-i}^R(u) K_{-i}\inv}{q-q\inv}
		,
	\end{align*}
	where in the second equality we used the defining property \cref{skew-der-def} of $\delustig_{j}$ and $\partial_j^R$.
	Now since $F_i \psi(u) \in \Aalg$ and  $\partial_j^R (u) K_j\inv \in \Upoly \operatorname{span}\{E_i K_{i}\inv , K_i\inv \}=\ker \psi, $ 
	by definition of the Letzter map we conclude that $
		\psi(B_i u) = F_i \psi(u) + c_i  K_i\inv \frac{K_{-i}\delustig_{-i} (u) }{q-q\inv}
		$, as desired. Finally, the second equality of \cref{letzter-lusztig-derivative} is the special case of \cite[Lem. 4.3, Thm. 4.5]{kolb2026short-star-products}.
\end{proof}

By \cref{letzter-map}, $\letzterIso$ preserves $\kleinCSA$-weights.
Note that in \cref{pos-roots-pic}, if $\rtPos{i}{j}\in \posrootSet_\succ$, then the $\kleinCSA$-weights of $F_{[i,j]} $ and $F_{[i,1-j]}$ are equal, and the $\kleinCSA$-weights of $F_{[-j,-i]} $ and $F_{[j-1,-i]}$ are also equal. Furthermore, these two $\kleinCSA$-weights are each other's negative. Consequently, $\posrootSet_\circ$ give exactly the root vectors of zero $\kleinCSA$-weight. The exact same statement holds in the $\linearGraph_r'$ case. 

\section{Cartan subalgebras}\label{csa-coideal-both}
Let $\linearGraph\in \{\linearGraph_r,\linearGraph_r'\}$. 
In this section we construct uniformly commutative Cartan subalgebras of $\coideal(\linearGraph)$. Various constructions of Cartan subalgebras existed already in \cite{Watanabe-crystal-basis-quantum-symmetric-pair} and \cite{Le-cartan-coideal}.
In the case $\linearGraph=\linearGraph_r'$, we have a different choice of generators of $\coideal(\linearGraph)$ than \cite{Watanabe-crystal-basis-quantum-symmetric-pair} and \cite{Le-cartan-coideal}, whence their formulas do not yield commuting subalgebras in our setup. Consequently, we recover the Cartan subalgebra of \cite{Le-cartan-coideal} in the $\linearGraph_r$ case and obtain a new commutative subalgebra in the $\linearGraph_r'$ case.

\subsection{The star product}\label{star-prod}
We derive an auxiliary commuting result which is cleaner when formulated with the associative star product on $\Aalg$  from \cite{kolb2026short-star-products} defined by \[
	u*v=\letzterIso \bigl(\letzterIso ^{-1}(u)\letzterIso ^{-1}(v)\bigr).
\] 
Then $\letzterIso :\coideal(\linearGraph)\to(\Aalg,*)$ is an algebra isomorphism.
For $u\in\Uq^-$, \cref{letzter-lusztig-derivative} is equivalent to
\begin{align}
F_s*u
&=F_su+\frac{c_s}{q-q\inv}\,\kCoideal_s\,\delustig_{-s}(u) ,
&
u*F_s
&=uF_s+\frac{c_s}{q-q\inv}\,\partial^R_{-s}(u)\kCoideal_s\inv.
\label{left-right-simple-star}
\end{align}
Here $c_s$ is as in \cref{letzter-lusztig-derivative}. The formulas extend from $\Uq^-$ to $\Aalg$ by $\kleinCSA$-bilinearity. For $s\in\graphEdges_{\linearGraph^+}$, denote $\kCoideal_{-s}=\kCoideal_s\inv$. 
The following auxiliary result is \cite[Lem. 3.5, Lem. 3.8]{Ber-Zel-string-bases-UqAr}.
\begin{lem}\label{del-i-Fab}
	Let $a>b\in \Z$, we have
	\begin{align}
	\delustig_s(F_{[a,b]}) & =
	\begin{cases}
		(q\inv-q)\,F_{[a-1,b]}, & \text{if } s=a \\
		0, & \text{otherwise}.
	\end{cases}
	&
	\partial_s^R(F_{[a,b]}) & =
	\begin{cases}
		(q\inv-q)\, F_{[a,b+1]}, & \text{if } s=b \\
		0, & \text{otherwise}.
	\end{cases}
	\label{root-vec-skew-der}
	\end{align}
	For $a=b$, we have $\delustig_a F_{a} = \partial^R_a F_{a}=1$.
\end{lem}

\begin{lem}\label{Fj*Ri=Ri*Fj}
	$F_{\pm j} * F_{[i,-i]} = F_{[i,-i]} * F_{\pm j}$ for all $ i> j \in \graphEdges_{\linearGraph^+}$.
\end{lem}
\begin{proof}
	Denote $R_i= F_{[i,-i]}$.
	By \cref{del-i-Fab} we have $\delustig_{\pm j} R_i = \partial^R_{\pm j} R_i$. Then we have $
		F_{\pm j} * R_i = F_{\pm j} R_i  ,\; R_i * F_{\pm j} = R_i F_{\pm j}. 
	$
	Now note that in the convex order we have $\alpha_{\pm j} > \rtPos{i}{-i}$, hence by \cref{rectangle-commutator-rootvec--Uq}, we have $F_{\pm j} R_i = R_i F_{\pm j}$, which concludes the proof.
\end{proof}

\newcommand{\starCommutator}[1]{\{#1\}}
\begin{prop}\label{Ri-Rj-commute}
	Consider $\linearGraph=\linearGraph_r'$. 
	Let $ i> j \in \graphEdges_{\linearGraph^+}$. Then we have \[
		F_{[i,-i]} *F_{[j,-j]} = F_{[j,-j]} * F_{[i,-i]}.
	\]
\end{prop}
\begin{proof}
	Let $\starCommutator{x,y}_q=x*y-qy*x $. Denote $R_i=F_{[i,-i]}$. By \cref{left-right-simple-star}, we have $R_i=\starCommutator{F_{[i-1,-i]} ,F_ i}_q + c_i R_{i-1}\kCoideal_i\inv $, expanding again $F_{[i-1,-i]}$, we get
	\begin{align}
		R_\half & = \starCommutator{F_{-\half},F_\half}_q
		+[ \kCoideal_\half ; 0 ]
		,
		& R_{i} = 
		\starCommutator{\starCommutator{F_{-i},R_{i-1}}_q,F_i}_q
		+ R_{i-1} \kCoideal_i\inv
		\qquad
		(i>\half)
		\label{Ri-as-starcommutator}
		.
	\end{align}
	Fix $i>\half$, we prove by induction on $j$. For the base case $j=\half$, by \cref{Ri-as-starcommutator} we have 
	\begin{align*}
		\starCommutator{R_{i},R_{\half}}_1&= R_i*Y* X -q R_i*Y*X -qR_i*X*Y   + q^2  R_i*X*Y
		\\
		&~~ -Y* X*R_i + q Y*X*R_i +q X*Y*R_i  - q^2 X*Y*R_i
		+  R_i[\kCoideal_\half;0] +[\kCoideal_\half;0] R_i,
	\end{align*}
	where $X=F_\half$, $Y=F_{-\half}$. Now by \cref{Fj*Ri=Ri*Fj}, $R_i$ commutes with $X,Y$ under the star product. Moreover, by definition we also have $[R_i ,\kCoideal_s]=0$ for all $i,s$. This implies that  $\starCommutator{R_{i},R_{\half}}_1=0$. 
	Now we do the induction step. Let $X=F_{j},Y=F_{-j}, R=R_{j-1}$, we have  \begin{align*}
		\starCommutator{R_{i},R_{j}}_1&= R_i*Y* R*X -q R_i*R*Y*X -qR_i*X*Y*R   + q^2  R_i*X*R*Y
		\nonumber\\
		&~~ -Y* R*X*R_i + q R*Y*X*R_i +q X*Y*R*R_i  - q^2 X*R*Y*R_i
		\nonumber\\
		&~~ +R_i*R*\kCoideal_j\inv - R*\kCoideal_j\inv*R_i
	\end{align*}
	By induction hypothesis, we have $R_i*R= R*R_i$, hence the last two terms cancel each other. By \cref{Fj*Ri=Ri*Fj} and the induction hypothesis, we have $R_i *X = X*R_i $ and $R_i*Y= Y*R_i$, which implies that all terms cancel and $\starCommutator{R_{i},R_{i-1}}_1=0$.
\end{proof}

\subsection{Uniform definition of Cartan subalgebras}

\begin{definition}\label{csa-def}
	Let $\linearGraph\in \{\linearGraph_r,\linearGraph_r'\}$. The Cartan subalgebra $\csaCoideal\subset \coideal(\linearGraph)$ is the subalgebra generated by $\kleinCSA$ together with $\{\letzterIso\inv (F_{[i,-i]})\mid i\geq 0\}$. 
\end{definition}

\begin{thm}\label{cartan-commute-both}
	The subalgebra $\csaCoideal$ is commutative.
	Moreover, there is an algebra isomorphism
	\[
	\csaCoideal \cong \begin{cases}
	\funcField [x_0^{\pm1},\ldots,x_{r}^{\pm1},y_1,\ldots,y_{r}] , & \text{ if } \linearGraph=\linearGraph_r',\\
	\funcField [x_0^{\pm1},\ldots,x_{r}^{\pm1},y_0,\ldots,y_{r}] , & \text{ if } \linearGraph=\linearGraph_r.
	\end{cases}
	\]
\end{thm}
\begin{proof}
	The commutativity of $\csaCoideal$ follows immediately from \cref{Ri-Rj-commute}. The isomorphism follows from the algebraic independence of the $\dCoideal_i$'s and $F_{[i,-i]}$'s. 
\end{proof}

We next find explicit generators of $\csaCoideal$ in both cases.
\subsection{Explicit generators: the $\linearGraph_r$ case}
We show that the Cartan subalgebra of $\Uq'(\linearGraph_r)$ constructed in \cite{Le-cartan-coideal} coincides with  \cref{csa-def}. 

Define inductively the following elements in $\Uq'(\linearGraph_r)$:
\begin{align}
	\label{def-XYZ}
	Y_i \coloneqq [Z_{i-1},B_{-i}]_{q\inv } ,
	\;\;
	X_ i \coloneqq [Z_{i-1},B_{i}]_{q},
	\; (i>0)
	,\qquad
	Z_i \coloneqq	\begin{cases}
		[B_{i}, Y_i]_{q\inv } & \text{ if } i>0,\\
		B_0 &   \text{ if }  i =0,
	\end{cases}
	.
\end{align}
Let $H'\subset \Uq'(\linearGraph_r)$ be the subalgebra generated by $\{Z_i, \hat{D}_j\mid 0\leq i\leq r, j\in\graphVertices_{\linearGraph_r} \}$. We have the following observation about the Cartan generators.
\begin{prop}\label{key-observation-Zi}
	Let $ \Uq^-(\linearGraph)\otimes  \Uq^0(\linearGraph)\otimes \Uq^+(\linearGraph) \cong \Uq(\linearGraph)$ be the triangular decomposition of $\Uq(\linearGraph)$. Then for $
	i\geq 0$, the PBW expansion of $Z_i$ has no summands in $ \Uq^0(\linearGraph)$.
\end{prop}
\begin{proof}
	We do induction on $i$, the case $i=0$ being clear from the definition of $B_0$. Now assume $i>0$.
	By definition, we have 
	\begin{align}
		\label{Zi-expansion}
		Z_{i} = B_{i} Z_{i-1} B_{-i} -q\inv B_{i} B_{-i}Z_{i-1} - q\inv  Z_{i-1} B_{-i}B_{i} +q^{-2}  B_{-i} Z_{i-1} B_{i} 
		. 
	\end{align}
	By induction hypothesis the PBW expansion of $Z_{i-1}$ has no summands in the Cartan subalgebra of $\Uq(\linearGraph)$. For each of its summand $u_-u_0u_+$ with $u_-\in \Uq^-(\linearGraph), \, u_0\in \Uq^0(\linearGraph), \, u_+\in \Uq^+(\linearGraph)$, we just need to check that each of $ B_{i} u_-u_0u_+ B_{-i}$, $B_{i} B_{-i}u_-u_0u_+$, $ u_-u_0u_+ B_{-i}B_{i} $, $ B_{-i} u_-u_0u_+ B_{i} $ have no summands in  $\Uq^0(\linearGraph)$. Observe that $u_+$ involves only $E_{j} $ for $1-i\leq j\leq i-1$ and hence commutes with $F_{\pm i}$. Similarly, $u_-$ commutes with $E_{\pm i}$. Thus, in each case rewriting into PBW monomials does not produce a summand in $ \Uq^0(\linearGraph)$. Hence, $Z_i$ does not have a summand in $ \Uq^0(\linearGraph)$.
\end{proof}

\begin{thm}[Compatibility theorem I]\label{psi(XYZ)}
Writing $F_{[-i,i]}=0$ for all $i>0$, we have:
\begin{align*}
	\letzterIso (Y_i) = -q^{-2i-1} F_{[i-1,-i]},\quad
	\letzterIso (Z_i) = q^{-2i} F_{[i,-i]} - q^{-2i} F_{[i-1,1-i]} \kCoideal_{i}\inv ,\quad
	\letzterIso (X_ i) = q^{2-2i} F_{[i-1,-i]} 
	.
\end{align*} 
\end{thm}
\begin{proof}
	We first prove the $\letzterIso(Z_{i})$ formula via induction on $i$. The base case $i=0$ follows immediately from the definition of $Z_0=B_0=F_0+q\inv E_0K_0\inv$. Now assume $i>0$, we want to compute $\letzterIso(Z_{i})$ using the expansion \cref{Zi-expansion}. For this we need to compute $\letzterIso(Z_{i-1} B_{-i}B_i) $ and $\letzterIso(Z_{i-1} B_{\pm i}) $.
	Write $Z_{i-1}=\letzterIso(Z_{i-1})+u$ with $u\in \ker\psi$, then $u$ has no summand in the Cartan subalgebra of $\Uq(\linearGraph)$ by \cref{key-observation-Zi}.  
	Consequently, neither does $ u \letzterIso(B_{-i}B_{i}) = u (F_{-i}F_{i} +\frac{1}{q-q\inv} \kCoideal_{i} \inv)$. This implies that 
	$ u \letzterIso(B_{-i}B_{i}) \in \ker \psi$. Similarly, we have  $u \letzterIso(B_{-i})\in \ker \psi$.
	Hence, by the second assertion of \cref{letzter-map} and the definition of $\letzterIso$,
	\begin{align*}
		\letzterIso(Z_{i-1} B_{-i}B_i) 
		&= \letzterIso(Z_{i-1} \letzterIso(B_{-i}B_i) ) 
		= \letzterIso((\letzterIso(Z_{i-1})+u) \letzterIso(B_{-i}B_{i}) ) 
		\\
		&= q^{2-2i}x \letzterIso(B_{-i}B_{i}) - q^{2-2i}y  \letzterIso(B_{-i}B_{i})
		,\\
		\letzterIso(Z_{i-1}B_{\pm i}) 
		&= \letzterIso(Z_{i-1} F_{\pm i}) 
		= \letzterIso((\letzterIso(Z_{i-1})+u)F_{\pm i})
		= q^{2-2i}xF_{\pm i} - q^{2-2i}y F_{\pm i}
		,
	\end{align*}
	where $x= F_{[i-1,1-i]}  ,  y=F_{[i-2,2-i]} \kCoideal_{i-1}\inv $, and we used the induction hypothesis that $\letzterIso(Z_{i-1}) = q^{2-2i}x - q^{2-2i}y$. 
	Note that $x\in U^-_{\rtPos{i-1}{1-i}}$, hence by the skew derivative property $\delustig_j (uv) =\delustig_j(u)v +  q^{-(\beta,\alpha_j)} u\delustig_j(v)$ (for $u\in U^-_{\beta}$), we compute by \cref{letzter-lusztig-derivative}
	\begin{align*}
		\letzterIso(B_i (F_{-i} x)) = F_iF_{-i}  x +\frac{1}{q-q\inv} \kCoideal_i \delustig_{-i} (F_{-i} x) 
		= F_iF_{-i}  x +\frac{1 }{q-q\inv} \kCoideal_i  x 
		,\\
		\letzterIso(B_i (x F_{-i} )) = F_i x F_{-i} +\frac{1}{q-q\inv} \kCoideal_i \delustig_{-i} (xF_{-i} ) 
		= F_i  x F_{-i}+\frac{q\inv }{q-q\inv} \kCoideal_i  x 
		,\\
		\letzterIso(B_{-i} (x F_{i} )) =  F_{-i} x F_i +\frac{1}{q-q\inv} \kCoideal_i\inv \delustig_{i} (xF_{i} ) 
		=F_{-i} x F_i +\frac{q\inv }{q-q\inv} \kCoideal_i \inv x 
		.
	\end{align*}
	For the $y$-terms, note that $\letzterIso(B_iyF_j )$ is a scalar multiple of $ \letzterIso(B_i F_{[i-2,2-i]} F_j \kCoideal_{i-1}\inv ) $, which by the second statement of \cref{letzter-map} is equal to $ \letzterIso(B_i F_{[i-2,2-i]} F_j) \kCoideal_{i-1}\inv  $. We can again use \cref{letzter-lusztig-derivative} to evaluate this.
	Now for $i>0$ we have $
		\letzterIso(B_{-i}B_i) = F_{-i}F_{i} +\frac{1}{q-q\inv} \kCoideal_{i} \inv \delustig_i(F_i) 
		=
		F_{-i}F_{i} +\frac{1}{q-q\inv} \kCoideal_{i} \inv
		.$
	We now can compute $\letzterIso$ applied to the terms in \cref{Zi-expansion} separately: \begin{align*}
		\letzterIso(B_{i} Z_{i-1} B_{-i}) 
		& = q^{2-2i} \letzterIso(B_ixF_{-i}) - q^{2-2i} \letzterIso(B_i yF_{-i})
		\\
		& =  q^{2-2i}(F_i x F_{-i}  +\frac{q\inv }{q-q\inv} \kCoideal_i  x  - F_iyF_{-i} -\frac{q\inv}{q-q\inv}\kCoideal_i  y)
		,\\
		\letzterIso(-q\inv B_{i} B_{-i}Z_{i-1} ) 
		& = q^{1-2i} \letzterIso(B_i  F_{-i}(y-x)) 
		=  q^{1-2i}( F_iF_{-i}y + \frac{1}{q-q\inv}\kCoideal_i  y - F_iF_{-i}  x - \frac{1}{q-q\inv} \kCoideal_i  x )
		,\\
		\letzterIso(- q\inv  Z_{i-1} B_{-i}B_{i} ) 
		&= -q\inv \letzterIso(Z_{i-1}) \letzterIso(B_{-i}B_{i} )
		= -q^{1-2i}(x-y) (F_{-i}F_{i} +\frac{1}{q-q\inv} \kCoideal_{i} \inv)
		,\\
		\letzterIso (q^{-2}  B_{-i} Z_{i-1} B_{i} )
		& = q^{-2i} \letzterIso (B_{-i}(x-y)F_i)
		= q^{-2i}(
			F_{-i} x F_i +\frac{q\inv}{q-q\inv} \kCoideal_i \inv  x  - F_{-i}yF_i -\frac{1}{q-q\inv}\kCoideal_i \inv y
		)
		.
	\end{align*}
	Now note that the $\kCoideal_i x$ terms in the first and second equalities cancel each other, and in view of \cref{root-vec-Uq}, the terms involving $x,F_{\pm i}$ assemble to exactly $q^{-2i}F_{[i,-i]}$. Using the fact that $yF_{-i}=q F_{-i}y $, it is easy to see that the $y$ terms cancel each other, hence we proved the desired formula for $\letzterIso(Z_{i-1})$. 
	Finally, using the definition \cref{def-XYZ}, one computes \begin{align*}
		\letzterIso(Y_i) &= \letzterIso(Z_{i-1} B_{-i} -q \inv B_{-i}Z_{i-1})  \\
		&= \letzterIso(Z_{i-1} )F_{-i} -q\inv F_{-i} \letzterIso(Z_{i-1} ) =  -q^{-2i-1} F_{[i-1,-i]},
	\end{align*}
	and the computation of $\letzterIso(X_i)$ is completely analogous.
\end{proof}

\begin{remark}
	One could use the right skew derivation formula from \cref{letzter-lusztig-derivative} to prove \cref{psi(XYZ)} more easily. However, it is more difficult to establish the right skew derivation formula compared with the left one, see \cite{kolb2026short-star-products}. Our proof using \cref{key-observation-Zi} is more elementary.
\end{remark}

The following statement follows immediately from \cref{psi(XYZ)}.
\begin{cor}
	We have $H' = \csaCoideal \subset \coideal(\linearGraph_r)$.
\end{cor}

\subsection{Explicit generators: the $\linearGraph_r'$ case}
Define $W_\half= [B_{-\half},B_{\half}]_q $ and then inductively 
\begin{align}
	W_{\half+i+1} = [B_{-\half-i-1},[W_i,B_{\half+i+1}]_q]_q + (q^2-1) B_{\half+i+1}B_{-\half-i-1} \kCoideal_{\half}\kCoideal_{\half+1}\cdots \kCoideal_{\half+i}.
	\label{Wi-def}
\end{align}

\begin{thm}[Compatibility theorem II]\label{psi-wi}
	Denote $ t_i = \kCoideal_{\half}\kCoideal_1\cdots \kCoideal_{i-1}$, then \[
		\letzterIso (W_i) = \begin{cases}
			F_{[\half,-\half]} -[\kCoideal_\half;0] , & \text{if } \; i=0,\\
			F_{[i,-i]} - F_{[i-1,1-i]} \kCoideal_i\inv - [\kCoideal_i;0]t_i ,& \text{otherwise},\\
		\end{cases}
		\text{ where } [\kCoideal_i;0]  \coloneqq \frac{\kCoideal_i- \kCoideal_{i}\inv}{q-q\inv}.
	\]
\end{thm}
\begin{proof}
	Denote $R_i=F_{[i,-i]}$.
	By \cref{letzter-lusztig-derivative}, we have $
		\letzterIso(W_\half) = [F_{-\half}, F_{\half}]_q - [\kCoideal_\half;0] 
	$, which prove the $i=0$ case. 
	Now assume $i>0$. Denote $P=B_{-i},Q=B_i , w= \letzterIso( W_{i-1} )$,  by \cref{del-i-Fab}, we have
	\begin{align*}
		\letzterIso(W_{i-1} Q) = w F_i , \quad \letzterIso(W_{i-1} P) = w F_{i-1}
		,\quad 
		\letzterIso(Q W_{i-1}) = F_i w, \quad \letzterIso(P W_{i-1}) = F_{i-1} w.
	\end{align*}
	Denote $W_i' = [B_{-i},[W_{i-1},B_i]_q]_q,\eta =(q-q\inv)\inv$, we compute
	\begin{align*}
		\letzterIso (W_i') & = \letzterIso(
			P W_{i-1} Q -q P Q W_{i-1} -q W_{i-1} Q P +q^2 Q W_{i-1} P )\\
		&=
		F_{-i}w F_i +\eta \kCoideal_i \inv \delustig_i (w F_i) 
		-q F_{-i} F_i w -q \eta  \kCoideal_i \inv  \delustig_i (F_i w ) 
		\\ &~~
		-q w F_i F_{-i} -q \eta \partial^R_i(w F_i) \kCoideal_i 
		+ q^2 F_i w F_{-i} +q^2 \eta \partial^R_{i}( F_ i w ) \kCoideal_i   
		\\
		&= [F_{-i} ,[w,F_{i}]_q]_q +\frac{
			\kCoideal_i \inv \delustig_i (w F_i)
			-q \kCoideal_i \inv  \delustig_i (F_i w ) 
			-q  \partial^R_i(w F_i) \kCoideal_i 
			+q^2 \partial^R_{i}( F_ i w ) \kCoideal_i }{q-q\inv}.
	\end{align*}
	Using the induction hypothesis we can insert $w$ and evaluate the right-hand side, where the skew derivations should be computed by first pulling out Cartan factors.
	On the other hand, we have $\letzterIso(B_i B_{-i} t_i) = 
		F_i F_{-i} t_i + \frac{1}{q-q\inv} \kCoideal_i t_i$, and combining these two terms concludes the induction.
\end{proof}

Let $H'\subset \coideal(\linearGraph_r')$ be the subalgebra of $\coideal(\linearGraph_r')$ generated by $\{W_i\}_{0\leq i\leq r}$ and $\{\dCoideal_p\}_{p\in (\linearGraph_r')^+}$. 
\begin{cor}
	We have $H' = \csaCoideal\subset \coideal(\linearGraph_r')$.
\end{cor}

\begin{remark}\label{ai}
	The formula in definition \cref{Wi-def} is inspired by ChatGPT 5.6 Sol, which found the elements $W_{\half+1} $ and $W_{\half+2} $ in the case of $\linearGraph_3'$. It was also able to prove $[W_\half,W_{\half+1}]=0$ via a computation of certain Serre relations, but failed to find anything useful for the proof of \cref{cartan-commute-both}. Based on these elements, we originally defined Cartan subalgebras separately in both cases then realised that they can be unified. 
\end{remark}

\subsection{Star products of anti-diagonal root vectors}
In this subsection we compute the star products of $R_i =F_{[i,-i]}\in \letzterIso(\csaCoideal)$.

\begin{lemma}\label{R_i*Rj=RiRj}
	Let $i\ne j\geq 0$. Then we have
	$
		R_ j * R_i = R_ jR_i \; (j\ne i) 
	$.
\end{lemma}
\begin{proof}
	First consider the case $j\ne i$. By \cref{Ri-Rj-commute}, we may assume that $j<i$ and prove by induction on $j$. The case $j=0$ is clear from \cref{root-vec-skew-der}. For the case $j=\half$, we use the first identity of \cref{Ri-as-starcommutator} to compute 
	\begin{align*}
		{R_{i}*R_{\half}} 
		&= R_i*Y*X -q R_i*X*Y
		 -   R_i [\kCoideal_\half;0]
		 \\
		 & = R_i YX + \frac{1 }{q-q\inv} R_i \kCoideal_\half
		 	-q R_i XY - q\frac{q\inv }{q-q\inv } R_i \kCoideal_\half\inv -   R_i[\kCoideal_\half;0]
		= R_i [X,Y]_q = R_i R_\half
		 ,
	\end{align*}
	where $X=F_{\half} , Y=F_{-\half}$. This shows the base case in the $\linearGraph_r$ respectively $\linearGraph_r'$ case. 
	Now using the second identity of \cref{Ri-as-starcommutator}, we have
	\begin{align*}
		{R_{i}*R_{j}} &= R_i*Y* R*X -q R_i*R*Y*X -qR_i*X*Y*R   + q^2  R_i*X*R*Y
		+  R_i*R*\kCoideal_j\inv
		\\
		& = Y * (R_ i R) * X -q (R_iR)*Y*X -qX*Y*(R_iR)  + q^2  X*(R_iR)*Y
		+  R_iR\kCoideal_j\inv
		\\
		& = Y R_ i R X + \cancel{\frac{q\inv}{q-q\inv} R_iR \kCoideal_j\inv }
		-q R_iRYX - \cancel{\frac{q}{q-q\inv}  R_iR \kCoideal_j\inv }
		\\
		&~~
		-qXYR_iR -\frac{q}{q-q\inv} R_i R \kCoideal_j 
		+ q^2  XR_iRY + q^2 \frac{q\inv}{q-q\inv}R_iR  \kCoideal_j
		+ \cancel{ R_iR\kCoideal_j\inv}
		\\
		&= R_i R_j +\cancel{q^2 \frac{q\inv}{q-q\inv}R_iR  \kCoideal_j} - \cancel{\frac{q}{q-q\inv} R_i R \kCoideal_j }
		,
	\end{align*}
	where $X=F_ j , Y=F_{-j}$, and in the second equality we used \cref{Fj*Ri=Ri*Fj} and the induction hypothesis $R_i*R= R_iR$. This proves that $R_i*R_j = R_iR_j$ for $j\ne i$. 
\end{proof}
	
It is now tempting to believe that the subalgebra of $(\Aalg,*)$ generated by $R_i=F_{[i,-i]} $ and $\kleinCSA$ under the star product coincides with the one generated by the same elements under the usual product. However, this is not the case, as shown by the following:

\begin{thm}\label{Ri*Ri-computation}
	Let $i\geq 0$.
	We have \[
	R_i*R_i = R_i^2 + \sum_{i\geq j>-i}  F_{[i,j]}  k_j F_{[-j,-i ]}   +c_i,
	\text{ where }\quad 
		c_i = q^4c_{i-1} \in \funcField, \; k_j \in \funcField \dCoideal_{\target{i}} \dCoideal_{\source{j}} ^{-1}
		.
	\]
\end{thm}	

The proof of \cref{Ri*Ri-computation} is a lengthy calculation which we put in the appendix.

\begin{remark}
	\cref{Ri*Ri-computation} implies that $\psi\inv (\csaCoideal)$ is NOT equal to the subalgebra of $\Uq^-$ generated by $\kleinCSA$ and the $F_{[i,-i]}$'s under the usual multiplication, because the latter is NOT closed under the star product! We will see that this causes some subtlety in the construction of Verma modules. 
\end{remark}

\section{Verma modules for $\Uq'(\linearGraph)$}
Let $\linearGraph\in \{\linearGraph_r,\linearGraph_r'\}$.
We construct uniformly a PBW-type basis together with a triangular decomposition of $\coideal(\linearGraph)$. The positive part will not be a subalgebra, but we are still able to work around it and define Verma modules of $\coideal(\linearGraph)$ as induced modules. This allows us to mimic the classical highest weight theory more easily and thus obtain many felicitous properties.

\subsection{Triangular decomposition of $\Aalg$}
Recall from \cref{letztermap-section} the subalgebra $\Aalg\subset \Uq(\linearGraph)$ generated by the negative root vectors together with $\kleinCSA$.
We start with a non-standard triangular decomposition of the subalgebra $\Aalg\subset \Uq(\linearGraph)$ and ``pull it back'' via the Letzter map to get the desired triangular decomposition of $\coideal(\linearGraph)$.

For $\bullet \in \{\prec,\circ ,\succ\}$, let $U^-_\bullet \subset \Uq^-(\linearGraph)$ be the subspace spanned by PBW monomials in $\{F_{[i,j]}\mid\rtPos{i}{j} \in \posrootSet_\bullet\}$. 
Let $\PBWset_\prec$ be the set of $uhv$ where $u,h,v$ run through usual PBW monomials in $\UmMinus,\kleinCSA\UmZero,\UmPlus$ respectively under the usual multiplication.

\begin{lem}\label{triang-decompo-Uq-minus}
	$\PBWset_{\prec}$ forms a basis of $\Uq^-(\linearGraph)$, giving rise to vector space isomorphisms
	\[
		\UmMinus\otimes \UmZero\otimes \UmPlus\cong \Uq^-(\linearGraph)
		,\qquad
		\UmMinus\otimes \kleinCSA\UmZero\otimes \UmPlus\cong \Aalg
		.
	\]
\end{lem}	
\begin{proof}
	We first show that $ \UmMinus\UmZero\UmPlus$ spans $\Uq^-$. 
	For root vectors $p_1,\ldots ,p_n$. Define the \emph{number of inversions} of the sequence $p_1,\ldots ,p_n$ by the total number of pairs $i<j$ such that $p_i\succ p_j$. 
	We show that the product $p_1\cdots p_n\in \UmMinus\UmZero\UmPlus$ via induction on the number of inversions. Let $i$ be such that $p_i\succ p_{i+1}$. 
	By \cref{rectangle-commutator-rootvec--Uq} we have either $p_ip_{i+1} \in q^\Z p_{i+1}p_{i} + \Z[q^{\pm1} ] p_i'$ or $p_ip_{i+1} \in q^\Z p_{i+1}p_{i} + \Z[q^{\pm1} ] p_i'p_{i+1}'$. If $p_i'\succ p_{i+1}'$, then by \cref{rectangle-commutator-rootvec--Uq} and the definition of $\prec$, they must $q$-commute. Hence, in either case, we can rewrite $p_ip_{i+1}$ to strictly decrease the number of inversions, and by induction hypothesis, we conclude that $p_1\cdots p_n\in \UmMinus\UmZero\UmPlus$.
	Since the usual PBW monomials under the convex order form a basis of $\Uq^-(\linearGraph)$, we have $\PBWset_\prec$ spans $\Uq^-$. Furthermore, within each PBW filtration piece $U_n$ we have a canonical bijection of $U_n\cap\PBWset_\prec$ with usual PBW monomials in $U_n$. This proves that $\PBWset_\prec$ is linear independent. The vector space isomorphisms then follow immediately from that $\PBWset_\prec$ form a basis.
\end{proof}

\subsection{The coideal PBW theorem}
Define for $\bullet\in \{\prec,\succ\}$ the subspaces $U'_\bullet = \letzterIso\inv(U^-_\bullet)$.

\begin{thm}[PBW theorem]\label{PBW-coideal}
	The following hold:
	\begin{enumerate}
		\item $\letzterIso\inv(\PBWset_{\prec})$ form an $\kleinCSA$-weight basis of $\Uq'(\linearGraph)$.
		\item Multiplication gives a vector space isomorphism $\negCoideal\otimes\csaCoideal \otimes \posCoideal\cong \Uq'(\linearGraph)$.
	\end{enumerate}
\end{thm}
\begin{proof}
	The first statement follows from \cref{triang-decompo-Uq-minus} and the fact that $\letzterIso$ is a linear isomorphism. 
	Next, by \cref{letzter-map}, the associated graded map of $\letzterIso$ is an algebra isomorphism, and hence brings the elements in $\PBWset'$ to $\PBWset_\succ$ in \cref{triang-decompo-Uq-minus}. This shows that $\PBWset'$ is a basis of $\csaCoideal$. Since $\letzterIso$ preserves $\kleinCSA$-weights, we obtain an  $\kleinCSA$-weight basis. This gives the desired triangular decomposition $\negCoideal\otimes\csaCoideal \otimes \posCoideal\cong \Uq'(\linearGraph)$. 
\end{proof}

Let $\posposCoideal\subset \posCoideal $ be the subspace of elements with strictly positive degrees. We have the following key observation that allows us to mimic the classical theory. 
\begin{lem}\label{pbw-hilfslemma}
	We have $\posposCoideal\csaCoideal \subset \negCoideal \csaCoideal\posposCoideal$. 
\end{lem}
\begin{proof}
	Let $u \in \posposCoideal$ and $h\in \csaCoideal$. 
	We may assume that $u$ is a weight vector for $\kleinCSA$. In view of the definition of $\posCoideal$, $u \in \posposCoideal$ implies that $u$ must have strictly positive $\kleinCSA$-weight.
	By \cref{PBW-coideal} we have $uh= \sum y_ih_ix_i $ for some $y_i \in \negCoideal, h_i \in \csaCoideal, x_i \in \posCoideal$. 
	Since $uh $ has the same $\kleinCSA$-weight as $u$, we must have each $x_i\in \posposCoideal$. The $\negnegCoideal$ part is completely analogous.
\end{proof}

\begin{remark}
	Using \cref{Ri*Ri-computation}, one can check that $F_{i+1} * R_i*R_i \notin \letzterIso (\csaCoideal) * \UmPlus $ and $F_{i+1} * R_i*R_i\notin \letzterIso (\csaCoideal)\UmPlus $ for $i= 1 $ in the $\linearGraph_r$ case 
	and $i=\half $ in the $\linearGraph_r'$ case. Consequently, one cannot improve \cref{pbw-hilfslemma}. In fact, $\csaCoideal\posCoideal$ is NOT a subalgebra except in the $\linearGraph_1$ and $\linearGraph_1'$ cases. Thus, \cref{PBW-coideal} is weaker than what one might desire.
\end{remark}

\subsection{A Harish-Chandra homomorphism}
Before proceeding, we propose a version of Harish-Chandra type homomorphism.
Let $\pr: \coideal \to \csaCoideal$ be the projection under the PBW basis from \cref{PBW-coideal}.

\begin{definition}
	The \emph{Harish-Chandra map} $\xi: Z(\coideal) \to \csaCoideal$ is the restriction of $\pr$ to $Z(\coideal)$.
\end{definition}

\begin{prop}
	The Harish-Chandra map is an algebra homomorphism. 
\end{prop}
\begin{proof}
	Let $\wtZero\subset \coideal$ be the subalgebra of zero $\kleinCSA$-weight elements, then clearly we have $Z(\coideal)\subset \wtZero$ and let $I= \wtZero\cap \ker \pr$. We show that $I$ is an ideal in $\wtZero$. Since $\csaCoideal \subset \wtZero$, the restriction of $\pr$ to $\wtZero$ is surjective, and consequently $I+ \csaCoideal = \wtZero$. Then it suffices to show that $\csaCoideal I \csaCoideal \subset I$. 
	By the definition of $\pr$, we have $I \subset \negCoideal \csaCoideal\posposCoideal\cap \wtZero$.
	By \cref{pbw-hilfslemma}, we have $I \csaCoideal \subset \negCoideal \csaCoideal\posposCoideal \csaCoideal \cap  \wtZero  \subset  \negCoideal \csaCoideal\posposCoideal \cap  \wtZero \subset I$. That $\csaCoideal I \subset I$ follows from an analogous argument which we leave to the reader. This shows that $I$ is an ideal and concludes the proof.
\end{proof}

In \cite{Ko-braided-module}, the center of $\coideal=\coideal(\linearGraph)$ was shown to be a polynomial ring with generators indexed by fundamental weights of $\Uq=\Uq(\linearGraph)$. 
For example, the center of $\coideal(\linearGraph_1)$ was explicitly determined in \cite[\S 6]{StWweight-mod-qsp}. Furthermore, in this case we were able to show that the $\xi$ is injective.

\begin{eg}
	The center of $\coideal(\linearGraph_1')$ is generated by $\dCoideal_0\dCoideal_1^2$ and \[
		C_1 = ([B_{-\half},B_{\half}]_q+ [\kCoideal_\half;0]  - \frac{q^2}{q-q\inv}\kCoideal_\half\inv) \dCoideal_1\inv
		,\;
		C_2 =([B_{-\half},B_\half]_{q\inv}   + [ \kCoideal_\half;0 ]+\frac{q^{-2}}{q-q\inv} \kCoideal_\half )\dCoideal_1 
		.
	\]
	In terms of \cite[Thm. 4.10]{Ko-braided-module}, $C_1$ is the central element in $\operatorname{ad}^l(\Uq) D_{-1}^{-2}$, and $C_2$ is the central element in $\operatorname{ad}^l(\Uq) D_1^{2}$. Note that by \cref{psi-wi} we have \[
		\letzterIso\circ \xi (C_1) = F_{[\half,-\half]} \dCoideal_1\inv -\frac{q^2}{q-q\inv}\dCoideal_0\inv 
		,\qquad
		\letzterIso\circ \xi (C_2) = F_{[\half,-\half]} \dCoideal_1 -\frac{q^{-2}}{q-q\inv}\dCoideal_0. 
	\]
	In this case $\letzterIso\circ \xi (C_1) ,\letzterIso\circ \xi (C_2) $ are linearly independent under the usual multiplication. Since the associated graded map of $\letzterIso$ is an algebra isomorphism, $\xi(C_1),\xi(C_2)$ must also be algebraically independent. Consequently, $\xi $ is injective.
\end{eg}

\begin{conjecture}
	Let $\linearGraph \in \{\linearGraph_r,\linearGraph_r'\}$.
	The Harish-Chandra homomorphism of $\coideal(\linearGraph) $ is injective.
\end{conjecture}

\subsection{Verma modules} 
Now we are ready to define Verma modules that come with almost all desired properties.

\begin{definition}
	Let $M$ be a $\Uq'(\linearGraph)$-module. A vector $v\in M$ is called a \emph{highest weight vector} if $v\ne 0$ is a weight vector for $\csaCoideal$ and $\posCoideal v=0$.
\end{definition}

In fact, we have the following useful criterion for highest weight vectors.
\begin{lemma}\label{check-hwvec}
	Let $M$ be a $\Uq'(\linearGraph)$-module. Let $v\in M$ be a  $\csaCoideal$-vector  such that $B_i v =0$ for all $i>0$. Then $v$ is a {highest weight vector}.
\end{lemma}
\begin{proof}
	Denote $B_{[i,j]} \coloneqq \letzterIso \inv (F_{[i,j]})$ for $i\geq j$. Note that for $i>0$ we have \begin{align*}
		F_{i+1} *F_{[i,j]}  = F_{i+1} F_{[i,j]} 
		,\quad  
		F_{[i,j]}*  F_{i+1}  = F_{[i,j]} F_{i+1}  + \frac{1}{q-q\inv} \partial^R_{-i-1} (F_{[i,j]}) \kCoideal_{i+1} \inv 
		,\quad  
		F_{[i+1,j]} = [ F_{[i,j]} , F_{i+1}]_q 
		.
	\end{align*}
	We first show that $B_{i,j} v=0$ for $F_{[i,j]}\in \UmPlus$ by induction on $i-|j|$. Since $F_{[i,j]}\in \UmPlus$, we have $i>|j|$ or $i=j$. In the latter case we have $B_{[i,i]} = \letzterIso\inv(F_i) = B_i$, and we are done by the assumption of the lemma. From now on assume $i>|j|$, then we have \begin{align}
		\label{B-ij-ind}
		F_{[i+1,j]}  = F_{[i,j]}*  F_{i+1}  -q F_{i+1} *F_{[i,j]}
		, \quad \text{ or equivalently}
		, \quad
		B_{[i+1,j]}  = [B_{[i,j]} , B_{i+1} ]_q
		. 
	\end{align}
	If $j=i$, then both $B_{[i,i]}=B_i$ and $B_{i+1}$ kill $v$ implies $B_{[i+1,j]} =0$. If $j=-i$, then by definition, $B_{[i,i]}\in \csaCoideal$ and hence by assumption $B_{[i,i]} v= c v$ for some $c\in \funcField$. Then we have \[
		B_{[i+1,j]}  v = B_{[i,j]}  B_{i+1} v - q B_{i+1}B_{[i,j]} v= -qc B_{[i,j]}v=0 .
	\] 
	This proves the base case of the induction. The induction step is immediately done from \cref{B-ij-ind}. Now it remains to show that for (not necessarily distinct) root vectors $p_1,\ldots ,p_n \in \UmPlus$, we have $\letzterIso\inv(p_1\cdots p_n)v=0$. We show this via induction on $\deg (p_1\cdots p_n)$, the case $p\deg_1\cdots p_n=1$ is just the case $n=1$ and $p_1= F_i$ for some $i>0$ and hence clear.
	Since the associated graded map of $\letzterIso$ is an algebra isomorphism, we have $\letzterIso\inv(p_1\cdots p_n)= \letzterIso\inv(p_1)\cdots \letzterIso\inv(p_n) + u$ for some $u$ with strictly lower degree than $p_1\cdots p_n$. 
	By \cref{PBW-coideal}, we can write $u = \sum y_l h_l x_l $ for some $y_l\in \negCoideal, h_l \in \csaCoideal , x_l \in \posCoideal$. Furthermore, \cref{letzter-map}, $\letzterIso$ preserves the $\kleinCSA$-weight. Since $p_1\cdots p_n$ has positive $\kleinCSA$-weight, so does each $x_l$. We may assume that $x_l = \letzterIso\inv (p_1'\cdots p_m')$ for some root vectors $p_1',\ldots ,p_m' \in \UmPlus$. Since further $\deg x_l \leq \deg u <\deg (p_1\cdots p_n)$, by induction hypothesis, we have $x_l v=0$. Consequently, we must have 
	$\letzterIso\inv(p_1\cdots p_n)v =0$.
\end{proof}

We now construct Verma modules for $\coideal(\linearGraph)$. 
Let $\ersatzUoUplus \subset \coideal(\linearGraph)$ be the subalgebra generated by $\{B_i\}_{i>0} \cup \csaCoideal$. Let $J = \negCoideal \csaCoideal \posposCoideal \cap \ersatzUoUplus$. Note that by \cref{B-ij-ind}, we have $\posCoideal\subset \ersatzUoUplus$. 
Since $\csaCoideal\posCoideal$ is not a subalgebra, the inclusion $\csaCoideal\posCoideal\subset \ersatzUoUplus$ is strict.

\begin{lem}\label{ersatzUoUplus}
	The subspace $J$ is the two-sided ideal generated by $\{B_i\}_{i>0}$. Furthermore, we have an algebra isomorphism $\ersatzUoUplus/J \cong \csaCoideal$.
\end{lem}
\begin{proof}
	Let $J'\subset \ersatzUoUplus$ be the two-sided ideal generated by $ \{B_i\}_{i>0}$, we need to show that $J=J'$. By \cref{pbw-hilfslemma}, we have $J' = \csaCoideal\{B_i\}_{i>0} \csaCoideal \subset \csaCoideal\posposCoideal \cap \ersatzUoUplus = J $. 
	Conversely, by definition we have a surjective algebra homomorphism $\csaCoideal \twoheadrightarrow \ersatzUoUplus /J'$. 
	On the other hand, by \cref{pbw-hilfslemma}, we have also $\ersatzUoUplus = \csaCoideal + J$. In view of the triangular decomposition in \cref{PBW-coideal}, we have $\ersatzUoUplus = \csaCoideal \oplus J$ and $\ersatzUoUplus/J \cong \csaCoideal $ as vector spaces. Combined with $\csaCoideal \twoheadrightarrow \ersatzUoUplus /J'$, we must have $J\subset J'$. 
\end{proof}

Let $\zeta_i \in \funcField$ and $l_{i}\in \Z$ for $0\leq i\leq r$. Denote $\lambda=(n_i,\,\mid \zeta_i)_{i\in \graphVertices_{\linearGraph^+}}$. Then by \cref{ersatzUoUplus}, there exists a unique $1$-dimensional $ \ersatzUoUplus $-module $\funcField_{\hwfont{\lambda}}$ such that \[
	\posCoideal \cdot 1 =0
	,\qquad
	\dCoideal_i \cdot 1 = q^{n_i}
	,\qquad
	\begin{cases}
		Z_i \cdot 1 = \zeta_i & \text{ if } \linearGraph = \linearGraph_r ,\\
		W_{i+\half} \cdot 1 = \zeta_ i  & \text{ if } \linearGraph = \linearGraph_r'.
	\end{cases}
\]

\begin{definition}
	Define the \emph{Verma module with highest weight $\hwfont{\lambda}$} to be $\verma(\hwfont{\lambda}) \coloneqq \coideal(\linearGraph)\otimes_{\ersatzUoUplus} \funcField_{\hwfont{\lambda}}$.
\end{definition}
Note that $1\otimes 1 \subset \verma(\hwfont{\lambda})$ is the generating highest weight vector of $\verma(\hwfont{\lambda})$.

\begin{thm}[Verma modules]\label{verma-properties}
	The following hold:
	\begin{enumerate}
	\item For any highest weight $\hwfont{\lambda}$, $\verma(\hwfont{\lambda})\cong \negCoideal$ as vector spaces.
	\item The PBW basis of any Verma module is a weight basis with respect to $\kleinCSA$.
	\item Let $V$ be any $\coideal(\linearGraph)$-module and $x\in V$ a highest weight vector with weight $\hwfont{\lambda}$. Then there is a unique $\coideal(\linearGraph)$-module homomorphism $\verma(\hwfont{\lambda})\to V$ with $1\otimes 1\mapsto v$.
	\item $\verma(\hwfont{\lambda}) \cong \verma(\hwfont{\mu})$ if and only if $\hwfont{\lambda} =\hwfont{\mu}$. 
	\item $\verma(\hwfont{\lambda})$ has a unique irreducible quotient $\irrquot(\hwfont{\lambda})$.
	\item $\irrquot(\hwfont{\lambda}) \cong \irrquot(\hwfont{\mu})$ if and only if $\hwfont{\lambda} =\hwfont{\mu}$. 
\end{enumerate}
\end{thm}
\begin{proof}
	First we show that  $\coideal(\linearGraph) = \negCoideal\csaCoideal \oplus \coideal(\linearGraph)J$ as vector spaces. Since $\posposCoideal\subset J$, we have $\coideal(\linearGraph)\posposCoideal\subset \coideal(\linearGraph)J$. By \cref{PBW-coideal} again, we have $\coideal(\linearGraph)\posposCoideal+ \negCoideal\csaCoideal = \coideal(\linearGraph)$. This implies that $\coideal(\linearGraph) = \negCoideal\csaCoideal + \coideal(\linearGraph)J$. Now note that $J\subset \coideal(\linearGraph)\posposCoideal$, by \cref{PBW-coideal}, we have $\negCoideal\csaCoideal \cap J=0$ and the direct sum follows. Consequently, we have \[
		\verma (\underline{\lambda }) = (\negCoideal\csaCoideal 
		\oplus \coideal(\linearGraph)J) 
		\otimes _{\ersatzUoUplus} \C(q)_{\underline{\lambda}}
		= \negCoideal  \otimes  \C(q)_{\underline{\lambda}}\oplus 
		\cancel{\coideal(\linearGraph) \otimes J \C(q)_{\underline{\lambda}}}
		\cong \negCoideal 
	\]
	as left $ \negCoideal$-modules. This proves the first assertion.
	The second statement then follows immediately from \cref{PBW-coideal}.
	The third statement follows the universal property of tensor product in the definition of the Verma module. The fourth one then follows immediately from this universal property. 
	To see that $\verma(\hwfont{\lambda})$ has a unique irreducible quotient, we use the fact that the PBW basis of $\verma(\hwfont{\lambda}) $ is a weight basis for $\kleinCSA$. Moreover, amongst this basis, the element $1\otimes 1 $ is the unique weight basis vector with maximal $\kleinCSA$-weight. This shows that the sum of all proper submodules of $\verma(\hwfont{\lambda}) $ does not contain $1\otimes 1$ and hence it has a unique irreducible quotient. 
	Moreover, if $\irrquot(\hwfont{\lambda}) \cong \irrquot(\hwfont{\mu}) $, then they have the same maximal $\kleinCSA$-weight. Since the maximal $\kleinCSA$-weight spaces in $\irrquot(\hwfont{\lambda}), \irrquot(\hwfont{\mu})$ are both $1$ dimensional, it follows that $\hwfont{\lambda} =\hwfont{\mu} $. 
\end{proof}

\begin{remark}\label{MPIrelations}
	For the special case of $\coideal(\linearGraph_1)$, Verma modules were defined in \cite{StWweight-mod-qsp} with a detailed treatment of their irreducible quotients and finite dimensional representation theory. The Verma modules we construct in this section agree with the ones there. 
	Even for $\coideal(\linearGraph_1)$, the Verma modules do not necessarily have a weight basis with respect to $\csaCoideal$. Furthermore, some Verma modules, even without a $\csaCoideal$-weight basis, still have finite dimensional quotients that do have a $\csaCoideal$-weight basis, see \cite[\S 4]{StWweight-mod-qsp} for explicit examples.
\end{remark}

\begin{remark}
	Following the spirit of \cite{Watanabe2}, one may also define a Verma module as follows: given $\hwfont{\lambda}=(n_i,\,\mid \zeta_i)_{0\leq i\leq r}$, one can form the modules
	\begin{align*}
		\coideal(\linearGraph_r) /(\coideal(\linearGraph_r) (Z_i-\zeta_i ) +  \coideal(\linearGraph_r) (\dCoideal_i-q^{n_i}) +\coideal(\linearGraph_r) \posposCoideal)
		,\\
		\coideal(\linearGraph_r') /(\coideal(\linearGraph_r') (W_i-\zeta_i ) +  \coideal(\linearGraph_r') (\dCoideal_i-q^{n_i}) +\coideal(\linearGraph_r') \posposCoideal).
	\end{align*}
	Using \cref{PBW-coideal}, \cref{pbw-hilfslemma} and \cref{verma-properties} one can show that this module is isomorphic to our Verma module $\verma(\hwfont{\lambda}) $ in each case. It is however nontrivial to make the required choices from \cite[Conj. 3.3.3]{Watanabe2} to fit our construction into Watanabe's framework.
\end{remark}

\newcommand{\coeffField}{\mathbb{F}}
\begin{thm}\label{fd-quotients}
	Let $V$ be a finite dimensional $\coideal(\linearGraph)$-module. Then there exists a field extension $\funcField\subset \coeffField$ such that $V\otimes_{\funcField}\coeffField$ has a highest weight vector. In particular, every finite dimensional irreducible $\csaCoideal$-semisimple module is a quotient of a Verma module.
\end{thm}
\begin{proof}
	Recall that the diagonal subalgebra $M\subset \coideal(\linearGraph)$ generated by $\kleinCSA$ and $\langle B_i\mid i\ne 0\rangle$ is isomorphic to $\Uq(\linearGraph^+_r)$. By the theory of quantum groups, the $\kleinCSA$-action on $V$ is diagonalisable with eigenvalues in $q^\Z$, and there exists an $\kleinCSA$-weight vector such that $B_iv=0$ for all $i>0$. Let $V_+\subset V$ be the subspace of such $\kleinCSA$-weight vectors. By \cref{cartan-commute-both}, the $Z_i$'s commute with $\kleinCSA$, there exists a field extension $\funcField\subset \coeffField$ such that there is a $\csaCoideal$-weight vector inside $V_+\otimes \coeffField$. Let $V_+'\subset V_+\otimes \coeffField$ be the subspace of $\csaCoideal$-weight vectors. Since they are killed by all $B_i$ for $i>0$, we conclude with \cref{check-hwvec} that $V_+'$ killed by $\posCoideal$. This proves the existence of a highest weight vector. For the second statement, a $\csaCoideal$-semisimple module is the same as a weight module. Taking $\coeffField = \funcField$, the same argument gives a highest weight vector, and by irreduciblility, it must be the quotient of some Verma module.
\end{proof}

\part{Diagrammatical categories}

\section{Double coset representatives and type $A$ web category}
We next recall certain double coset combinatorics in $S_d$ and $W_d$, as well as the $\gl_n$-web category in \cite{Br-schur} which is a slight modification of the $\mathfrak{sl}_n$-web category defined in \cite{CaKaMo-webs-skew-howe}.
\subsection{Double coset representatives}
Let $S_d$ be the symmetric group of order $d!$.
Let $\underline{a}=(a_1,\ldots, a_r)$ be a sequence of positive integers. We denote by $|\underline{a}|=\sum a_i$ and $\len(\underline{a})=r$. $\underline{a}$ determines the following intervals of integers : \[
	\interval_j(\underline{a}) \coloneqq \{ a_1+\cdots + a_{j-1}+1, a_1+\cdots + a_{j-1}+2,\ldots , a_1 +\cdots + a_{j}\}
	,
	\qquad 
	1\leq j\leq r
	. 
\]
Denote $S_{\underline{a}}= S_{a_i}\times \cdots\times S_{a_r}\subset S_{|\underline{a}|}$ the parabolic subgroup of $S_{|\underline{a}|}$.

We view the symmetric group $S_d$ as bijections on the set of sequences $(l_1,\ldots ,l_d)$ that form a permutation of $\{1,\ldots, d\}$, where we let $S_d$ act from the right, and the simple reflection $s_j$ interchanges the $j$-th and the $(j+1)$-th positions of the sequences. Using \emph{string diagrams}, an \emph{expression} in the simple reflections can be visualised, where the composition given by vertical stacking. 
For instance the following diagram, read from bottom to top, corresponds to the expression $w=s_2s_6s_3s_5s_4s_7$ read from left to right. 
\begin{align}\label{pic-shuffle}
\begin{tikzpicture}[scale=0.5]
	\begin{scope}[shift={(-2,0)}]
		\node at (0, 3) {$s_7$};
		\node at (0, 2) {$s_3$};
		\node at (0, 1) {$s_2$};
	\end{scope}
	\begin{scope}[shift={(-1,-0.5)}]
		\node at (0, 4) {$s_4$};
		\node at (0, 3) {$s_5$};
		\node at (0, 2) {$s_6$};
	\end{scope}
    \draw[ fill=gray!15] (0.6, -0.4) rectangle (2.4, 0.4);
    \draw[ fill=gray!15] (2.6, -0.4) rectangle (6.4, 0.4);
    \draw[ fill=gray!15] (6.6, -0.4) rectangle (9.4, 0.4);
    \foreach \x/\y in {1/1,2/5,3/2/,4/3,5/6,6/8,7/4/,8/7,9/9} {
        \node (b\x) at (\x, 0) {};
        \node[] at (\x, -0.1) {\y};
    }
    \foreach \x in {1,...,9} {
        \node (t\x) at (\x, 4)  {};
        \node[above] at (\x, 4) {\x};
    }
    \draw[ blue] (b1) .. controls (1, 1.5) and (1, 2.5) .. (t1);
    \draw[ blue] (b2) .. controls (2, 1.5) and (5, 2.5) .. (t5);
    \draw[ red!80!black] (b3) .. controls (3, 1.5) and (2, 2.5) .. (t2);
    \draw[ red!80!black] (b4) .. controls (4, 1.5) and (3, 2.5) .. (t3);
    \draw[ red!80!black] (b5) .. controls (5, 1.5) and (6, 2.5) .. (t6);
    \draw[ red!80!black] (b6) .. controls (6, 1.5) and (8, 2.5) .. (t8);
    \draw[ green!60!black] (b7) .. controls (7, 1.5) and (4, 2.5) .. (t4);
    \draw[ green!60!black] (b8) .. controls (8, 1.5) and (7, 2.5) .. (t7);
    \draw[ green!60!black] (b9) .. controls (9, 1.5) and (9, 2.5) .. (t9);
\end{tikzpicture}
\qquad
w = \left(\begin{matrix}
	1 & 2 & 3 &4 & 5 & 6&7&8&9\\
	1&5&2&3&6&8&4&7&9
\end{matrix}\right)
\end{align}
Note that the element in $S_d$ determined by the diagram depends only on connections of the strands, and in particular, different string diagrams can give the same element in $S_d$.

Recall that $\elmntSn\in S_d$ is an $\underline{a}$-shuffle if, setting $ (l_1,\ldots, l_d) = (1,2,\ldots, d) \cdot \elmntSn\inv $, we have $l_{i_1}<l_{i_2}$ for any $1\leq j\leq \len(\underline{a})$ and any $i_1<i_2 \in \interval_j(\underline{a})$.

For example, $w$ from \cref{pic-shuffle} is a $(2,4,3)$-shuffle. Diagrammatically, this just means strands of the same colour do not intersect.

\begin{definition}\label{ab-shuffle-def}
	Let $\underline{a},\underline{b}$ be sequences of positive integers such that $|\underline{b}| =|\underline{a}|= d$. An element $\elmntSn\in S_d$ is called an \emph{$(\underline{a},\underline{b})$-shuffle} if $\elmntSn$ is an $\underline{a}$-shuffle and $\elmntSn\inv$ is a $\underline{b}$-shuffle. Let $\abshuffleSet(\underline{a},\underline{b})$ denote the set of all $(\underline{a},\underline{b})$-shuffles.
\end{definition}

For example, \cref{pic-shuffle} is a $((2,4,3),(4,3,2))$-shuffle. Any $\underline{a}$-shuffle is a $(\underline{a},(1,\ldots ,1))$-shuffle.

\begin{lem}\label{Sn-doublecoset-rep}
	Let $\underline{a},\underline{b}$ be sequences of positive integers such that $|\underline{b}| =|\underline{a}|= d$. 
	The set $\abshuffleSet(\underline{a},\underline{b})\subset S_d$ is precisely the set of double coset representatives of $S_{\underline{a}} \backslash S_d /S_{\underline{b}}$.
\end{lem}

Now consider the type $B$ Weyl group $W_d\supset S_d$ with additional Coxeter generator $s_0$.
Denote $m_1=s_0$, and $m_i= s_{i-1}m_{i-1} s_{i-1}$ for $2\leq i\leq d$. Then $m_1,\ldots ,m_d$ generate a normal subgroup of $W_d$ isomorphic to $(\Z/2\Z)^d$, yielding an isomorphism $W_d\cong S_d\ltimes (\Z/2\Z)^d$. Thus,
$W_d$ can be viewed as the group of signed permutations, where $m_i$ simply flips the sign of the $i$-th position. 
Diagrammatically, we represent $m_i $ as a dot on the $i$-th strand. For example, in the following picture we have the elements $s_0s_2s_6m_{5}s_3s_5s_4s_7$ and $s_0s_2s_6m_{5}s_3s_5m_{3}s_4s_7$, respectively. 
\begin{align}\label{pic-dotted-shuffle}
\begin{tikzpicture}[scale=0.5]
    \draw[ fill=gray!15] (0.6, -0.4) rectangle (2.4, 0.4);
    \draw[ fill=gray!15] (2.6, -0.4) rectangle (6.4, 0.4);
    \draw[ fill=gray!15] (6.6, -0.4) rectangle (9.4, 0.4);
    \draw[ fill=gray!15] (0.6, 4-0.4) rectangle (4.4, 4.4);
    \draw[ fill=gray!15] (4.6, 4-0.4) rectangle (7.4, 4.4);
    \draw[ fill=gray!15] (7.6, 4-0.4) rectangle (9.4, 4.4);
    \foreach \x/\y in {1/1,2/5,3/2/,4/3,5/6,6/8,7/4/,8/7,9/9} {
        \node (b\x) at (\x, 0) {};
    }
    \foreach \x in {1,...,9} {
        \node (t\x) at (\x, 4)  {};
    }
    \draw[ blue,in=-90,out=90] (b1) to node[black,pos=0.3]{$\bullet$} (t1);
    \draw[ blue,in=-90,out=90] (b2) to (t5);
    \draw[ red!80!black,in=-90,out=90] (b3) to (t2);
    \draw[ red!80!black,in=-90,out=90] (b4) to (t3);
    \draw[ red!80!black,in=-90,out=90] (b5) to node[black,pos=0.3]{$\bullet$} (t6);
    \draw[ red!80!black,in=-90,out=90] (b6) to (t8);
    \draw[ green!60!black,in=-90,out=90] (b7) to (t4);
    \draw[ green!60!black,in=-90,out=90] (b8) to (t7);
    \draw[ green!60!black,in=-90,out=90] (b9) to (t9);
\end{tikzpicture}
\qquad
\qquad
\begin{tikzpicture}[scale=0.5]
    \draw[ fill=gray!15] (0.6, -0.4) rectangle (2.4, 0.4);
    \draw[ fill=gray!15] (2.6, -0.4) rectangle (6.4, 0.4);
    \draw[ fill=gray!15] (6.6, -0.4) rectangle (9.4, 0.4);
    \draw[ fill=gray!15] (0.6, 4-0.4) rectangle (4.4, 4.4);
    \draw[ fill=gray!15] (4.6, 4-0.4) rectangle (7.4, 4.4);
    \draw[ fill=gray!15] (7.6, 4-0.4) rectangle (9.4, 4.4);
    \foreach \x/\y in {1/1,2/5,3/2/,4/3,5/6,6/8,7/4/,8/7,9/9} {
        \node (b\x) at (\x, 0) {};
    }
    \foreach \x in {1,...,9} {
        \node (t\x) at (\x, 4)  {};
    }
    \draw[ blue,in=-90,out=90] (b1) to node[black,pos=0.3]{$\bullet$} (t1);
    \draw[ blue,in=-90,out=90] (b2) to (t5);
    \draw[ red!80!black,in=-90,out=90] (b3) to (t2);
    \draw[ red!80!black,in=-90,out=90] (b4) to node[black,pos=0.3]{$\bullet$}(t3);
    \draw[ red!80!black,in=-90,out=90] (b5) to node[black,pos=0.3]{$\bullet$} (t6);
    \draw[ red!80!black,in=-90,out=90] (b6) to (t8);
    \draw[ green!60!black,in=-90,out=90] (b7) to (t4);
    \draw[ green!60!black,in=-90,out=90] (b8) to (t7);
    \draw[ green!60!black,in=-90,out=90] (b9) to (t9);
\end{tikzpicture}
\end{align}
Note that two diagrams that differ by sliding a dot along a string or by removing two dots on the same strand represent the same element in $W_d$.

Next we describe a convenient set of double coset representatives in $W_d$. In contrast to the $S_d$ case, they are generally not shortest in their double cosets. 
\begin{definition}\label{dotted-shuffle-def}
	Given $\underline{a},\underline{b}$ such that $|\underline{b}| =|\underline{a}|= d$, an element $\elmntSn \in W_d$ is a \emph{dotted $(\underline{a},\underline{b})$-shuffle} if the following hold: write $\elmntSn=\elmntSn' m $ with $\elmntSn'\in S_d, m\in (\Z/2\Z)^d$, then $\elmntSn' $ is an $(\underline{a},\underline{b})$-shuffle, and writing $(k_1,\ldots, k_d) = (1,\ldots, d) \cdot \elmntSn\inv $, for any $i<j \in \interval_l(\underline{a})$ such that $|k_i|,|k_j| \in \interval_{l'}(\underline{b})$ for some $l,l'$, we have $|k_i|<|k_j|$ and $k_j<0 \implies k_i<0$.  
Denote by $\dSh(\underline{a},\underline{b}) $ the set of dotted $(\underline{a},\underline{b})$-shuffles.
\end{definition}
In terms of string diagrams, the dotted shuffle condition is equivalent to taking the diagram of an $(\underline{a},\underline{b})$-shuffle and decorate some of the strands with a dot at the top, such that amongst strings ending at the top and the bottom in a common block, all dots are left adjusted.
For example, in \cref{pic-dotted-shuffle}, the element given by the left picture is a dotted $((2,4,3),(4,3,2))$-shuffle, while the right one is not because there is a dot on the second but not the first leftmost red strand.

\begin{prop}\label{WB-double-coset-rep}
	Let $\underline{a},\underline{b}$ be sequences of positive integers such that $|\underline{b}| =|\underline{a}|= d$.
	Then the set $\dSh(\underline{a},\underline{b})$ is a set of double coset representatives of $ S_{\underline{a}} \backslash W_d /S_{\underline{b}}$.
\end{prop}
\begin{proof}
	For any $\elmntSn\in W_d\cong S_d\ltimes (\Z/2\Z)^d $, we can uniquely write $\elmntSn = \elmntSn' m $ for $\elmntSn\in S_d, m\in (\Z/2\Z)^d$. By \cref{Sn-doublecoset-rep} there exists $u\in S_{\underline{a}}, v\in S_{\underline{b}}$ such that $ \elmntSn''= u \elmntSn'v$ is an $(\underline{a},\underline{b})$-shuffle. Then we have $u \elmntSn v=  \elmntSn'' m'$, where $m' = v\inv m v \in (\Z/2\Z)^d $. Moreover, writing \[
		(k_1,\ldots, k_d) = (1,\ldots, d) \cdot (\elmntSn''m')\inv  
		,\quad \text{then}\quad 
		(|k_1|,\ldots, |k_d|) = (1,\ldots, d) \cdot (\elmntSn'')\inv.
	\] 
	Since $\elmntSn''$ is an $(\underline{a},\underline{b})$-shuffle,
	for any $i<j \in \interval_l(\underline{b})$ such that $|k_i|,|k_j| \in \interval_{l'}(\underline{a})$, we have $|k_i| < |k_j| $. 
	Suppose there exists $i<j \in \interval_l(\underline{a})$ with $|k_i|<|k_j| \in \interval_{l'}(\underline{b})$ and $k_j<0,  k_i>0$. 
	Let $x\in S_{\underline{a}}$ be the reflection that interchanges the $i,j$-th entries and $y\in S_{\underline{b}}$ the reflection that interchanges the $|k_i|$ and $|k_j|$-th entries. Then we have $x \elmntSn''m' y = \elmntSn'' y m' y $, where $ym'y\in (\Z/2\Z)^d$, and we have\begin{align*}
		&(1,\ldots ,d) \cdot ym'y
		=(1, \ldots , -|k_i|, \ldots , |k_j|,\ldots ,d)  
		,\\
		& (1,\ldots, d) \cdot (\elmntSn''ym'y)\inv 
		 =
		 (k_1,\ldots,-|k_i|,\ldots, |k_j|,\ldots, k_d)  
		.
	\end{align*}
	Repeating this operation we end up with a dotted $(\underline{a},\underline{b})$-shuffle in the double coset $S_{\underline{a}}w S_{\underline{b}}$. Now let $\elmntSn m $ and $\elmntSn' m '$ be  two dotted $(\underline{a},\underline{b})$-shuffles, where $\elmntSn,\elmntSn'\in S_d$ and $m,m'\in (\Z/2\Z)^d$. Suppose there exists $u\in S_{\underline{a}}, v\in S_{\underline{b}}$ such that $ u \elmntSn m v= \elmntSn' m'$. Then $u \elmntSn m v = u \elmntSn v (v\inv m v) = \elmntSn' m'$ and $(v\inv m v)\in (\Z/2\Z)^d$, hence $u \elmntSn v  = \elmntSn '$. This forces $\elmntSn=\elmntSn'$ and $ v\inv m v = m'$. Write \[
		(1,\ldots, d) \cdot (\elmntSn m)\inv  = (k_1,\ldots , k_d) ,
		\qquad (1,\ldots, d) \cdot (\elmntSn' m')\inv  = (k_1',\ldots , k_d'),
	\]
	then we have $ k_i = \pm k_i'$.
	Assume $m\ne m'$. Then there exists $ \interval_l(\underline{a})$ such that $(k_i)_{i\in \interval_l(\underline{a})} \ne (k_i')_{i\in \interval_l(\underline{a})} $. Similarly, there exists $\interval_{l'}(\underline{b})$ such that $(k_i)_{i\in \interval_l(\underline{a}), |k_i|\in \interval_{l'}(\underline{b})} \ne (k_i')_{i\in \interval_l(\underline{a}), |k_i|\in \interval_{l'}(\underline{b})} $. Since $m' = v\inv m v$, the number of negative entries in $(k_i)_{i\in \interval_l(\underline{a}), |k_i|\in \interval_{l'}(\underline{b})} \ne (k_i')_{i\in \interval_l(\underline{a}), |k_i|\in \interval_{l'}(\underline{b})} $ must be equal.
	Then there must exist $i<j \in \interval_{l}(\underline{a}),  |k_i|<|k_j|\in \interval_{l'}(\underline{b})$ such that either $k_i>0,k_j<0$ or $k_i'>0,k_j'<0$,
	contradicting to the assumption that both $\elmntSn m $ and $\elmntSn' m '$ are dotted shuffles.
\end{proof}

\subsection{The type A web category}
In principle one can define the web category over $\Z[q,q\inv]$, see \cite{Br-schur}. However, for simplicity we work over the local PID $\groundring$. It turns out that all coefficients involved lie in $\Z[q,q\inv]$. We take $\webQpar$ to be one of the two roots of the quadratic relation in the type $A$ Hecke algebra. These two theories are dual to each other, see \cite{StrICM} for details.
\begin{definition}
Let $\webQpar\in \{-q ,q\inv\}$. 
Let $\webCategoryA$ be the $\groundring$-linear strict monoidal category whose objects are sequences of strictly positive integers and morphisms generated by the labeled trivalent diagrams 
	$
		\begin{tikzpicture}[scale=0.25,anchorbase]
			\draw[] (0,0)node[below]{$a$} to (1,1);
			\draw[] (2,0)node[below]{$b$} to (1,1);
			\draw[,] (1,1) to (1,2)node[above]{$a{+}b$};
		\end{tikzpicture}
		,\;
		\begin{tikzpicture}[scale=-0.25,anchorbase]
			\draw[] (0,0)node[above]{$a$} to (1,1);
			\draw[] (2,0)node[above]{$b$} to (1,1);
			\draw[,] (1,1) to (1,2)node[below]{$a{+}b$};
		\end{tikzpicture}, 
		(a,b\in \N)
		$,
	subject to the following relations: 
	\begin{align}
		\text{associativity}:
		\begin{tikzpicture}[anchorbase,scale=0.5]
			\draw[] (1,0)node[below]{$b$} to (0.5,0.5);
			\draw[] (0,0)node[below]{$a$} to (1,1);
			\draw[] (2,0)node[below]{$c$} to (1,1);
			\draw[,] (1,1) to (1,1.5)node[above]{$a{+}b+c$};
		\end{tikzpicture}
		=
		\begin{tikzpicture}[anchorbase,scale=0.5]
		\draw[] (1,0)node[below]{$b$} to (1.5,0.5);
		\draw[] (0,0)node[below]{$a$} to (1,1);
		\draw[] (2,0)node[below]{$c$} to (1,1);
		\draw[,] (1,1) to (1,1.5)node[above]{$a+b+c$};
		\end{tikzpicture},
		&\quad 
		\text{coassociativity}:
		\begin{tikzpicture}[anchorbase,scale=0.5]
		\draw[,] (1,0)node[above]{$b$} to (0.5,-0.5);
		\draw[,] (0,0)node[above]{$a$} to (1,-1);
		\draw[,] (2,0)node[above]{$c$} to (1,-1);
		\draw[] (1,-1) to (1,-1.5)node[below]{$a+b+c$};
		\end{tikzpicture}
		=
		\begin{tikzpicture}[anchorbase,scale=0.5]
		\draw[,] (1,0)node[above]{$b$} to (1.5,-0.5);
		\draw[,] (0,0)node[above]{$a$} to (1,-1);
		\draw[,] (2,0)node[above]{$c$} to (1,-1);
		\draw[] (1,-1) to (1,-1.5)node[below]{$a+b+c$};
		\end{tikzpicture},
		\label{assoc-coassoc}
		\\
		\text{(thin) square switch}: &
		\begin{tikzpicture}[scale=0.2,anchorbase]
			\draw (0,0) node[below] {$a$} -- (0,4);
			\draw (3,0) node[below] {$b$} -- (3,4);
			\coordinate (L1) at (0,3);
			\coordinate (R1) at (3,2.5);
			\coordinate (L2) at (0,1);
			\coordinate (R2) at (3,1.5);
			\draw (L1) -- (R1) node[midway, above] {$1$};
			\draw (L2) -- (R2) node[midway, below] {$1$};
		\end{tikzpicture}
		-
		\begin{tikzpicture}[scale=0.2,anchorbase]
			\draw (0,0) node[below] {$a$} -- (0,4);
			\draw (3,0) node[below] {$b$} -- (3,4);
			\coordinate (L1) at (0,2.5);
			\coordinate (R1) at (3,3);
			\coordinate (L2) at (0,1.5);
			\coordinate (R2) at (3,1);
			\draw (L1) -- (R1) node[midway, above] {$1$};
			\draw (L2) -- (R2) node[midway, below] {$1$};
		\end{tikzpicture}
		= 
		\Qint{a-b}
		\begin{tikzpicture}[scale=0.2,anchorbase]
			\draw (0,0) node[below] {$a$} -- (0,4);
			\draw (3,0) node[below] {$b$} -- (3,4);
		\end{tikzpicture}
		\label{square-switch}
		.
	\end{align}
	We will refer to the natural number $a$ that labels the strand as the \emph{thickness} of the strand. Strands of thickness $1$ are called \emph{thin strands}.
\end{definition}
The following \emph{explosion} relation follows directly from the {thin square switch} relation.
\begin{align}\label{explosion}
	\begin{tikzpicture}[scale=0.5*0.5]
		\def\wingspan{0.5}
		\def\hoehe{2}
		\def\dist{2.4}
		\path[out=130,in=-90,] (0,\wingspan)  to  +(-1.5*\wingspan,2*\wingspan)  ;		
		\path[out=50,in=-90]	(0,\wingspan)	to  +(1.5*\wingspan,2*\wingspan);
		\path[very thick,out=-90,in=90, draw]	(0,\wingspan)  to (0,0) node[below] {$a$} ;
		\path[in=-90,out=90] (1.5*\wingspan,3*\wingspan) 	to  node[ right]{$a-b$} +(0,2*\wingspan)
		;
		\path[, in=-90,out=90] (-1.5*\wingspan,3*\wingspan) to  node[left]{$b$}+(0, 2*\wingspan);
		\path[, in=-90,out=90] 	(-1.5*\wingspan,5*\wingspan)	to +(1.5*\wingspan, 2*\wingspan)
		to +(1.5*\wingspan, 2*\wingspan)
		(1.5*\wingspan,5 *\wingspan) to +(-1.5*\wingspan, 2*\wingspan)	
		to +(-1.5*\wingspan, 3*\wingspan)	
		;
	\end{tikzpicture}
	=
	\Qbinom{a}{b}
	\begin{tikzpicture}[anchorbase]
		\def\height{1}
		\draw[] (0,0) -- +(0,-1*\height) node [below] {$a$};
	\end{tikzpicture} 
	.
\end{align}


\begin{definition}\label{def:overcrossing}
	For any $a,b\in \N$ define the over and under crossings of two strands of thickness $a$ and $b$ as follows:
	\begin{align}
		\begin{tikzpicture}[anchorbase,xscale=0.75,yscale=0.6]
			\def\hoehe{1.3}
			\def\dist{0.8}
			\coordinate (bleft) at (0,0);
			\coordinate (bright) at (\dist, 0);
			\coordinate (tleft) at (0,\hoehe);
			\coordinate (tright) at (\dist,\hoehe);
			\node[below] at (bleft) {$a$};
			\node[below] at (bright) {$b$};
			\path[,in=-90,out=90] (bright) to (tleft);
			\path[overline,in=-90,out=90] (bleft) to (tright); 
		\end{tikzpicture} 
		= \sum_{0\leq r\leq \min(a,b)} (-\webQpar )^{-r}
		\begin{tikzpicture}[anchorbase,xscale=0.75,yscale=0.6]
			\def\hoehe{1.3}
			\def\dist{0.8}
			\def\r{0.2}
			\path[draw] (0,0) node[below] {$b$} to node[left] {$r$}+(0,\hoehe )node[above] {$a$} ;
			\path[draw] (\dist,0) node[below] {$a$}  to +(0,\hoehe )node[above] {$b$};
			\path[draw] (0,\r*\hoehe) to +(\dist, \r*\hoehe);
			\path[draw] (\dist,\hoehe-2*\r*\hoehe) to +(-\dist, \r*\hoehe);
		\end{tikzpicture}
		\quad
		\begin{tikzpicture}[anchorbase,yscale=0.6,xscale=-0.75]
			\def\hoehe{1.3}
			\def\dist{0.8}
			\coordinate (bleft) at (0,0);
			\coordinate (bright) at (\dist, 0);
			\coordinate (tleft) at (0,\hoehe);
			\coordinate (tright) at (\dist,\hoehe);
			\node[below] at (bleft) {$a$};
			\node[below] at (bright) {$b$};
			\path[,in=-90,out=90] (bright) to (tleft);
			\path[overline,in=-90,out=90] (bleft) to (tright); 
		\end{tikzpicture} 
		= \sum_{0\leq r\leq \min(a,b)} (-\webQpar )^{r}
		\begin{tikzpicture}[anchorbase,xscale=0.75,yscale=0.6]
			\def\hoehe{1.3}
			\def\dist{0.8}
			\def\r{0.2}
			\path[draw] (0,0) node[below] {$a$} to node[left] {$r$}+(0,\hoehe )node[above] {$b$} ;
			\path[draw] (\dist,0) node[below] {$b$}  to +(0,\hoehe )node[above] {$a$};
			\path[draw] (0,\r*\hoehe) to +(\dist, \r*\hoehe);
			\path[draw] (\dist,\hoehe-2*\r*\hoehe) to +(-\dist, \r*\hoehe);
		\end{tikzpicture}
		.
	\end{align}
\end{definition}
We refer to \cite[Cor. 6.2.3]{CaKaMo-webs-skew-howe} and \cite[Thm 6.1]{Br-schur} for the following result.
\begin{thm}
	The over and under crossings are inverse to each other. Furthermore, the over-crossings yield a braiding on $\webCategoryA$.
\end{thm}
In particular, the over and under crossing of two thin strands are
\begin{equation}
	\label{eqn: Hecke generators}
	\begin{tikzpicture}[anchorbase,scale=0.7]
	\def\height{1.3}
	\def\dist{0.8}
	\coordinate (bleft) at (0,0);
	\coordinate (bright) at (\dist, 0);
	\coordinate (tleft) at (0,\height);
	\coordinate (tright) at (\dist,\height);
	\node[below] at (bleft) {1};
	\node[below] at (bright) {1};
	\overcrossing{bleft}{bright}{tright}{tleft}{}
	\end{tikzpicture} 
	= 
	\begin{tikzpicture}[anchorbase,scale=0.7]
	\def\height{1.3}
	\def\dist{0.8}
	\coordinate (bleft) at (0,0);
	\coordinate (bright) at (\dist, 0);
	\coordinate (tleft) at (0,\height);
	\coordinate (tright) at (\dist,\height);
	\coordinate (center) at (0.5*\dist,0.5*\height);
	\node[below] at (bleft) {1};
	\node[below] at (bright) {1};
	\webCenter{center}{bleft}{bright}
	\path (WebCenter) edge (bleft)
	edge (bright)
	edge (center);
	\webCenter{center}{tleft}{tright}
	\path (WebCenter) edge (center)
	[] edge (tright)
	edge (tleft);
	\end{tikzpicture} 
	-\webQpar\inv \qquad\qquad
	\begin{tikzpicture}[anchorbase,scale=0.7	]
	\def\height{1.3}
	\def\dist{0.6}
	\coordinate (bleft) at (0,0);
	\coordinate (bright) at (\dist, 0);
	\coordinate (tleft) at (0,\height);
	\coordinate (tright) at (\dist,\height);
	\node[below] at (bleft) {1};
	\node[below] at (bright) {1};
	\path[] (bleft) edge (tleft) (bright) edge (tright);
	\end{tikzpicture}
	,\quad
	\begin{tikzpicture}[anchorbase,scale=0.7]
	\def\height{1.3}
	\def\dist{0.8}
	\coordinate (bleft) at (0,0);
	\coordinate (bright) at (\dist, 0);
	\coordinate (tleft) at (0,\height);
	\coordinate (tright) at (\dist,\height);
	\node[below] at (bleft) {1};
	\node[below] at (bright) {1};
	\undercrossing{bleft}{bright}{tright}{tleft}{}
	\end{tikzpicture} 
	=  
	\begin{tikzpicture}[anchorbase,scale=0.7]
	\def\height{1.3}
	\def\dist{0.8}
	\coordinate (bleft) at (0,0);
	\coordinate (bright) at (\dist, 0);
	\coordinate (tleft) at (0,\height);
	\coordinate (tright) at (\dist,\height);
	\coordinate (center) at (0.5*\dist,0.5*\height);
	\node[below] at (bleft) {1};
	\node[below] at (bright) {1};
	\webCenter{center}{bleft}{bright}
	\path (WebCenter) edge (bleft)
	edge (bright)
	edge (center);
	\webCenter{center}{tleft}{tright}
	\path (WebCenter) edge (center)
	[]edge (tright)
	edge (tleft);
	\end{tikzpicture} 
	-\webQpar \quad
	\begin{tikzpicture}[anchorbase,scale=0.7]
	\def\height{1.3}
	\def\dist{0.6}
	\coordinate (bleft) at (0,0);
	\coordinate (bright) at (\dist, 0);
	\coordinate (tleft) at (0,\height);
	\coordinate (tright) at (\dist,\height);
	\node[below] at (bleft) {1};
	\node[below] at (bright) {1};
	\path[] (bleft) edge (tleft) (bright) edge (tright);
	\end{tikzpicture}.
\end{equation}

We end this subsection by listing some useful relations.

\begin{lem}\label{lem:crossing absorption}
	The following \emph{local} relations hold in $\webCategoryA$:
	\begin{align}
		\begin{tikzpicture}[anchorbase,scale=0.7]
			\def\height{1.3}
			\def\dist{0.8}
			\coordinate (bleft) at (0,0);
			\coordinate (bright) at (\dist, 0);
			\coordinate (tleft) at (0,\height);
			\coordinate (tright) at (\dist,\height);
			\coordinate (center) at (0.5*\dist,0.5*\height);
			\node[above] at (tleft) {$a$};
			\node[below] at (bleft) {$a$};
			\node[below] at (bright) {$1$};
			\node[above] at (tright) {$1$};
			\webCenter{center}{bleft}{bright}
			\path (WebCenter) edge (bleft)
			edge (bright)
			edge (center);
			\webCenter{center}{tleft}{tright}
			\path (WebCenter) edge (center)
			[] edge (tright)
			edge (tleft);
		\end{tikzpicture}  
		\;=\; 
		\begin{tikzpicture}[scale=0.6]
			\def\height{1.3}
			\def\dist{0.8}
			\def\shape{0.2}
			\path[] (\dist,0) node[below]{$1$} to (0,\height-\shape*\height);
			\path[overline](\dist,\height) node[above]{$1$} to (0, \shape*\height);
			\path[] (0,0) node[below]{$a$}to  (0,\height) node [above]{$a$};
		\end{tikzpicture}
		+\webQpar^{-a} 
		\begin{tikzpicture}[scale=0.6]
			\def\height{1.3}
			\def\dist{0.8}
			\def\shape{0.3}
			\path[draw] (0,0) node[below]{$a$}to (0,\height)
				(\dist, 0) node[below] {$1$} to (\dist,\height);
		\end{tikzpicture}
		\hspace{2cm}
		\begin{tikzpicture}[anchorbase,scale=0.7]
			\def\height{1.3}
			\def\dist{0.8}
			\coordinate (bleft) at (0,0);
			\coordinate (bright) at (\dist, 0);
			\coordinate (tleft) at (0,\height);
			\coordinate (tright) at (\dist,\height);
			\coordinate (center) at (0.5*\dist,0.5*\height);
			\node[above] at (tleft) {$1$};
			\node[below] at (bleft) {$1$};
			\node[below] at (bright) {$a$};
			\node[above] at (tright) {$a$};
			\webCenter{center}{bleft}{bright}
			\path (WebCenter) edge (bleft)
			edge (bright)
			edge (center);
			\webCenter{center}{tleft}{tright}
			\path (WebCenter) edge (center)
			[] edge (tright)
			edge (tleft);
		\end{tikzpicture}  
		\;=\; 
		\begin{tikzpicture}[xscale=-0.6,yscale=0.6]
			\def\height{1.3}
			\def\dist{0.8}
			\def\shape{0.2}
			\path[](\dist,\height) node[above]{$1$} to (0, \shape*\height);\path[overline] (\dist,0) node[below]{$1$} to (0,\height-\shape*\height);
			\path[] (0,0) node[below]{$a$}to  (0,\height) node [above]{$a$};
		\end{tikzpicture}
		+\webQpar^{-a} 
		\begin{tikzpicture}[scale=0.6]
			\def\height{1.3}
			\def\dist{0.8}
			\def\shape{0.3}
			\path[draw] (0,0) node[below]{$1$}to (0,\height)
			(\dist, 0) node[below] {$a$} to (\dist,\height);
		\end{tikzpicture}
		\label{useful}
		\\
		\begin{gathered}
			\begin{tikzpicture}[yscale=0.25,xscale=0.35,anchorbase,font=\small]
			\node (ab) at (0,-2) {$a$}; 
			\node (a) at (-1,1) {};
			\node (b) at (1,1) {};
			\webCenter{ab}{a}{b}
			\path[very thick,out=-90,in=90,]	(WebCenter)	   edge  (ab)
			[out=130,in=-90,]		edge (a)
			[out=50,in=-90]		edge  (b); 
			\path[ out=90,in=-90](b.south) to ($(a)+(0,2)$) node[above]{$a-1$};
			\path[ out=90,in=-90,preaction={white , line width =8pt}](a.south) to ($(b)+(0,2)$)node[above]{$1$};
		\end{tikzpicture}
		=
		\webQpar^{a-1} \splitWebTextStyle{a}{a-1}{1}
		\hspace{2cm}
		\begin{tikzpicture}[yscale=0.25,xscale=0.35,anchorbase,font=\small]
			\node (ab) at (0,2) {$a$}; 
			\node (a) at (-1,-1) {};
			\node (b) at (1,-1) {};
			\webCenter{ab}{a}{b}
			\path[very thick,out=90,in=-90,]	(WebCenter)	   edge  (ab)
			[out=-130,in=90,]		edge node {} (a)
			[out=-50,in=90]		edge  (b); 
			\path[ out=-90,in=90](a.north) to ($(b)+(0,-2)$) node[below]{$1$};
			\path[ out=-90,in=90,preaction={white , line width =8pt}](b.north) to ($(a)+(0,-2)$)node[below]{$a-1$};
		\end{tikzpicture}
		=
		\webQpar^{a-1} \mergeWebTextStyle{a}{a-1}{1}\\
		\begin{tikzpicture}[yscale=0.25,xscale=0.35	,anchorbase,font=\small]
			\node (ab) at (0,-2) {$a$}; 
			\node (a) at (-1,1) {};
			\node (b) at (1,1) {};
			\webCenter{ab}{a}{b}
			\path[very thick,out=-90,in=90,]	(WebCenter)	   edge  (ab)
			[out=130,in=-90,]		edge (a)
			[out=50,in=-90]		edge  (b); 
			\path[ out=90,in=-90](b.south) to ($(a)+(0,2)$) node[above]{$1$};
			\path[ out=90,in=-90,preaction={white , line width =8pt}](a.south) to ($(b)+(0,2)$)node[above]{$a-1$};
		\end{tikzpicture}
		=
		\webQpar^{a-1} \splitWebTextStyle{a}{1}{a-1}
		\hspace{2cm}
		\begin{tikzpicture}[yscale=0.25,xscale=0.35	,anchorbase,font=\small]
			\node (ab) at (0,2) {$a$}; 
			\node (a) at (-1,-1) {};
			\node (b) at (1,-1) {};
			\webCenter{ab}{a}{b}
			\path[very thick,out=90,in=-90,]	(WebCenter)	   edge  (ab)
			[out=-130,in=90,]		edge node {} (a)
			[out=-50,in=90]		edge  (b); 
			\path[ out=-90,in=90](a.north) to ($(b)+(0,-2)$) node[below]{$a-1$};
			\path[ out=-90,in=90,preaction={white , line width =8pt}](b.north) to ($(a)+(0,-2)$)node[below]{$1$};
		\end{tikzpicture}
		=
		\webQpar^{a-1} \mergeWebTextStyle{a}{1}{a-1}
		\end{gathered}
		\label{crossing-absorption}
	\end{align}
\end{lem}
\begin{proof}
	We only prove the first equality of \cref{useful}, the other is completely analogous.
	Use the thin square switch we obtain
	\begin{align*}
		\begin{tikzpicture}[anchorbase,scale=0.7]
			\def\height{1.3}
			\def\dist{0.8}
			\coordinate (bleft) at (0,0);
			\coordinate (bright) at (\dist, 0);
			\coordinate (tleft) at (0,\height);
			\coordinate (tright) at (\dist,\height);
			\coordinate (center) at (0.5*\dist,0.5*\height);
			\node[above] at (tleft) {$a$};
			\node[below] at (bleft) {$a$};
			\node[below] at (bright) {$1$};
			\node[above] at (tright) {$1$};
			\webCenter{center}{bleft}{bright}
			\path (WebCenter) edge (bleft)
			edge (bright)
			edge (center);
			\webCenter{center}{tleft}{tright}
			\path (WebCenter) edge (center)
			[] edge (tright)
			edge (tleft);
		\end{tikzpicture}  
			& \;=\; 	
		\begin{tikzpicture}[yscale=0.35,xscale=0.5]
			\draw[] (0,3) --(0,0) node [below,] {$a$};
							\draw (1,3) --(1,0) node [below,] {$1$};
			\path (0,0.7) edge node [below] {$1$}  (1,1.4)
				(1,1.6) edge node [above] {$1$} (0,2.3);
			\end{tikzpicture}
		-\Qint{a-1}
		\begin{tikzpicture}[scale=0.6]
			\def\height{1.5}
			\def\dist{0.8}
			\def\shape{0.2}
			\path[draw] (\dist,0) node[below]{$1$}to (\dist,\height);
			\path[](0,\height)  to (0,0)node[below]{$a$};
		\end{tikzpicture}
		\;=\;
		\begin{tikzpicture}[]
			\def\height{1}
			\def\dist{0.7}
			\def\shape{0.2}
			\path[] (0,0) node[below]{$a$}to  (0,\height) node [above]{$a$}
			(\dist,0) node[below]{$1$}to (0,\height-\shape*\height);
			\path[overline](\dist,\height) node[above]{$1$} to (0, \shape*\height);
			\path[] (0,0) node[below]{$a$}to  (0,\height) node [above]{$a$};
		\end{tikzpicture}
		+ \webQpar\inv 
		\begin{tikzpicture}
			\def\height{1}
			\def\dist{0.5}
			\def\shape{0.2}
			\path[draw] (\dist,0) node[below]{$1$}to (\dist,\height);
			\path[](0,\height)  to (0,0)node[below]{$a$};
			\path[in=0,out=0] (0,\shape*\height) to node [right] {$1$} (0,\height-\shape *\height);
		\end{tikzpicture}
		-\Qint{a-1}
		\begin{tikzpicture}[scale=0.6]
			\def\height{1.5}
			\def\dist{0.8}
			\def\shape{0.2}
			\path[draw] (\dist,0) node[below]{$1$}to (\dist,\height);
			\path[](0,\height)  to (0,0)node[below]{$a$};
		\end{tikzpicture}
	\end{align*}
	Removing the bigon by the explosion relation gives the desired result. The rest of the equalities follow easily from \cref{useful}, and we omit the proof.
\end{proof}

\subsection{Webs and fundamental representations}
\begin{thm}\label{webA-functor}
	Set $\webQpar=-q$. 
	For any $\linearGraph$, there is a full $\C(q)$-linear monoidal functor \[
		\funcA: \webCategoryA_{\funcField}\coloneqq\funcField\otimes_{\groundring}\webCategoryA \to \Fund
	\]
	which sends each generating object $a $ to $ \bigwedge^a \vecrep$, and the generating webs to the following maps:
	\begin{align}
		&
		\funcA\left(
		\begin{tikzpicture}[scale=0.25,anchorbase]
			\draw[] (0,0)node[below]{$1$} to (1,1);
			\draw[] (2,0)node[below]{$a-1$} to (1,1);
			\draw[,] (1,1) to (1,2)node[above]{$a$};
		\end{tikzpicture}
		\right)
		\colon\left(  \basis_p\otimes \monWedge_\beta  
		\mapsto
		\webQpar ^{|\{j\in \supp(\beta) \mid j> p\}|} \monWedge_{\beta\cup p}
		\right)
		\label{merge1nFormula}
		\,, \\
		&
		\funcA\left(
		\begin{tikzpicture}[scale=-0.25,anchorbase]
			\draw[] (0,0)node[above]{$a-1$} to (1,1);
			\draw[] (2,0)node[above]{$1$} to (1,1);
			\draw[,] (1,1) to (1,2)node[below]{$a$};
		\end{tikzpicture}
		\right) \colon 
		\left( \monWedge_\beta
		\mapsto
		\sum_s  \webQpar ^{|\{j\in \supp(\beta) \mid i> j\}|} \basis_{i} \otimes \monWedge_{\beta \setminus i}\right)  
		,  \\
		&
		\funcA\left(
		\begin{tikzpicture}[scale=-0.25,anchorbase]
			\draw[] (0,0)node[above]{$1$} to (1,1);
			\draw[] (2,0)node[above]{$a-1$} to (1,1);
			\draw[,] (1,1) to (1,2)node[below]{$a$};
		\end{tikzpicture}
		\right)
		\colon
		\left( \monWedge_\beta 
		\mapsto 
		\sum_s  \webQpar^{|\{j\in \supp(\beta)\mid  j> i\}|} \monWedge_{\beta \setminus i} \otimes \basis_{i}
		\right) \,,\\ 
		&
		\funcA\left(
		\begin{tikzpicture}[scale=0.25,anchorbase]
			\draw[] (0,0)node[below]{$a-1$} to (1,1);
			\draw[] (2,0)node[below]{$a$} to (1,1);
			\draw[,] (1,1) to (1,2)node[above]{$a$};
		\end{tikzpicture}
		\right)
		\colon
		\left(  \monWedge_\beta \otimes \basis_p 
		\mapsto \webQpar^{|\{j\in \supp(\beta)\mid i > j \}|}\monWedge_{\beta\cup p} \right)  
		. \label{merge-n1-formula-last}
	\end{align}
	The over-crossings of two thin strands are sent to the braiding on $\vecrep\otimes \vecrep$.
	Furthermore, $\funcA$ induces a $\funcField$-linear braided monoidal equivalence $\bar\funcA : \webCategoryA_{\funcField} / \id_{\dim \vecrep+1} \cong \Fund $.
\end{thm}

\begin{remark}\label{hecke=overcrossing}
Denoting $\underline{1^d} = (\overset{d}{\overbrace{1,\ldots, 1}})$, we have an algebra isomorphism  $\End_{\webCategoryA}(\underline{1^d})\cong \hecke[d]$ mapping the over-crossing in \cref{def:overcrossing} between the $i$-th and $(i+1)$-th strands to the standard generator $H_i$ for $1\leq i<d$. Via the functor $\funcA$, both give the braiding morphism in \cref{eqn:R-matrix generic form}.
\end{remark}

Denote $\webCategoryA_\C\coloneqq \C\otimes_{\groundring}\webCategoryA $.
\begin{prop}
	For any $\linearGraph$,
	there is a full monoidal functor $\webCategoryA_\C \to \polyRepCat[\mathfrak{gl}(\linearGraph)]$, such that $\funcA_\C(a)= \bigwedge^a \vecrep[\C]$, and the generating morphisms to \cref{merge1nFormula}--\cref{merge-n1-formula-last} with $\webQpar=-1$. 
\end{prop}

\begin{notation}\label{merge-split-web-notation}
Given object $\underline{a}=(a_1,a_2,\ldots,a_r)\in\webCategoryA$, denote
$	\splitNotation(\underline a)\coloneqq
	\begin{tikzpicture}[anchorbase,scale=0.7]
	\def\dist{0.8}
	\def\hoehe{0.8}
	\def\splitdist{0.1}
	\node[, ] (a) {};
	\draw($(a.south) +(0,-0.5*\hoehe)$) --+(0,-0.5*\hoehe) node [below]{$a_{1}$};
	\foreach \j in {0,1,...,4}{
	\draw ($(a.south) +(0,-0.5*\hoehe)$) --+(-0.2+\j*\splitdist,0.5*\hoehe);
	}
	\begin{scope}[xshift=1cm]
	\node[, ] (a) {};
	\draw($(a.south) +(0,-0.5*\hoehe)$) --+(0,-0.5*\hoehe) node [below]{$a_{2}$};
	\foreach \j in {0,1,...,4}{
	\draw ($(a.south) +(0,-0.5*\hoehe)$) --+(-0.2+\j*\splitdist,0.5*\hoehe);
	}
	\end{scope}
	\begin{scope}[xshift=2cm,yshift=-0.741cm]
	\node {$\cdots$};
	\end{scope}
	\begin{scope}[xshift=3cm]
	\node[, ] (a) {};
	\draw($(a.south) +(0,-0.5*\hoehe)$) --+(0,-0.5*\hoehe) node 					[below]{$a_{r}$};
	\foreach \j in {0,1,...,4}{
	\draw ($(a.south) +(0,-0.5*\hoehe)$) --+(-0.2+\j*\splitdist,0.5*\hoehe);
	}	
	\end{scope}	
	\end{tikzpicture}
$, $
	\mergeNotation(\underline a)\coloneqq
	\begin{tikzpicture}[anchorbase,yscale=-0.7,xscale=0.7]
	\def\dist{0.8}
	\def\hoehe{0.8}
	\def\splitdist{0.1}
	\node[, ] (a) {};
	\draw($(a.north) +(0,-0.5*\hoehe)$) --+(0,-0.5*\hoehe) node [above]{$a_{1}$};
	\foreach \j in {0,1,...,4}{
	\draw ($(a.north) +(0,-0.5*\hoehe)$) --+(-0.2+\j*\splitdist,0.5*\hoehe);
	}
	\begin{scope}[xshift=1cm]
	\node[, ] (a) {};
	\draw($(a.north) +(0,-0.5*\hoehe)$) --+(0,-0.5*\hoehe) node [above]{$a_{2}$};
	\foreach \j in {0,1,...,4}{
	\draw ($(a.north) +(0,-0.5*\hoehe)$) --+(-0.2+\j*\splitdist,0.5*\hoehe);
	}
	\end{scope}
	\begin{scope}[xshift=2cm,yshift=-0.541cm]
	\node {$\cdots$};
	\end{scope}
	\begin{scope}[xshift=3cm]
	\node[, ] (a) {};
	\draw($(a.north) +(0,-0.5*\hoehe)$) --+(0,-0.5*\hoehe) node 					[above]{$a_{r}$};
	\foreach \j in {0,1,...,4}{
	\draw ($(a.north) +(0,-0.5*\hoehe)$) --+(-0.2+\j*\splitdist,0.5*\hoehe);
	}
	\end{scope}
	\end{tikzpicture}
$, where all unlabeled strands are thin strands. 
For each ${\elmntSn} \in \abshuffleSet(\underline{a},\underline{b})$ we fix a reduced expression $\underline{\elmntSn}$. Let $\webnotation(\underline{\elmntSn}) $ be the web associated to $\underline{\elmntSn}$ obtained by replacing the crossings in the string diagram for $\underline{\elmntSn}$ by the thin over-crossing \cref{eqn: Hecke generators}. 
\end{notation}

\begin{remark}
	By Matsumoto's theorem, any two reduced expressions for $\elmntSn$ are related by a sequence of braid moves. Since the over-crossings in $\webCategoryA$ satisfy the braid relations, it follows that $\webnotation(\underline{\elmntSn})$ is independent of the choice of reduced expression $\underline{\elmntSn}$ for $\elmntSn$.
\end{remark}

We construct a basis for the morphisms spaces in $\webCategoryA$ out of the webs $\webnotation(\underline{\elmntSn})$.
\begin{thm}\label{webA-basis}
	The morphism space $\hom_{\webCategoryA}(\underline{a},\underline{b})$ is a free $\groundring$-module with a basis given by $\{\splitNotation(\underline{a})  \circ \webnotation(\underline{\elmntSn}) \circ  \mergeNotation(\underline{b}) \mid \elmntSn \in \abshuffleSet(\underline{a},\underline{b})\}$. 
\end{thm}
\begin{proof}
	Fix $\underline{a},\underline{b}\in \webCategoryA$ with $|\underline{a}|=|\underline{b}|= d $. Since quantum integers are invertible in $\groundring$, by explosion we have an embedding $\hom_{\webCategoryA}(\underline{a},\underline{b}) \hookrightarrow \hom_{\webCategoryA}(\underline{1^d},\underline{1^d}) $, hence we have $\hom_{\webCategoryA}(\underline{a},\underline{b}) =\{ \splitNotation(\underline{a})\circ f \circ \mergeNotation(\underline{b})\mid f\in \hom_{\webCategoryA}(\underline{1^d},\underline{1^d}) \}$. By \cref{lem:crossing absorption}, $\{\splitNotation(\underline{a})  \circ \webnotation(\underline{\elmntSn}) \circ  \mergeNotation(\underline{b}) \mid \elmntSn \in \abshuffleSet(\underline{a},\underline{b})\}$ is a spanning set of $\hom_{\webCategoryA}(\underline{a},\underline{b})$. 
	To see that they are linearly independent, by \cref{webA-functor}, there exists $\linearGraph$ large enough such that the map $\funcA:  \End_{\webCategoryA}(\underline{1^d})\otimes_{\groundring}\funcField \to \End_{\Uq(\linearGraph)}(\vecrep^{\otimes d})$ is bijective. By Schur--Weyl duality we may assume also $\End_{\Uq(\linearGraph)}(\vecrep^{\otimes d}) \cong \hecke[d]$ is the type $A$ Hecke algebra. Then the image of $\{\splitNotation(\underline{a})  \circ \webnotation(\underline{\elmntSn}) \circ  \mergeNotation(\underline{b}) \mid \elmntSn \in \abshuffleSet(\underline{a},\underline{b})\}$ in the Hecke algebra is linearly independent, hence they are linearly independent. 
\end{proof}

\section{The type $B$ web category} 
Let $\linearGraph\in\{\linearGraph_r,\linearGraph_r'\}$. We extend $\webCategoryA$ into $\webCategoryB$ to get all $\coideal(\linearGraph)$-module homomorphisms amongst tensor powers of fundamental representations. Since $\coideal(\linearGraph)$ is a right coideal, $\webCategoryB$ should be a right module category over $\webCategoryA$. In view of \cref{Schur-Weyl-quantum}, it is enough to add into the picture the $K$-matrix, which satisfies the defining relations \cref{DefHecke} of $\heckeSpecialized{d}$.
 
\begin{definition}
Let $\webQpar\in \{-q ,q\inv\}$. 
	The \emph{type $B$ web category} $\webCategoryB$ is the strict $\groundring$-linear right module category over $\webCategoryA$ generated by a single morphism $\dotNotationText[1]$ on the object $1$ satisfying: 
	\begin{align}\label{reflection-equation}
		\begin{tikzpicture}[scale=0.8]
			\draw[webdotted=0.3,webdotted=0.7]  (0,0) node[below] {$1$} -- +(0,1.5);
		\end{tikzpicture}
		=
		\begin{tikzpicture}[scale=0.8]
			\draw  (0,0) node[below] {$1$}-- +(0,1.5);
		\end{tikzpicture}
		\text{ and the \emph{reflection equation}}
		\quad
		\begin{tikzpicture}[scale=0.5]
			\def\dist{1}
			\def\hoehe{1.5}
			\path[in =-90,out=90,] (0,0) node[below]{$1$} to+ (\dist,\hoehe) ;
			\path[in =-90,out=90,]	(\dist,0)node[below]{$1$} to +(-\dist,\hoehe);
			\foreach \i in {0,1}{
			\path[in =-90,out=90,webdotted=0.85] (\dist,\i*\hoehe )   to + (-\dist,\hoehe)  ;
			\path[in =-90,out=90,overline] (0,\i*\hoehe )   to  + (\dist,\hoehe)  ;
			}
			\end{tikzpicture} 
			=
			\begin{tikzpicture}[scale=0.5]
			\def\dist{1}
			\def\hoehe{1.5}
			\path[in =-90,out=90,] (0,0) node[below]{$1$} to+ (\dist,\hoehe) ;
			\path[in =-90,out=90,]	(\dist,0)node[below]{$1$} to +(-\dist,\hoehe);
			\foreach \i in {0,1}{
			\path[in =-90,out=90,] (\dist,\i*\hoehe )   to  + (-\dist,\hoehe)  ;
			\draw[,in =-90,out=90,overline,webdotted=0.15] (0,\i*\hoehe )   to + (\dist,\hoehe)  ;
			}	
			\end{tikzpicture} 
	\end{align}
	Define $\webCategoryB_\C = \C\otimes_\groundring\webCategoryB $.
\end{definition}

Recall the identification $\End_{\webCategoryA}(\underline{1^d})\cong \hecke[d]$ in \cref{hecke=overcrossing}. By comparing the presentations \cref{reflection-equation} and \cref{DefHecke}, we get an isomorphism $\End_{\webCategoryB}(\underline{1^d}) \cong \heckeSpecialized{d}$ with $\dotNotationText[1 ]\mapsto H_0$.

\begin{thm}\label{rep-webB}
	Let $\webQpar=-q$.
	The functor $\funcA:\webCategoryA\to \Fund $ extends to a full functor \[
		\functorWebB \colon \webCategoryB_{\funcField}\to \FundCoideal
	\] 
	with $\functorWebB(\dotNotationText[1])= \kmatrix\in \End_{\funcField}(\vecrep)$, such that the following diagram commutes:
	\begin{align}
		\begin{tikzpicture}[scale=0.75,every path/.style={thin},font=\normalsize]
			\node (A) at (0,2) {$\webCategoryB_{\funcField} \times \webCategoryA_{\funcField}$};
			\node (B) at (6,2) {$\FundCoideal\times \Fund$};
			\node (C) at (0,0) {$\webCategoryB_{\funcField}$};
			\node (D) at (6,0) {$\FundCoideal$};
			\draw[->] (A) -- node[above]{$\funcA\times \functorWebB$} (B);
			\draw[->] (A) -- node[left]{$\otimes$} (C);
			\draw[->] (B) -- node[left]{$\otimes$} (D);
			\draw[->] (C) -- node[above]{$\functorWebB$} (D);
		\end{tikzpicture}
		\label{module-functor}
	\end{align}
\end{thm}
\begin{proof}
	The identification $\End_{\webCategoryB}(\underline{1^d}) \cong \heckeSpecialized{d}$ together with \cref{Schur-Weyl-quantum} implies that $\functorWebB$ is well-defined and $\functorWebB:\End_{\webCategoryB}(\underline{1^d}) \to \End_{\coideal(\linearGraph)}(\vecrep^{\otimes d})$ is surjective. Since any object in $\FundCoideal$ is a direct summand of $\vecrep^{\otimes d}$, we conclude that $\functorWebB$ is full. That the diagram \cref{module-functor} commutes is clear.
\end{proof}

\begin{definition}
	Let $a_1,a_2,\ldots,a_r,a\in \N$. Define the following shorthand notations:
\begin{align}\label{def-thick-dots}
	\begin{tikzpicture}[anchorbase,scale=0.25]
		\def\hoehe{2.5}
		\def\dist{1.2}
		\draw[]  (0*\dist,\hoehe) to +(0,-\hoehe) node[below]{$a_1$}; 
		\draw[]  (1*\dist,\hoehe) to +(0,-\hoehe) node[below]{$a_2$}; 
		\draw[]  (2.5*\dist,\hoehe) to +(0,-\hoehe) node[below]{$a_r$}; 
		\draw[webdotted=0.5]  (3.5*\dist,\hoehe) to +(0,-\hoehe) node[below]{$1$}; 
		\node[] at (1.75*\dist,0.5*\hoehe) {$\cdots$};
	\end{tikzpicture}
	\coloneqq  
	\begin{tikzpicture}[anchorbase,scale=0.3]
		\def\hoehe{2.5}
		\def\dist{1.2}
		\path[in=-90,out=90] (4*\dist,0) node[below]{$1$ }to (\dist,0.5*\hoehe) ;
		\foreach \i in {1,3}{
		\ifthenelse{\equal{\i}{1}}{
		\path[in=-90,out=90,preaction={white,line width=4pt}]  (\i*\dist,0)node[below]{$a_1$} to (\i*\dist+\dist, 0.5*\hoehe) to (\i*\dist,\hoehe); 
		}{
		\path[in=-90,out=90,preaction={white,line width=4pt}]  (\i*\dist,0)node[below]{$a_r$} to (\i*\dist+\dist, 0.5*\hoehe) to (\i*\dist,\hoehe); 
		}
		}
		\node[] at (3*\dist,0.5*\hoehe) {$\cdots$};
		\path[in=-90,out=90,preaction={white,line width=4pt},webdotted=0] (\dist,0.5*\hoehe) to (4*\dist,\hoehe);
	\end{tikzpicture},
	\quad
	\begin{tikzpicture}[anchorbase]
		\def\height{1}
		\draw[webdotted=0.5] (0,0) -- +(0,-1*\height) node [below] {$  a$};
	\end{tikzpicture} 
	\coloneqq
	\begin{tikzpicture}[scale=0.5*0.5]
		\def\wingspan{0.5}
		\def\hoehe{2}
		\def\dist{2.4}
		\path[out=130,in=-90,] (0,\wingspan)  to  +(-1.5*\wingspan,2*\wingspan)  ;		
		\path[out=50,in=-90]	(0,\wingspan)	to  +(1.5*\wingspan,2*\wingspan);
		\path[very thick,out=-90,in=90, draw]	(0,\wingspan)  to (0,0) node[below] {$a$} ;
		\path[in=-90,out=90] (1.5*\wingspan,3*\wingspan) 	to  node[ right]{$a-1$} +(0,2*\wingspan)
		;
		\path[, in=-90,out=90,webdotted=0.25] (-1.5*\wingspan,3*\wingspan) to  node[left]{$1$}+(0, 2*\wingspan);
		\path[, in=-90,out=90] 	(-1.5*\wingspan,5*\wingspan)	to +(1.5*\wingspan, 2*\wingspan)
		to +(1.5*\wingspan, 2*\wingspan)
		(1.5*\wingspan,5 *\wingspan) to +(-1.5*\wingspan, 2*\wingspan)	
		to +(-1.5*\wingspan, 3*\wingspan)	
		;
	\end{tikzpicture}
	,\quad
	\begin{tikzpicture}[scale=0.5]
		\def\hoehe{2}
		\def\dist{1.2}
		\foreach \i in {-1}{
		\path[draw]  (\i*\dist,0) node[below] {$a_1$} to +(0,\hoehe) ;}
		\foreach \i in {1}{
		\path[draw]  (\i*\dist,0) node[below] {$a_r$} to +(0,\hoehe) ;}
		\path[webdotted=0.5] (2*\dist,0) node[below]{$a$} to +(0,\hoehe);
		\node[] at  (0,0.5*\hoehe) {$\cdots$};
	\end{tikzpicture}
	\coloneqq 
	\begin{tikzpicture}[scale=0.5*0.5]
		\def\wingspan{0.5}
		\def\hoehe{2}
		\def\dist{2.4}
		\path[out=130,in=-90,] (0,\wingspan)  to  +(-1.5*\wingspan,2*\wingspan)  ;		
		\path[out=50,in=-90]	(0,\wingspan)	to  +(1.5*\wingspan,2*\wingspan);
		\path[very thick,out=-90,in=90, draw]	(0,\wingspan)  to (0,0) node[below] {$a$} ;
		\path[in=-90,out=90] (1.5*\wingspan,3*\wingspan) 	to  node[ right]{$a-1$} +(0,2*\wingspan)
		;
		\path[, in=-90,out=90,webdotted=0.25] (-1.5*\wingspan,3*\wingspan) to  node[left]{$1$}+(0, 2*\wingspan);
		\path[, in=-90,out=90] 	(-1.5*\wingspan,5*\wingspan)	to +(1.5*\wingspan, 2*\wingspan)
		to +(1.5*\wingspan, 2*\wingspan)
		(1.5*\wingspan,5 *\wingspan) to +(-1.5*\wingspan, 2*\wingspan)	
		to +(-1.5*\wingspan, 3*\wingspan)	
		;
		\begin{scope}[xshift=-1.5cm]
		\foreach \i in {-1}{
		\path[draw]  (-4-\i*\dist,-0*\hoehe) node[below] {$a_r$} to +(0,2*\hoehe) ;}
		\foreach \i in {1}{
		\path[draw]  (-4-\i*\dist,-0*\hoehe) node[below] {$a_1$} to +(0,2*\hoehe) ;}
		\node[] at  (-4,1*\hoehe) {$\cdots$};
		\end{scope}
	\end{tikzpicture}
	.
\end{align}
\end{definition}
The definition \cref{def-thick-dots} is inspired by the Jucys--Murphy elements of $\heckeSpecialized{d}$, see \cref{specht-module}. Under the identification $\End_{\webCategoryB}(\underline{1^d}) \cong \heckeSpecialized{d}$, the dot on the $i$-th strand corresponds precisely to the $i$-th Jucys--Murphy element $J_i\in \heckeSpecialized{d}$.
 
\begin{remark}	
	Dots do not slide through crossings. One can see that the following pairs of elements are distinct since they have distinct images under $\functorWebB$ (we leave the check to the readers):
\begin{align}
	\begin{tikzpicture}[anchorbase,scale=0.7]
		\def\height{1.3}
		\def\dist{0.8}
		\coordinate (bleft) at (0,0);
		\coordinate (bright) at (\dist, 0);
		\coordinate (tleft) at (0,\height);
		\coordinate (tright) at (\dist,\height);
		\node[below] at (bleft) {1};
		\node[below] at (bright) {1};
		\overcrossing{bleft}{bright}{tright}{tleft}{}
		\path[ webdotted=0.2,in=-90,out=90] (bleft) to (tright);
	\end{tikzpicture} 
	\ne 
	\begin{tikzpicture}[anchorbase,scale=0.7]
		\def\height{1.3}
		\def\dist{0.8}
		\coordinate (bleft) at (0,0);
		\coordinate (bright) at (\dist, 0);
		\coordinate (tleft) at (0,\height);
		\coordinate (tright) at (\dist,\height);
		\node[below] at (bleft) {1};
		\node[below] at (bright) {1};
		\overcrossing{bleft}{bright}{tright}{tleft}{}
		\path[ webdotted=0.8,in=-90,out=90] (bleft) to (tright);
	\end{tikzpicture} 
		,\qquad
	\begin{tikzpicture}[anchorbase,scale=0.7]
		\def\height{1.3}
		\def\dist{0.8}
		\coordinate (bleft) at (0,0);
		\coordinate (bright) at (\dist, 0);
		\coordinate (tleft) at (0,\height);
		\coordinate (tright) at (\dist,\height);
		\node[below] at (bleft) {1};
		\node[below] at (bright) {1};
		\path[ webdotted=0.2,in=90,out=-90] (tleft) to (bright);
		\overcrossing{bleft}{bright}{tright}{tleft}{}
	\end{tikzpicture} 
	\ne 
	\begin{tikzpicture}[anchorbase,scale=0.7]
		\def\height{1.3}
		\def\dist{0.8}
		\coordinate (bleft) at (0,0);
		\coordinate (bright) at (\dist, 0);
		\coordinate (tleft) at (0,\height);
		\coordinate (tright) at (\dist,\height);
		\node[below] at (bleft) {1};
		\node[below] at (bright) {1};
		\path[ webdotted=0.8,in=90,out=-90] (tleft) to (bright);
		\overcrossing{bleft}{bright}{tright}{tleft}{}
	\end{tikzpicture} 
	,\qquad
	\begin{tikzpicture}[anchorbase,scale=0.4]
		\def\hoehe{2.5}
		\def\dist{1.2}
		\draw[]  (0*\dist,\hoehe) to +(0,-\hoehe) node[below]{$a_1$}; 
		\draw[]  (1*\dist,\hoehe) to +(0,-\hoehe) node[below]{$a_2$}; 
		\draw[]  (2.5*\dist,\hoehe) to +(0,-\hoehe) node[below]{$a_r$}; 
		\draw[webdotted=0.5]  (3.5*\dist,\hoehe) to  +(0,-\hoehe) node[below]{$a$}; 
		\node[] at (1.75*\dist,0.5*\hoehe) {$\cdots$};
	\end{tikzpicture}
	\;\ne \; 
	\begin{tikzpicture}[anchorbase,scale=0.4]
		\def\hoehe{2.5}
		\def\dist{1.2}
		\path[in=-90,out=90] (4*\dist,0) node[below]{$a$ }to (\dist,0.5*\hoehe) ;
		\foreach \i in {1,3}{
		\ifthenelse{\equal{\i}{1}}{
		\path[in=-90,out=90,preaction={white,line width=4pt}]  (\i*\dist,0)node[below]{$a_1$} to (\i*\dist+\dist, 0.5*\hoehe) to (\i*\dist,\hoehe); 
		}{
		\path[in=-90,out=90,preaction={white,line width=4pt}]  (\i*\dist,0)node[below]{$a_r$} to (\i*\dist+\dist, 0.5*\hoehe) to (\i*\dist,\hoehe); 
		}
		}
		\node[] at (3*\dist,0.5*\hoehe) {$\cdots$};
		\path[in=-90,out=90,preaction={white,line width=4pt},webdotted=0.01] (\dist,0.5*\hoehe) to (4*\dist,\hoehe);
	\end{tikzpicture}
	\quad
	\text{ for $a>1$}
	.
\end{align}
\end{remark}

\begin{lem}\label{dot-commute-with-far-away-web}
	Let $\webnotation\in \webCategoryB$ be any web, then the following relations hold:
	\begin{align}
	\begin{tikzpicture}[scale=0.5]
		\def\hoehe{2}
		\def\dist{1.2}
		\foreach \i in {-1}{
		\path[draw]  (\i*\dist,0) node[below] {$a_1$} to +(0,\hoehe) ;}
		\foreach \i in {1}{
		\path[webdotted=0.3]  (\i*\dist,0) node[below] {$a_r$} to +(0,\hoehe) ;}
		\node[] at (2*\dist,0.5*\hoehe) {$\cdots$};
		\path[webdotted=0.7] (3*\dist,0) node[below]{$a_s$} to +(0,\hoehe);
		\node[] at  (0,0.5*\hoehe) {$\cdots$};
	\end{tikzpicture}
	=
	\begin{tikzpicture}[scale=0.5]
		\def\hoehe{2}
		\def\dist{1.2}
		\foreach \i in {-1}{
		\path[draw]  (\i*\dist,0) node[below] {$a_1$} to +(0,\hoehe) ;}
		\foreach \i in {1}{
		\path[webdotted=0.7]  (\i*\dist,0) node[below] {$a_r$} to +(0,\hoehe) ;}
		\node[] at (2*\dist,0.5*\hoehe) {$\cdots$};
		\path[webdotted=0.3] (3*\dist,0) node[below]{$a_s$} to +(0,\hoehe);
		\node[] at  (0,0.5*\hoehe) {$\cdots$};
	\end{tikzpicture}
	,\qquad\qquad
	\begin{tikzpicture}[scale=0.3]
	\draw (0,2.5) -- (0,1.2);
	\draw (0,0.2) -- (0,-1.2);
	\draw (3,2.5) -- (3,1.2);
	\draw (3,0.2) -- (3,-1.2);
	\node at (0,3) {$b_1$};
	\node at (1.5,2.5) {$\cdots$};
	\node at (3,3) {$b_r$};
	\draw (-0.1,0.2) rectangle (3.1,1.2);
	\node at (1.5,0.7) {$\webnotation$};
	\node at (1.5,-0.4) {$\cdots$};
	\node at (0,-1.8) {$a_1$};
	\node at (3,-1.8) {$a_r$};
	\draw[webdotted=0.2] (5,2.5) -- (5,-1.2);
	\node at (5,-1.8) {$a_{r+1}$};
	\end{tikzpicture}
	=
	\begin{tikzpicture}[scale=0.3]
	\draw (0,2.5) -- (0,1.2);
	\draw (0,0.2) -- (0,-1.2);
	\draw (3,2.5) -- (3,1.2);
	\draw (3,0.2) -- (3,-1.2);
	\node at (0,3) {$b_1$};
	\node at (1.5,2.5) {$\cdots$};
	\node at (3,3) {$b_r$};
	\draw (-0.1,0.2) rectangle (3.1,1.2);
	\node at (1.5,0.7) {$\webnotation$};
	\node at (1.5,-0.4) {$\cdots$};
	\node at (0,-1.8) {$a_1$};
	\node at (3,-1.8) {$a_r$};
	\draw[webdotted=0.8] (5,2.5) -- (5,-1.2);
	\node at (5,-1.8) {$a_{r+1}$};
	\end{tikzpicture}
	\end{align}
\end{lem}
\begin{proof}
	Unwinding the definition, the first relation follows from the reflection equation \cref{reflection-equation}. Then the second equation follows from the first one and the naturality of the crossings for type $A$ webs together with the reflection equation \cref{reflection-equation}. 
\end{proof}

The following somewhat surprising result is a shadow of the fact that the type $B$ braid group is isomorphic to the affine type $A$ braid group.
\begin{lemma}\label{\webQpar=-1WebB-is-Monoidal} 
	Horizontal juxtaposition of diagrams defines a strict monoidal structure on $\webCategoryB_\C$. 
\end{lemma}
\begin{proof}
	Tensoring with $ \dotNotationText[1]$ from the right is interpreted as \cref{def-thick-dots} and \cref{def-thick-dots}. Then by \cref{dot-commute-with-far-away-web}, this indeed defines a monoidal structure on $\webCategoryB_\C $.
\end{proof}

\begin{remark}
	In the proof of \cref{\webQpar=-1WebB-is-Monoidal}, we did not use the condition $q=1$. In fact, one could also view $\webCategoryB$ as a monoidal category with tensor product given by horizontal juxtaposition of diagrams. Moreover, this tensor product coincides with the one obtained from the affinization of $\webCategoryA$ in the sense of \cite{MoSa-affinization}, if one interprets the $\dotNotationText[1]$ on the leftmost thin strand as one of the generating morphisms of the affinization. 
\end{remark}

We can understand indecomposable direct summands in $\webCategoryB_\C$.  
Let $\polyRepCat$ be the category of polynomial representations of $\classicalFixPointSubalg$, i.e. the full subcategory of $\classicalFixPointSubalg$-modules whose objects are direct summands of finite direct sums of tensor powers of the natural representation $\vecrep[\C]$.
\begin{prop}\label{classical-webB-action-ss}
	Let $\linearGraph\in\{\linearGraph_r,\linearGraph_r'\}$. Then there is a $\C$-linear full \emph{monoidal} functor $\functorWebB_\C \colon \webCategoryB_\C \to \polyRepCat$ which sends $\dotNotationText$ to the $K$-matrix.
\end{prop}
\begin{proof}
	That $\functorWebB_\C$ is a full functor follows from the same argument as \cref{rep-webB}, where one uses the Schur-Weyl duality \cref{SW-duality-q=1}. 
	To check that $\functorWebB_\C$ is indeed monoidal, we need to show that for any $\webnotation\in\webCategoryB_\C$, we have $\functorWebB_\C (\webnotation \otimes \dotNotationText ) = \functorWebB_\C(\webnotation) \otimes \functorWebB_\C(\dotNotationText)$. By \cref{def-thick-dots}, we may assume that $\dotNotationText$ is on a thin strand.
	Since the crossing in $\webCategoryA_\C $ is sent to swapping tensor factors, $\functorWebB_\C (\webnotation \otimes \dotNotationText )$ is just applying $\functorWebB_\C (\webnotation)$ and the $\kmatrix_k$ from \cref{K-mat-on-k} on the corresponding tensor factors respectively, which is equal to $\functorWebB_\C(\webnotation) \otimes \functorWebB_\C(\dotNotationText)$. This shows that $\functorWebB_\C$ is monoidal.
\end{proof}

\begin{lem}\label{kar-env-WebB}
	Let $\Kar(\webCategoryB_\C)$ be the Karoubian envelope of~ $\webCategoryB_\C$. The following hold:
	\begin{enumerate}
		\item Let $\linearGraph\in\{\linearGraph_r,\linearGraph_r'\}$. The functor $\functorWebB_\C$  from \cref{classical-webB-action-ss} extends to a full $\C$-linear additive monoidal functor $\widehat{\functorWebB_\C}:\Kar(\webCategoryB_\C)\to  \polyRepCat$.
		\item Let $\linearGraph=\linearGraph_\infty$. Then $\widehat{\functorWebB_\C}$ is a $\C$-linear additive monoidal equivalence $\Kar(\webCategoryB_\C)\cong \polyRepCat$. In particular, $\Kar(\webCategoryB_\C)$ is a semisimple monoidal Abelian category with simple objects indexed by bipartitions.
	\end{enumerate}  
\end{lem}
\begin{proof}
	Since  $\polyRepCat$ is Abelian and monoidal with $\Fund|_{\classicalFixPointSubalg}$ as an additive monoidal subcategory, there exists a unique extension $\Kar(\webCategoryB_\C)\to \polyRepCat$.
	Now let $\linearGraph=\linearGraph_\infty$.
	By \cref{SW-duality-q=1}, in this case the functor $\functorWebB_\C $ is also faithful, hence so is $\widehat{\functorWebB_\C}$. Then $\Kar(\webCategoryB)$ is equivalent to its essential image, which is the Karoubian envelope of $\Fund|_{\classicalFixPointSubalg}$, i.e. the category of finite dimensional $\classicalFixPointSubalg$-modules. 
	Let $L$ be a finite dimensional simple $\classicalFixPointSubalg$-module,
	which is equivalent to the action map $U(\classicalFixPointSubalg)\to \End_\C(L)$ being surjective. Then there exists $N\in \N$ such that $U(\classicalFixPointSubalg[N])\to \End_\C(L)$ is surjective and hence $L$ is already irreducible over $\classicalFixPointSubalg[\linearGraph_N]$, and hence its highest weight is a bipartition of at most $N$ rows. Now Let $V$ be a finite dimensional $\classicalFixPointSubalg$-module, and take a Jordan--H\"older series of $V$. There is $N\in \N$ such that each simple factor is already simple over $\classicalFixPointSubalg[\linearGraph_N]$, hence $V$ splits into simple summands. 
	Pulling back via the equivalence $\widehat{\functorWebB_\C}$ concludes the proof.
\end{proof}

\subsection{Some relations in $\webCategoryB$}
We deduce now some, probably non-obvious, relations from the defining relations. We will later apply these relations to decompose anti-spherical Hecke modules.
\begin{lem}\label{local-relations}
	The following \emph{local} relations hold in $\webCategoryB$:
	\begin{align}\label{commuting-relations-web}

	\end{align}
	\begin{align}\label{eqn:crossing absoption dot left thick left}
		%
	\end{align}	
\end{lem}
The proof is just unwinding definitions and left to the reader. We next use these relations to derive the following $\webQpar$-(co)derivation formulas:
\begin{prop}
	The following $\webQpar$-(co)derivation relations hold \emph{locally} in $\webCategoryB$:
	\begin{align}\label{q-derivation}
		%
		\label{many-strand-derivation}
	\end{align}
	where the unlabeled strands in the last two equalities are thin.
\end{prop}
\begin{proof}
	We only prove \cref{q-derivation}, then \cref{many-strand-derivation} follows easily.
	First we show the $a=1$ case of \cref{q-derivation}:
	\begin{align*}
		\hspace{-0.5cm}
		%
		\;\overset{\cref{eqn:crossing absoption dot left thick left}}{=}\;
		\webQpar\inv
		%
	\end{align*}
	as desired. Next we prove the case $a>1$ inductively (i.h.= induction hypothesis):
	\begin{align*}
		\Qint{a}
		%
	\end{align*}
	Now the equality $\webQpar^{1-a}+\webQpar\Qint{a-1}=\Qint{a}$ concludes the induction. This proves the first identity of \cref{q-derivation}, the second is similar.
\end{proof}

In what follows we write $\eta = \webQpar-\webQpar^{-1}$.
\begin{lem}
The following \emph{non-local} identity holds, where the top strands are all thin.
\begin{align}\label{mpi-formula}
	%
	\;(i\geq 2).
\end{align}
\end{lem}
\begin{proof}
We prove by induction on $k$. For the base case $k=2$, one calculates
\begin{align}\label{base-case-mpi}
	%
\end{align}
Now we assume the claim holds for $k-1$ and prove it for $k$. We have
\begin{align*}
	%
	\\
	&\overset{\cref{eqn:crossing absoption dot left thick left},\cref{crossing-absorption}}{=}
	\webQpar\eta\;
	\sum_{1\leq j<i}
	%
\end{align*}
where i.h. stands for induction hypothesis. This finishes the proof.
\end{proof}

Recall \cref{commuting-relations-web} that $\dotNotationText[1]\dotNotationText[1]\ldots \dotNotationText[1]$ commutes with the crossings and hence lie in the center of $\heckeSpecialized{d}$. We compute the product of this central element with $\dotNotationText$ composed with $\splitNotation(k)$. 
\begin{thm}\label{symfunc-relation}
The following \emph{non-local} relation holds in $\webCategoryB$:
\begin{align}
	%
	,
\end{align}
where all unlabeled strands are thin.
\end{thm}
\begin{proof}
We do induction on $k$. The first equation is clear for $k=1$ and holds by \cref{base-case-mpi} for $k=2$ which is the base of induction. By definition, \cref{commuting-relations-web}, and induction we then have
\begin{align*}
	%
\end{align*}
This proves the first equation. 
To show the second one, we use the first one to compute
\begin{align*}
	%
\end{align*}
Finally, one checks $\sum_{i<j\leq k}\webQpar^{2j-3}\eta + \webQpar^{2i-2} = \webQpar^{2k-2}$ ($\forall i<k$) and finishes the proof.
\end{proof}

\begin{cor}\label{ippencafe}
\tikzset{
    dot/.style={circle, draw=black, fill=black!80, inner sep=1.2pt}
}
	The following relations hold \emph{on the leftmost strand} in $\webCategoryB$:
	\begin{align}
%
\label{bells}
\end{align}
where all stands in the middle are thin.
\end{cor}
\begin{proof}
	Recall \cref{merge-split-web-notation}, the desired result follows from
	composing the relations in \cref{symfunc-relation} with $\mergeNotation(k)$ on the top, then use the first identity of \cref{commuting-relations-web} and sum up the coefficients. 
\end{proof}

\subsection{A basis theorem}

Recall the type $B$ Weyl group $W_d\cong S_d\ltimes (\Z/2\Z)^d$, the set of $(\underline{a},\underline{b})$-shuffles $\abshuffleSet(\underline{a},\underline{b})\subset S_d$ from \cref{ab-shuffle-def}, and  the set of dotted $(\underline{a},\underline{b})$-shuffles $\dSh(\underline{a},\underline{b})\subset W_d$ in \cref{dotted-shuffle-def}. We now establish a type $B$ analogue of the basis in \cref{webA-basis}. 
\begin{notation}
	For each dotted $(\underline{a},\underline{b})$-shuffle $m \elmntSn  \in W_d$ with $\elmntSn\in \abshuffleSet(\underline{a},\underline{b})$, $m\in (\Z/2\Z)^d$, we associate a web $\webnotation(\elmntSn m)$ by putting the string diagram of $m$ at the bottom of $\webnotation(\elmntSn )$.
\end{notation}

\begin{thm}[Diagrammatic basis]\label{basis-typeB-Web}
	For objects $\underline{a},\underline{b}$ of $\webCategoryB$, the morphism space
	$\hom_{\webCategoryB}(\underline a, \underline b)$ is a free $\groundring$-module with an $\groundring$-basis given by 
	$\{ \splitNotation(\underline a) \circ \webnotation(\underline{\elmntSn m})\circ	\mergeNotation(\underline b)\mid \elmntSn m\in \dSh(\underline{a},\underline{b}) \} $.
\end{thm}
\begin{proof}
	Fix $\underline{a},\underline{b}$ with $|\underline{a}|=|\underline{b}|=d$.
	Note that we have $\End_{\webCategoryB}(\underline{1^d})\otimes_{\groundring}\C = \C[W_d]$ is the group algebra of the type $B$ Weyl group. Then by explosion and \cref{SW-duality-q=1}, we get that $\hom_{\webCategoryB}(\underline{a},\underline{b}) = \splitNotation(\underline{a})\circ \hom_{\webCategoryB}(\underline{1^d},\underline{1^d})\circ \mergeNotation(\underline{b}) $.
	Then by \cref{lem:crossing absorption}, the claimed basis elements generate $\hom_{\webCategoryB}(\underline{a},\underline{b})$. 
	To see linear independence, by \cref{Schur-Weyl-quantum}, there exists $\linearGraph_r$ such that the map $\functorWebB : \End_{\webCategoryB}(\underline{1^d}) \to \End_{\coideal(\linearGraph_r)}  $ is an isomorphism. Then $\{ \splitNotation(\underline a) \circ \webnotation(\underline{\elmntSn m})\circ	\mergeNotation(\underline b)\mid \elmntSn m\in \dSh(\underline{a},\underline{b}) \} $ is linearly independent because its image in $ \heckeSpecialized{d}$ is linearly independent.
\end{proof}

\section{Spectrum of the dot}
Fix a linear graph $\linearGraph\in\{\linearGraph_r,\linearGraph_r'\}$. Let $\webQpar=-q$ throughout the rest of this paper.

\subsection{The action of the dot}
Recall the notation $\flip{i}{\beta}$ from \cref{boxdiag-opeations}.
\begin{prop}
    \label{officeAlone}
	Let $\linearGraph\in\{\linearGraph_r,\linearGraph_r'\}$. For any $\beta\in {}_k\boxdiagramSet(\linearGraph)$ we have
     \begin{align}
      \functorWebB(\dotNotationText)( \monWedge_\beta )= \sum_{p\in \supp(\beta)} \webQpar^{\bnumber_p^\beta} \monWedge_{\flip{p}\beta} 
	  ,\quad \text{ respectively }\quad
	  \functorWebB_\C(\dotNotationText)( \monWedge_\beta )= \sum_{p\in \supp(\beta)}  (-1)^{\bnumber_p^\beta} \monWedge_{\flip{p}\beta} 
	  \,.
	  \label{dot-action-monwedge}
    \end{align}
\end{prop}
\begin{proof}
	Suppose the vertices in $\supp(\beta)$ are $ p_1>\cdots>p_r$. In view of \cref{merge1nFormula}, we have
    \begin{align*}
        \functorWebB(\dotNotationText) \cdot \monWedge_\beta &=
        \sum_{i} 
        \basis_{\foldVertex{p_i}}\wedge \webQpar^{i-|\supp(\beta)|} \monWedge_{{\beta}\setminus{p_i}} 
		= \sum_{i} \webQpar^{i-|\supp(\beta)|} \webQpar^{|\{s\mid p_s> \foldVertex{p_i}\} |} \monWedge_{\Ijbar{\beta}{i}}, 
    \end{align*}    
	now consider the straightening ${}^{\unfoldSupScript}\beta$, then $|\supp(\beta)|-i$ is the number of nonempty boxes below the $p$-th box, and $|\{s\mid p_s\to \foldVertex{p_i}\}|$ is the number of nonempty boxes above the $\foldVertex{p_i}$-th one. Hence, regardless of whether $p_i\in \linearGraph^+$, we always have $i-|\supp(\beta)|+|\{s\mid p_s> \foldVertex{p_i}\}| = \bnumber_{p_i}^\beta$. This proves the first identity. Finally, the same calculation with $\webQpar=-1$ gives the second identity in \cref{dot-action-monwedge}.
\end{proof}

\begin{prop}[Funny vectors]
	\label{eigenvec-constr}
	Let $w = \sum_\beta a_\beta \monWedge_\beta \in \bigwedge^k \vecrep$ be an eigenvector of~ $\functorWebB(\dotNotationText)$ of eigenvalue $\Qint{m}$. Assume $p >0-\supp(\beta)$ satisfies $p'> p$ for all $ p'\in \supp(\beta)$ with 
	$a_\beta\ne 0$.
	Then the following hold:
	\begin{enumerate}
		\item $w \wedge (\basis_p + \webQpar^{m-k} \basis_{\foldVertex{p}}) $ is an eigenvector of~ $\functorWebB(\dotNotationText[k+1])$ of eigenvalue $\Qint{m+1}$,
		\item $w \wedge (\basis_p - \webQpar^{-m-k} \basis_{\foldVertex{p}}) $ is an eigenvector of~ $\functorWebB(\dotNotationText[k+1])$ of eigenvalue $\Qint{m-1}$.
	\end{enumerate}
\end{prop}
\begin{remark}
	\Cref{eigenvec-constr} has a specialisation $\webQpar=-1$ describing its classical analogue.
\end{remark}
\begin{proof}
	Using \cref{officeAlone}, we have for any box diagram $\beta$ that $\functorWebB(\dotNotationText[k+1]) $ acts by
	\begin{align*}
		\hspace{-0.5cm}
	 	\monWedge_{\beta\cup p} 
		\mapsto
		\sum_{p'} \webQpar^{\bnumber^\beta_{p'}-1} \monWedge_{({\flip{{p'}}\beta}) \cup p} + \monWedge_{\beta\cup\foldVertex{p}}
		\,,\qquad
		\monWedge_{\beta\cup{\foldVertex{p}}} 
		\mapsto 
		\sum_{p'} \webQpar^{\bnumber^\beta_{p'}+1} \monWedge_{({\flip{p'}\beta})\cup {\foldVertex{p}}} + \monWedge_{\beta\cup{{p}}}  
		\,.
	\end{align*}
	By \cref{merge-n1-formula-last}, we have $\monWedge_\beta \wedge \basis_p = \monWedge_{\beta\cup p}$ and $\monWedge_\beta \wedge \basis_{\foldVertex{p}} = \webQpar^{|\supp(\beta)|} \monWedge_{\beta\cup{{\foldVertex{p}}}}=\webQpar^k\monWedge_{\beta\cup\foldVertex{p}}$. Then 
	\begin{align*}
		\functorWebB(\dotNotationText[k+1])( w\wedge \basis_p )
		&= 
		\functorWebB(\dotNotationText[k+1])(\sum_\beta a_\beta \monWedge_{\beta\cup p})
		=
		\sum_\beta a_\beta \,  \left( \sum_{p'} \webQpar^{\bnumber^\beta_{p'}-1} \monWedge_{({\flip{p'}\beta})\cup p} + \monWedge_{\beta\cup\foldVertex{p}}  \right)
		\\
		&= \webQpar\inv \Qint{m} w\wedge \basis_p +\webQpar^{-k} w\wedge \basis_{\foldVertex{p}}
		,\\
		\functorWebB(\dotNotationText[k+1])(\webQpar^{-k} w\wedge \basis_{\foldVertex{p}} )
		&= 
		\functorWebB(\dotNotationText[k+1])(\sum_\beta a_\beta   \monWedge_{\beta\cup\foldVertex{p}})
		=
		\sum_\beta a_\beta   \left( \sum_{j} \webQpar^{\bnumber^\beta_j+1} \monWedge_{({\flip{p'}\beta}) \cup\foldVertex{p}} + \monWedge_{\beta\cup p}  \right)
		\\
		&= \webQpar^{1-k} \Qint{m} w\wedge \basis_{\foldVertex{p}} + w\wedge \basis_{p}
		.
	\end{align*}
	Hence,  $\functorWebB(\dotNotationText)$acts on the $2$-dimensional subspace spanned by $\{ w\wedge \basis_p,\webQpar^{-k} w  \wedge \basis_{\foldVertex{p}}\}$ with matrix\[
		\begin{pmatrix}
			\webQpar\inv\Qint{m} & 1\\
			1 & \webQpar\Qint{m}
		\end{pmatrix},
	\]
	whose eigenvalues are exactly $\Qint{m\pm1}$, with eigenvectors as claimed. 
\end{proof}

\begin{remark}
	A slight modification of our proof here shows that the conclusion of \cref{eigenvec-constr} holds under the weaker assumption $p'> p$ for all $ p'\in \supp(\beta)$ with $\foldVertex{{p'}}\notin \supp(\beta)$. 
\end{remark}

\subsection{Eigenvectors via box diagrams}
Now we construct eigenvectors of~ $\functorWebB(\dotNotationText)$. 

\begin{definition}\label{def-omega-vecs-bothcases}
Given $\beta\in {}_{k}\boxdiagramSet(\linearGraph_r)$ with $\supp(\beta)\cap  \graphVertices_{\linearGraph^+} = \{p_1> \cdots > p_s \}$, define 
\begin{align}
	\label{OmegaDef}
	\Omega_{\beta}  \coloneqq z(p_1,\beta)\wedge \cdots \wedge z(p_s,\beta)  \in \bigwedge^{k}\vecrep
	,
\end{align}
where, denoting $a_i = |\beta\inv(-)\cap \{p_1,\ldots, p_{i-1}\}|, b_i = |\beta\inv(\pm)\cap \{p_1,\ldots, p_{i-1}\}|$,
\[
	z(p_i,\beta) \coloneqq \begin{cases}
		\basis_{\foldVertex{p_i}}\wedge\basis_{{p_i}} , & \beta(p_i)=\pm,\\
		\basis_{p_i} + \webQpar^{-2a_i-2b_i}\basis_{\foldVertex{p_i}} , & \beta(p_i)= +,\\
		\basis_{p_i} - \webQpar^{2a_i+2-2i}\basis_{\foldVertex{p_i}} , &\beta(p_i)=-,\\
		\basis_0, & p_i =0 \in \supp(\beta).
	\end{cases}
\]
\end{definition}
For example, consider $\linearGraph_3$, then the corresponding vector is\[
	\text{ for }
	\beta = \ytabcenter{+\\ \pm\\-\\+\\+} ,
	\quad
	\Omega_\beta= 
	(\basis_{0}+\basis_{\foldVertex{0}})
	\wedge 
	\basis_{\foldVertex{1}} \wedge \basis_{1} 
	\wedge (\basis_{2}-\webQpar^{-4} \basis_{\foldVertex{2}})
	\wedge (\basis_{3}+\webQpar^{-4}\basis_{\foldVertex{3}})
	\wedge (\basis_{4}+\webQpar^{-4}\basis_{\foldVertex{4}}).
\]
Write $\Omega_\beta = \sum_{\gamma} a_{\gamma}\monWedge_{\gamma}$ as a linear combination of standard monomial wedges, and let $N=\{\gamma\mid a_\gamma\ne 0\}$. 
Then by our construction, any $\gamma\in N$ is obtained from $\beta$ by flipping some signs in the boxes $\beta\inv(\{+,-\})$. In particular, we have $\gamma\inv(\pm)=\beta\inv(\pm) $ for all $\gamma\in N$. Furthermore, if $\gamma\in N$ satisfies $|\gamma\inv(+) |=| \beta\inv(+)|$, then we have $\gamma=\beta$.

\begin{prop}\label{mondaylazy}
	Let $\linearGraph\in\{\linearGraph_r,\linearGraph_r'\}$.
	If $\beta(\linearGraph^+)\subset \{\pm\}$, then $\functorWebB(\dotNotationText)$ annihilates $ \Omega_{\beta}$.
\end{prop}
\begin{proof}
	From the assumption we have $\Omega_{\beta}  = 
	(\basis_{\foldVertex{j_1}}\wedge\basis_{{j_1}})\wedge \cdots \wedge (\basis_{\foldVertex{j_s}}\wedge \basis_{\foldVertex{j_s}}) $. Then the conclusion is
	clear from \cref{officeAlone}.
\end{proof}

The following is the quantum analogue of \cref{intro-prop}.
\begin{thm}[Spectrum of the dot]
	\label{mainTheorem}
	The following hold:
	\begin{enumerate}
		\item An eigenbasis of $\bigwedge^k\vecrep$ is given by $ \{ \Omega_\beta \mid \beta \in {}_{k}\boxdiagramSet(\linearGraph) \}$, and we have
		\begin{align}
			 \functorWebB(\dotNotationText ) (\Omega_\beta)  = \Qint{m_\beta} \Omega_\beta
			,\text{ where } 
			m_\beta = |\beta\inv(+)| + |\beta\inv(0)| - |\beta\inv(-)|
			.
			\label{hausmeister}
		\end{align}
		\item If~ $\linearGraph=\linearGraph_r$, then the spectrum of~ $\functorWebB(\dotNotationText ) $ on $\bigwedge^k \vecrep$ is 
		\[
			\{\Qint{k-2i}\mid \max(0,k-r-1)\leq i\leq \min(k,r+1)\}, 
		\]
		where $\Qint{k-2i}$ has multiplicity $\binom{r+1}{k-i}\binom{r+1}{i}$. 
		\item If~ $\linearGraph=\linearGraph_r'$, then the spectrum of~ $\functorWebB(\dotNotationText ) $ on $\bigwedge^k \vecrep$ is 
		\[
			 \{\Qint{k-2i}\mid\max(0,k-r-1)\leq i\leq \min(k,r)\},
		\]
		where $\Qint{k-2i}$ has multiplicity $\binom{r+1}{k-i}\binom{r}{i}$.
	\end{enumerate}
	The $\webQpar=-1$ version of the above statements holds for $\webCategoryB_\C$. 
\end{thm}
\begin{proof}
	Let $\linearGraph\in\{\linearGraph_r,\linearGraph_r'\}$. The fact that  $\Omega_\beta$ is an eigenvector of $\functorWebB(\dotNotationText)$ follows easily via induction on $k$ using \cref{eigenvec-constr} and \cref{mondaylazy}, the base case being the eigenbasis $\{\basis_p\pm \basis_{\foldVertex{p}}\}_{p\in\linearGraph^+}$ of $ \vecrep$. Moreover, \cref{hausmeister} follows directly from \cref{eigenvec-constr} and the definition of $\Omega_\beta$.

	Recall from \cref{boxdiag-count}, that $ |{}_{k}\boxdiagramSet(\linearGraph)| = |\binom{\graphVertices_{\linearGraph}}{k}| = \dim \bigwedge^k\vecrep$. Thus, to show that $\{\Omega_\beta\mid \beta\in {}_k\boxdiagramSet(\linearGraph)\}$ is an eigenbasis, it suffices to show that they span $\bigwedge^k$. For this we can again use induction on $k$. The case $k=1$ is clear. Let $U_k = \operatorname{span}\{\Omega_\beta\mid \beta \in {}_{k}\boxdiagramSet(\linearGraph)\}$. Suppose $U_{k-1}= \bigwedge^{k-1}\vecrep$, then we just need to show that all $\Omega_\beta\wedge \basis_p\in U_k$ for $\beta \in {}_{k-1} \boxdiagramSet(\linearGraph)$. By \cref{def-omega-vecs-bothcases}, we have three cases: \[
		\begin{cases}
			\Omega_\beta \wedge \basis_p =0 , & \text{ if } p,\,\foldVertex{p}\in \supp(\beta) ,\\
			\Omega_\beta \wedge \basis_p = c\cdot \Omega_{\beta \cup p}
			\text{ for some } c\in \funcField 
			, & \text{ if } \foldVertex{p}\in \supp(\beta) ,\; p\notin \supp(\beta), \\
				\begin{cases}
					\Omega_{\beta \cup p} = \Omega \wedge (\basis_{|p|} + \webQpar^{m-k} \basis_{\foldVertex{|p|}} ) \\
				\Omega_{\beta \cup \foldVertex{p}} = \Omega \wedge (\basis_{|p|}-\webQpar^{-m-k} \basis_{\foldVertex{|p|}} )
				\end{cases},
			\text{ for some } m\in \Z
			, & \text{ if } p,\; \foldVertex{p}\notin \supp(\beta) .
		\end{cases}
	\]
	In each of the cases one easily verifies that $\Omega_\beta\wedge \basis_p\in U_k$.
	This concludes that $U_k = \bigwedge^k\vecrep$ and therefore the $\Omega_{\beta}$'s form an eigenbasis of $\bigwedge^k\vecrep$. By \cref{hausmeister}, the $\Qint{k-2i}$ eigenspace is spanned by $\Omega_\beta$ with $|\supp(\beta) \cap \linearGraph^+| =k-i$. The number of such box diagrams with $k$ symbols is exactly as stated. This concludes that $\{\Omega_\beta\mid \beta\in {}_k\boxdiagramSet(\linearGraph)\}$ form an eigenbasis of $\bigwedge^k\vecrep$.
\end{proof}

We also notice that the eigenbasis we construct is a weight basis for the action of $\hat{D}_i$'s.
\begin{lemma}\label{D-wts}
	Let $\beta\in {}_{k}\boxdiagramSet$. Then we have \begin{align}
			\label{Di-action-Omega}
			\dCoideal_p \Omega_\beta = \begin{cases}
				\Omega_\beta ,  & \text{ if } p\notin \supp(\beta ), \\
				q \Omega_\beta ,  & \text{ if } \beta(p) \in \{+,-\}, \\
				q^2 \Omega_\beta ,  & \text{ if } \beta(p) = \pm,
			\end{cases}
			\quad \forall\, p >0.
		\end{align}
\end{lemma}
\begin{proof}
	By definition, $ \monWedge_\beta $ is an eigenvector for all $\dCoideal_p $.
	Since $\basis_p$ and $\basis_{\foldVertex{p}}$ have the same $\kleinCSA$-weight, for any $\beta\in \boxdiagramSet(\linearGraph_r)$, the  $\dCoideal_p $-eigenvalue of $\monWedge_\beta $ is given by the same formulas as \cref{Di-action-Omega}. From \cref{def-omega-vecs-bothcases}, $\Omega_\beta$ is a linear combination of $\monWedge_\gamma$'s where each $\gamma$ is obtained from $\beta$ by flipping some signs in $\beta$. Hence, each summand of $\Omega_\beta$ is an eigenvector of the same eigenvalue, and we are done.
\end{proof}

\begin{remark}
	In \cite[Prop. 4.2]{Mueller-Schneider-homogen} all possible eigenvalues of the action of $\bo$ was determined. Taking $\alpha= 0,\beta=q\inv,\gamma=1$ there, these eigenvalues $\lambda_r$ become $[n],[n-2],\ldots ,[-n]$, which coincides with the eigenvalues in \cref{SURPRISE-SURPRISE}. In fact, this can be explained using the skew Howe duality in \cite[Thm. 10.28]{EhSt-nw-algebras-howe}, where the $\webCategoryB$ is identified with another action of $\coideal(\linearGraph_n)$ as morphisms in $\FundCoideal$. Under this identification, the action of $\dotNotationText$ is exactly given by $\bo\in \coideal(\linearGraph_n) $. In other words, the fact that $\bo$ and $\dotNotationText$ eigenvalues coincide is not a coincidence.
\end{remark}

\part{Applications} 

\section{Representation theory and diagrammatics}
We now obtain three applications of \cref{mainTheorem} about $\webCategoryB$ and the representation theory of $\coideal(\linearGraph)$ and $\heckeSpecialized{k}$.
We find the following result surprising because we could not see it purely diagrammatically. Instead, we deduce it from the interplay with representation theory.
\begin{thm}\label{SURPRISE-SURPRISE}
	The relation $\prod_{i=0}^{k} (\dotNotationText-\Qint{k-2i})=0$ holds already in $\webCategoryB$. 
\end{thm}
\begin{proof}
	Fix $k$ and denote $\mathfrak{d}= \dotNotationText$. Using the Schur--Weyl duality, we can
	take sufficiently large $r$ such that $\functorWebB: \End_{\webCategoryB}(1^k)\to \End_{\Uq'(\linearGraph)}(\vecrep^{\otimes k})$ is bijective. It then follows that $\functorWebB: \End_{\webCategoryB}(k)\to \End_{\Uq'(\linearGraph)}(\bigwedge^k\vecrep)$ is also bijective.
	By \cref{basis-typeB-Web}, $\dim \End_{\webCategoryB}(k) = k+1$, so $1,\mathfrak{d},\ldots, \mathfrak{d}^{k+1}$ must be linearly dependent, which is equivalent to a degree $k+1$ monic polynomial relation $f(\mathfrak{d})=0$. On the other hand, by \cref{mainTheorem}, we have $\prod_{i=0}^{k} (\functorWebB(\mathfrak{d})-\Qint{k-2i})=0$, which has degree $k+1$. Therefore, we must have $f(x) = \prod_{i=0}^{k} (x-\Qint{k-2i})$.
\end{proof}

By \cref{Schur-Weyl-quantum}, we have a multiplicity-free decomposition into simple $(\coideal(\linearGraph),\heckeSpecialized{k})$-bimodules
\begin{align}
	\vecrep^{\otimes k} = \bigoplus L_{\lambda,\mu} \boxtimes \specht_{\lambda,\mu}, \text{ where } (\lambda,\mu) \text{ runs through bipartitions of } k.
	\label{SW-bimod}
\end{align}
Here the $\specht_{\lambda,\mu}$ are the Specht modules which we recalled at the end of \cref{hecke algebra section}.
In particular, $\coideal(\linearGraph)$-submodules of the tensor powers are semisimple, see also \cref{Cor:ss}. For $i\in \N$, let $(1^i)$ be the one-column partition with $i$ boxes.
\begin{thm}[Decomposition theorem]\label{liver-coideal-thm}
	Let $\linearGraph\in\{\linearGraph_r,\linearGraph_r'\}$. The following hold:
	\begin{enumerate}
		\item The decomposition of $\bigwedge^k\vecrep $ into irreducible $\coideal(\linearGraph)$-modules is multiplicity-free. 
		\item In fact, $\bigwedge^k \vecrep =\bigoplus_{0\leq i\leq k} L_{(1^i),(1^{k-i})}$.
		\item The copy of $L_{(1^i),(1^{k-i})}\subset \bigwedge^k \vecrep$ is the $\Qint{k-2i}$-eigenspace of~ $\functorWebB(\dotNotationText)$.
	\end{enumerate}
\end{thm}
\begin{proof}
	By \cref{SURPRISE-SURPRISE}, $\End_{\webCategoryB}(k) $ is generated as a $\funcField $-algebra  by $\dotNotationText$. Since $\functorWebB$ is full, the endomorphism ring of $\bigwedge^k \vecrep$ as a $\Uq'(\linearGraph)$-module is generated by $\functorWebB(\dotNotationText)$. By \cref{SURPRISE-SURPRISE}, it is isomorphic to $\funcField[x]/\prod_{n\leq i\leq m} (x-\Qint{k-2i}) \cong  \prod_{n\leq i\leq m}\funcField[x]/ (x-\Qint{k-2i}) $ for some $0\leq n\leq m\leq k$, 
	which is in particular commutative. So the decomposition of $\bigwedge^k \vecrep$ into irreducible $\Uq'(\linearGraph)$-modules must be multiplicity-free. Each eigenspace of  $\functorWebB(\dotNotationText)$ is a semisimple $\Uq'(\linearGraph)$-submodule with endomorphism ring $\funcField[x]/ (x-\Qint{k-2i}) $, hence have to be irreducible.

	Next we show that as a $\heckeSpecialized{d}$-module, the simple summands of $(\bigwedge^k \vecrep )\splitNotation(k) $ are of the form $  \specht_{(1^i),(1^{k-i})} $. 
	Consider the embedding $(\bigwedge^k \vecrep) \splitNotation(k) \hookrightarrow\vecrep^{\otimes k}$. 
	For $0\leq i\leq k$ and $w\in \bigwedge^k\vecrep$ be an eigenvector of $\functorWebB(\dotNotationText)$ of eigenvalue $\Qint{k-2i}$.
	Denote $\symFuncElmtry_{k-1}( J_l) = \sum_j\prod_{l\ne j} J_l$.
	By \cref{symfunc-relation},  
	\begin{align}
		\Qint{k-2i} w \cdot  \splitNotation(k)\cdot  \prod J_l = \webQpar^{k-1} w \cdot \splitNotation(k) \cdot \symFuncElmtry_{k-1}( J_l) .
		\label{ek-ek-1-action}
	\end{align}
	On the other hand, we can write $w\splitNotation(k)=\sum w_{\lambda,\mu}$ where $w_{\lambda,\mu}\in \specht_{\lambda,\mu}$ and $(\lambda,\mu)$ runs through all bipartitions of $k$. Let $(S,T)$ be any standard bitableau of shape $\lambda,\mu$, then we have each $w_{\lambda,\mu}\prod J_l = \epsilon_l q^{-2\sum \cont(l)} $ where $\cont$ is the content of the box in the bitableau $(S,T)$, and $\epsilon_l=1$ if the box with $l$ is in $S$ and $-1$ otherwise. Consequently,
	\begin{align*}
		w_{\lambda,\mu}\symFuncElmtry_{k-1}( J_l) =\sum_j \prod_{l\ne j}\epsilon_l q^{-2\sum \cont(l)}q^{2\cont(j)} w_{\lambda,\mu}=\sum_{j\in \lambda \cup \mu \setminus \lambda\cap \mu} \epsilon_j q^{2\cont(j)} w_{\lambda,\mu} \prod J_l .
	\end{align*}
	Comparing with the $q$-degree in \cref{ek-ek-1-action}, we must have $\{\cont(j)\mid j\in \lambda \cup \mu \setminus \lambda\cap \mu\}=\{i,i+1,\ldots, k-1-i\}$, and all the $\epsilon_j$'s are equal to $1$ if $k-2i>0$ and to $-1$ otherwise. This means that $\lambda\subset \mu$ if $k>2i$ and $\mu\subset \lambda$ otherwise, and their difference is a one-column partition of length $k-i$.
	Since $\lambda\cap \mu$ has only $i$ boxes in total and has a box of content $i-1$, we get that $\lambda\cap \mu$ must be a one-column partition of length $i$. In conclusion, $ w \splitNotation(k) =w_{\lambda,\mu}$ where \[
		(\lambda,\mu)=\begin{cases}
			 (1^i,1^{k-i}) , & \text{ if } k>2i,\\
			  (1^{k-i},1^i), & \text{ if } k<2i.
		\end{cases}
	\]
	This shows that $\splitNotation({k})$ sends the $\Qint{k-2i}$-eigenspace of $\dotNotationText$ into $L_{1^i,1^{k-i}} \boxtimes \specht_{(1^i),(1^{k-i})} $, and the third assertion holds.
	Finally, in view of the spectrum of $\functorWebB(\dotNotationText)$ on $\bigwedge^k\vecrep$ in \cref{mainTheorem}, it follows that $(\bigwedge^k \vecrep) \splitNotation(k) = \bigoplus_{0\leq i\leq k} \specht_{(1^i),(1^{k-i})} ^{\oplus m_i}$ as $\heckeSpecialized{k}$-modules for some $m_i\in \N$. By the bimodule decomposition \cref{SW-bimod}, this shows that $ (\bigwedge^k \vecrep) \splitNotation(k)  \subset \bigoplus_{0\leq i\leq k} L_{1^i,1^{k-i}} \boxtimes \specht_{(1^i),(1^{k-i})} $. Combined with the first statement, we contend that $\bigwedge^k \vecrep =\oplus_{0\leq i\leq k} L_{(1^i),(1^{k-i})}$.
\end{proof}
Later we will determine the isomorphism class of each $L_{(1^i),(1^{k-i})}$ via a Verma module approach.

Before proceeding, we also obtain the following interesting result for Hecke algebras:
Let $\sgn=\funcField $ be the $1$-dimensional $\hecke[d]$-module where each $H_i$ acts by $\webQpar=-q$.
The anti-spherical Hecke module is the induced module $\sgn\otimes_{\hecke[k]}\heckeSpecialized{k}$. 
\begin{thm}\label{beef-liver-thm}
	We have a multiplicity-free decomposition $\sgn\otimes_{\hecke[k]}\heckeSpecialized{k}= \bigoplus_{0\leq i\leq k} \specht_{(1^i),(1^{k-i})}$.
\end{thm}
\begin{proof}
	Note that for some $v\in \bigwedge^k\vecrep$ we have $v\splitNotation(k) \heckeSpecialized{k} \cong \sgn \otimes_{\hecke[k]} \heckeSpecialized{k}$. It follows from \cref{liver-coideal-thm} that $ \sgn \otimes_{\hecke[k]} \heckeSpecialized{k} = \sum_l\specht_{(1^l),(1^{k-l})}^{\oplus m_i}$ for some $m_i\in \N$.
	By \cref{basis-typeB-Web} we have $\dim \sgn\otimes_{\hecke[k]}\heckeSpecialized{k} = |W_k/S_k| = 2^k$. Since $\dim \specht_{(1^i),(1^{k-i})}= \binom{k}{i}$, we have $\sum m_i\dim  \specht_{(1^i),(1^{k-i})}= \sum m_i\binom{k}{i} =2^k$. This implies each $m_i=1$, and we are done.
\end{proof}

\section{Irreducible $\Uq'(\linearGraph_r)$-modules in exterior powers}
Let $\linearGraph\in \{\linearGraph_r,\linearGraph_r'\}$. 
In this section we determine the highest weight vectors in $\bigwedge^k\vecrep$ and compute their weights with respect to the Cartan subalgebras from \cref{csa-coideal-both}. By \cref{fd-quotients} and \cref{verma-properties}, this determines the isomorphism classes of these irreducible submodules.

\newcommand{\zugBoxDiag}[1]{\partial_{#1}} 
\newcommand{\stat}{\mathsf{s}}
\newcommand{\untenB}[2]{{}^{\smash{#1}}_{\smash{#2}}\!\!\downarrow}
\newcommand{\obenB}[2]{{}^{\smash{#2}}_{\smash{#1}}\!\!\uparrow}

\subsection{Actions in terms of box diagrams}
We start from a combinatorial rule that describes the $\Uq'(\linearGraph)$-action on modules in $\FundCoideal$.
Denote 
\begin{align*}
	{}_{\underline{k}}\boxdiagramSet(\linearGraph) & \coloneqq {}_{k_1}\boxdiagramSet(\linearGraph)\times \cdots \times {}_{k_n}\boxdiagramSet(\linearGraph)
	\; \qquad \text{ for }  \underline{k} = (k_1,\ldots,k_n)\in \N^n
	,\\
	\monWedge_{\underline{\beta}} &\coloneqq \monWedge_{\beta_1}\otimes \cdots \otimes \monWedge_{\beta_n}
	\; \qquad  \text{ for } \underline{\beta} = (\beta_1,\ldots, \beta_n)\in {}_{\underline{k}}\boxdiagramSet(\linearGraph)
	.
\end{align*}
We visualise $\underline{\beta}$ as the horizontal juxtaposition of the box diagrams $\beta_i$, e.g. if $\linearGraph=\linearGraph_r'$, 
\begin{align}
	\text{Given } 
	\beta_1 = \ytabcenter{\bullet\\-\\{}\\ \pm}
	\,, \quad
	\beta_2 =  \ytabcenter{{}\\ \pm\\{}\\{}}
	\,, \quad
	\beta_3= \ytabcenter{\bullet\\{} \\+\\ -}
	\,,\quad
	\rightsquigarrow
	\quad
	\text{ we have }
	\underline{\beta} 
	= \ytabcenter{\bullet&{}&\bullet \\- &\pm& {}\\{}&{}&+ \\ \pm & {}&-}.
	\label{multi-col-box-diag-example}
\end{align}
Using notations from \cref{sec-linear-graph} and \cref{boxdiag-opeations}, define for any box diagram $\beta$ and $p\in \graphVertices_{\linearGraph}$\[
	\untenB{+}{p}\beta \coloneqq (\beta\setminus p )\cup p_{\leftarrow} 
	,\quad
	\untenB{-}{p}\beta \coloneqq (\beta\setminus p) \cup \foldVertex{p_{\leftarrow} }
	,\quad
	\obenB{+}{p}\beta  \coloneqq (\beta\setminus p) \cup p_{\rightarrow} 
	,\quad
	\obenB{-}{p}\beta \coloneqq (\beta\setminus p) \cup \foldVertex{p_{\rightarrow} }
	.
\]
If $p=0$, we make the convention that the $\bullet$ in the $0$-th box can move down to be either of $+,-$, leaving the $0$-th box empty. Conversely, both $+,-$ become a $\bullet$ when moved into the $0$-th box.
Note that by definition, if the box $p$ is filled with $\pm$, we can take away a $+$ or $-$ and leave the remaining symbol $-$ respectively $+$ still in the $p$-th box.  For example, 
\[
	\untenB{+}{0}\beta_1  = \ytabcenter{{}\\\pm\\{}\\ \pm}
	\,,\quad
	\untenB{+}{1}\beta_2 = \ytabcenter{\\ - \\ + \\{}}
	\,,\quad
	\obenB{+}{1} \beta_2 = \ytabcenter{\bullet \\- \\ \\{}}
	\,,\quad
	\obenB{-}{1}\beta_2 = \ytabcenter{\bullet \\+ \\ \\{}}
	\,,\quad
	\untenB{+}{2} \beta_3  = \ytabcenter{\bullet\\ - \\ {} \\\pm}
\]
and $\untenB{-}{0}\beta_1$ is not a valid box diagram.

\begin{definition}
	Let $\underline{k}\in \N^n$. 
	For $i\in \graphEdges_{\linearGraph}$ and $1\leq j\leq n$, define the following linear operators:\begin{align}
		\zugBoxDiag{i,j}^\pm \monWedge_{(\beta_1,\ldots, \beta_n)} &= \monWedge_{(\beta_1,\ldots,\beta_{j-1}, \beta_j',\beta_{j+1},\ldots, \beta_n)}, \text{ where }
		\beta_j' =\begin{cases}
			\untenB{\pm}{\source{-i}}\beta_j , & \text{ if } i<0,\\
			\obenB{\pm}{\target{i}}\beta_j , & \text{ if } i>0,
		\end{cases}
		\\
		\zugBoxDiag{0,j} \monWedge_{(\beta_1,\ldots, \beta_n)} &= \monWedge_{(\beta_1,\ldots,\beta_{j-1}, \flip{0}{\beta_j} ,\beta_{j+1},\ldots, \beta_n)}
		.
	\end{align}
\end{definition}
 
Recall the $\classicalFixPointSubalg$ generators from \cref{gens-g-fixed}. The following action of these generators on tensor products of classical wedges follow immediately from unwinding the definitions.
\begin{lemma}
	Let $\underline{k}\in \N^n$ and $\underline{\beta}\in {}_{\underline{k}}\boxdiagramSet(\linearGraph) $. 
	We have
	\begin{align*}
		\matrixElmtFont{b}_0 \monWedge_{\underline{\beta}} =  \sum_{1\leq j\leq n}  \zugBoxDiag{0,j} \monWedge_{\underline{\beta}}
		\qquad
		\matrixElmtFont{b}_{i} \monWedge_{\underline{\beta}} = 
		\sum_{1\leq j\leq n}(\zugBoxDiag{i,j}^++\zugBoxDiag{i,j}^-)\monWedge_{\underline{\beta}}
		\quad (i\ne 0)
		.
	\end{align*}
\end{lemma}
Diagrammatically, $\matrixElmtFont{b}_{i}\monWedge_{\underline{\beta}}$ is the sum over multi-column box diagram obtained from $\underline{\beta} $ by moving a symbol $\smiley$ up (if $i>0$) or down (if $i<0$) (allowing $\ytabcenter{~\\ \pm} \mapsto \ytabcenter{+\\-}$ or $\ytabcenter{-\\+}$), or flipping the symbol on the top box if $i=0$ (where we skip the boxes with $\pm$). This is illustrated by the following picture.
\begin{align}\label{multicol-box-diag}
	\begin{tikzpicture}[scale=0.5,font=\normalsize]
		\def\rows{6}
		\def\columns{5}
		\draw[fill=orange] (3,2) rectangle +(-3,1);
		\path[draw=none] (3,2) -- +(-3,1) node[midway] {$a$};
		\draw[fill=cyan] (3,2) rectangle +(-3,-1);
		\path[draw=none] (3,2) -- +(-3,-1) node[midway] {$b$};
		\draw[fill=pink] (3,2) rectangle (\rows,3);
		\path[draw=none] (3,2) -- (\rows,3) node[midway] {$c$};
		\draw[fill=green!70!gray] (3,2) rectangle (\rows,1);
		\path[draw=none] (3,2) -- (\rows,1) node[midway] {$d$};
		\draw (3,2) rectangle +(1,1);
		\path[draw=none] (3,2) -- +(1,1) node[midway] {$\smiley$};
		\draw[] (3,1) to +(1,0);
		\draw (0,0) rectangle +(\rows,\columns);
		\draw[] (3,0) -- +(0,\columns) 
						(4,0) -- +(0,\columns)
						(0,3) -- +( \rows,0)    
						(0,2) -- +( \rows,0);   	
		\draw[thick,->] (3.5,2.25) -- +(0,-1); 
	\end{tikzpicture}
	\quad\qquad\begin{tikzpicture}[scale=0.5,font=\normalsize]
		\def\rows{6}
		\def\columns{5}
		\draw[fill=orange] (4,2) rectangle (0,3);
		\path[draw=none] (4,2) -- (0,3) node[midway] {$a$};
		\draw[fill=cyan] (4,2) rectangle (0,1);
		\path[draw=none] (4,2) -- (0,1) node[midway] {$b$};
		\draw[fill=pink] (4,2) rectangle +(2,1);
		\path[draw=none] (4,2) -- +(2,1) node[midway] {$c$};
		\draw[fill=green!70!gray] (4,2) rectangle +(2,-1);
		\path[draw=none] (4,2) -- +(2,-1) node[midway] {$d$};
		\draw (0,0) rectangle +(\rows,\columns);
		\draw (3,1) rectangle +(1,1);
		\path[draw=none] (3,1) -- +(1,1) node[midway] {$\smiley$};
		\draw[] (3,0) -- +(0,\columns) 
						(4,0) -- +(0,\columns)
						(0,3) -- +( \rows,0)    
						(0,2) -- +( \rows,0);   
		\draw[<-,thick] (3.5,2.75) -- +(0,-1); 
	\end{tikzpicture}
	\quad\qquad
	\begin{tikzpicture}[scale=0.5,font=\normalsize]
		\def\rows{6}
		\def\columns{5}
		\draw[fill=pink] (4,\columns-1) rectangle +(2,1);
		\path[draw=none] (4,\columns-1) -- +(2,1) node[midway] {$c$};
		\draw[] (3,0) -- +(0,\columns) 
						(4,0) -- +(0,\columns)
				(0,\columns-1) -- +(\rows,0);
		\draw (0,0) rectangle +(\rows,\columns);
		\draw (3,\columns-1) rectangle +(1,1);
		\path[draw=none] (3,\columns-1) -- +(1,1) node[midway] {$\foldVertex{\smiley}$};
	\end{tikzpicture}
\end{align}
Next we upgrade this to a q version, where the colored regions in \cref{multicol-box-diag} come in. For example, we have (writing $\underline{\beta}$ instead of $\monWedge_{\underline{\beta}}$ for simplicity, and the linear graph $\linearGraph=\linearGraph_r'$) \[
	\newcommand{\red}[1]{\textcolor{red}{#1}}
	B_\half \,\ytabcenter{\bullet&{}&\bullet \\- &\red{\pm}& {}\\{}&{}&+ \\ \pm & {}&-} = 
	\,\ytabcenter{\bullet&\red{\bullet} &\bullet \\- &\red{+}& {}\\{}&{}&+ \\ \pm & {}&-} 
	+q\inv \,\ytabcenter{\bullet&\red{\bullet}&\bullet \\- &\red{-}& {}\\{}&{}&+ \\ \pm & {}&-} 
	,\qquad
	B_{-\half} \,\ytabcenter{\bullet&{}& \red{\bullet} \\- &\pm& {}\\{}&{}&+ \\ \pm & {}&-} = 
	q^{-3}\,\ytabcenter{{}&{}&\bullet \\\pm &\pm& {}\\{}&{}&+ \\ \pm & {}&-}
	+\,\ytabcenter{\bullet&{}&{} \\- &\pm& \red{+}\\{}&{}&+ \\ \pm & {}&-}
	+q^{-1}\,\ytabcenter{\bullet&{}&{} \\- &\pm& \red{-}\\{}&{}&+ \\ \pm & {}&-}
	.
\]
In general, we need the following statistics.
\begin{definition}
Let $\underline{k}\in \N^n$.
For $p\in \graphVertices_{\linearGraph}$, $1\leq j\leq n$ and $\underline{\beta}\in {}_{\underline{k}}\boxdiagramSet(\linearGraph) $, define 
\begin{align*}
	&\stat^+_{<j}(p,\underline{\beta}) = 
		| \{1\leq l<j\mid {p} \in\supp(\beta_l) \}|,
	\quad
	\stat^-_{<j}(p,\underline{\beta}) = 
			| \{1\leq l<j\mid \foldVertex{p} \in\supp(\beta_l)\}| ,
	\\
	&\stat^+_{> j}(p,\underline{\beta}) =  | \{
			j< l \leq n \mid 
		{p} \in\supp(\beta_l)
			\}|
	,\quad
	\stat^-_{> j}(p,\underline{\beta}) = 
			| \{j< l \leq n  \mid \foldVertex{p} \in\supp(\beta_l)\}|.
\end{align*}
\end{definition}
Diagrammatically, $\stat^+_{<j}$ respectively $\stat^-_{<j}$ counts the number of $+$ respectively $-$ in the $p$-th row to the left of the $j$-th column in the multi-column box diagram $\underline{\beta}$. Similarly, $\stat^+_{>j}$ respectively $\stat^-_{>j}$ counts the number of the corresponding symbols to the right of the $j$-th column. 
For example, let $\underline{\beta}$ be as in \cref{multi-col-box-diag-example}, then we have $\stat^+_{<2}(3,\underline{\beta})= 1, \stat^+_{> 2}(3,\underline{\beta})= 0, \stat^+_{> 1}(0,\underline{\beta})=\stat^-_{> 1}(0,\underline{\beta})= 1,$ and $ \stat^-_{<3}(1,\underline{\beta})= 2 $. 

Denote $\stat_\bullet=\stat^+_\bullet+\stat^-_\bullet$.
The following result is already contained in \cite{EhSt-nw-algebras-howe}. 
\begin{prop}\label{Bi-action-std-wedge}
	Let $\underline{k}\in \N^n$ and $\underline{\beta}\in {}_{\underline{k}}\boxdiagramSet(\linearGraph) $. 
	We have for $i>0$:
	\begin{align}
    & B_0 \monWedge_{\underline{\beta}} =  \sum_{1\leq j\leq n} q^{ \stat^+_{>j}(0,\underline{\beta})-\stat^-_{>j}(0,\underline{\beta})} \zugBoxDiag{0,j} \monWedge_{\underline{\beta}}
    \\
    & \begin{aligned}
        & B_i \monWedge_{\underline{\beta}} = 
        \sum_{1\leq j\leq n} 
		q^{a- \stat_{<j+1}(\target{i},\underline{\beta})+1+ c^--d^-}  
		\zugBoxDiag{i,j}^+\monWedge_{\underline{\beta}}
        +
        \sum_{1\leq j\leq n}
        q^{ c^--d^-  } 
         \zugBoxDiag{i,j}^-\monWedge_{\underline{\beta}},
        \\
        & \text{ where }  
        a = \stat_{<j+1}(\source{i},\underline{\beta}) +\delta_{i,\half},
        \quad 
        c^- = \stat_{>j}^-(\source{i},\underline{\beta}) 
        , \quad 
        d^- =\stat_{>j}^-(\target{i},\underline{\beta}) ,
    \end{aligned}
    \\
    & \begin{aligned}
        & B_{-i} \monWedge_{\underline{\beta}} = 
        \sum_{1\leq j\leq n} q^{b-a + d^{+} - \stat_{>j-1}^+(\source{i},\underline{\beta}) } \zugBoxDiag{-i,j}^-\monWedge_{\underline{\beta}}
        +
        \sum_{1\leq j\leq n}
		q^{ d^{+} - \stat_{>j}^+(\source{i},\underline{\beta})  }
		\zugBoxDiag{-i,j}^+\monWedge_{\underline{\beta}}
		,
        \\
        & \text{ where }  
        a =\stat_{<j}(\source{i},\underline{\beta}) +\delta_{i,\half}
        , \quad 
        b =\stat_{<j}(\target{i},\underline{\beta})
        , \quad 
        d^+ =\stat_{>j-1}^+(\target{i},\underline{\beta}) .
    \end{aligned}
\end{align}
\end{prop}
\begin{proof}
	We have $\Delta B_0^n =  \sum 1^{\otimes r} \otimes B_0 \otimes (K_0\inv )^{\otimes n-r}$ and for $i> 0$
	\begin{align}
		\label{comult-Bi}
			\Delta^nB_i &= \sum \left( ( K_{-i}K_i\inv)^{\otimes r}\otimes E_{-i}K_{i}\inv \otimes (K_{i}\inv )^{\otimes n-r}\right. 
			\left.   + 	1 ^{\otimes r} \otimes F_{i} \otimes (K_{i}\inv )^{\otimes n-r} \right)
		,\\
		\label{comult-B-i}
		\Delta^nB_{-i} &= \sum \left( ( K_iK_{-i}\inv)^{\otimes r}\otimes E_{i}K_{-i}\inv \otimes (K_{-i}\inv )^{\otimes n-r}\right. 
		\left.   + 	1 ^{\otimes r} \otimes F_{-i} \otimes (K_{-i}\inv )^{\otimes n-r} \right)
		\quad (i\ne \half)
		,\\
		\Delta^nB_{-\half} &= \sum q ( K_\half K_{-\half}\inv)^{\otimes r}\otimes E_{\half}K_{-\half}\inv \otimes (K_{-\half}\inv )^{\otimes n-r} 
		  + 	1 ^{\otimes r} \otimes F_{-\half} \otimes (K_{-\half}\inv )^{\otimes n-r} 
		.
	\end{align}
	Note that if $\underline{k}=(1,\ldots, 1)$, then $\monWedge_{\underline{\beta}}$ is the standard monomial basis of $\vecrep^{\otimes n}$. In this case the desire results follow immediately from the comultiplication formulas. For the general case, suppose $\supp (\beta_j) = \{ p_1^j> p_2^j> \ldots > p_{k_j}^j \}$, then we have 
	\begin{align*}
		B_i (\monWedge_{\underline{\beta}} ) &= \mergeNotation(\underline{k}) (\Delta^{|\underline{k}|-1}B_i (\basis_{p_1^1} \otimes \cdots \otimes  \basis_{p^n_{k_n}})).
	\end{align*}
	On the right-hand side, the $\underline{k}=(1,\ldots, 1)$ case gives the desired scalars for $(\Delta^{|\underline{k}|-1}B_i (\basis_{p_1^1} \otimes \cdots \otimes  \basis_{p^n_{k_n}}))$. Now
	note that \mergeNotation(\underline{k}) does not produce further scalars whenever $\partial_{i,j}^\bullet \monWedge_{\underline{\beta}}\ne0$, we are done.
\end{proof}

\begin{remark}
	The coefficients of $B_{-i},B_{i},(i>0) ,B_0$ are in fact the number of certain symbols in the colored regions of the three picture in \cref{multicol-box-diag} respectively, e.g. $a,b$ counts the number of $+$ plus the number of $-$ in the respective regions, where $\pm$ is viewed as one $+$ and one $-$ separately. 
\end{remark}

\subsection{Box diagrams for highest weight vectors}
Let $\linearGraph\in \{\linearGraph_r,\linearGraph_r'\}$. 
\newcommand{\negHWboxdiag}[1][k]{\beta^-_{#1}}
\begin{definition}\label{hw-boxdiag-def}
Let $\hwboxdiag(\linearGraph_r)\subset {}_{k}\boxdiagramSet(\linearGraph_r)$ be the subset of box diagrams with all symbols aligned to the top. 
Let $ \negHWboxdiag\in {}_{k}\boxdiagramSet(\linearGraph_r')$ be the box diagram defined by $\negHWboxdiag(1) = \cdots = \negHWboxdiag(k)=-$. 
Let $\hwboxdiag(\linearGraph_r')\subset {}_{k}\boxdiagramSet(\linearGraph_r')$ be the subset consisting of $\negHWboxdiag$ together with the box diagrams with all symbols aligned to the top.
\end{definition}
Note that all box diagrams $\beta \in \hwboxdiag(\linearGraph_r')$ except $\negHWboxdiag$ satisfies $\beta(0)=\bullet$.

Graphically, in the $\linearGraph_r$ case, $\hwboxdiag(\linearGraph_r)$ consists of the left two kinds of the following diagrams. For the $\linearGraph_r'$ case one just adds a bullet on top of them, with one exception given by leaving the top box empty and fill the next top $k$ boxes with $-$.
\begin{align}\label{Zaction-picture}
\begin{tikzpicture}[scale=0.4]
	\def\rows{7}
	\def\columns{1}
	\def\top{2}
	\def\middle{3}
	\def\bracedist{0.15}
	\draw[decorate,decoration={brace,amplitude=8pt}] (\columns+\bracedist,\rows) to node[right,xshift=10pt] {$l$} +(0,-\top);
	%
	%
	\draw[decorate,decoration={brace,amplitude=8pt}] (-\bracedist,\rows-\top-\middle) to node[left,xshift=-10pt] {$k-l$} (-\bracedist,\rows);
	\draw (0,0) rectangle +(\columns,\rows);
	\draw[fill=orange] (0,\rows) rectangle node[below]{$\pm$} +(\columns,-\top) ;
	\draw[fill=cyan] (0,\rows-\top) rectangle node[below]{$+$} +(\columns,-\middle);
	\node at (0.5,0.5) {$\emptyset$};
\end{tikzpicture}
\quad
\begin{tikzpicture}[scale=0.4]
	\def\rows{7}
	\def\columns{1}
	\def\top{3}
	\def\middle{2}
	\def\bracedist{0.15}
	\draw (0,0) rectangle +(\columns,\rows);
	\draw[decorate,decoration={brace,amplitude=8pt}] (\columns+\bracedist,\rows) to node[right,xshift=10pt] {$l$} +(0,-\top);
	%
	%
	\draw[decorate,decoration={brace,amplitude=8pt}] (-\bracedist,\rows-\top-\middle) to node[left,xshift=-10pt] {$k-l$} (-\bracedist,\rows);
	\draw[fill=orange] (0,\rows) rectangle node[below]{$\pm$} +(\columns,-\top) ;
	\draw[fill=cyan] (0,\rows-\top) rectangle node[below]{$-$} +(\columns,-\middle);
	\node at (0.5,0.5) {$\emptyset$};
\end{tikzpicture}
\quad
\begin{tikzpicture}[scale=0.4]
	\def\rows{7}
	\def\columns{1}
	\def\top{2}
	\def\middle{3}
	\def\bracedist{0.15}
	\draw[decorate,decoration={brace,amplitude=8pt}] (\columns+\bracedist,\rows-1) to node[right,xshift=10pt] {$l$} +(0,1-\top);
	%
	%
	\draw[decorate,decoration={brace,amplitude=8pt}] (-\bracedist,\rows-\top-\middle) to node[left,xshift=-10pt] {$k-l-1$} (-\bracedist,\rows-1);
	\draw (0,0) rectangle +(\columns,\rows);
	\draw[] (0,\rows) rectangle node[]{$\bullet$} +(\columns,-1) ;
	\draw[fill=orange] (0,\rows-1) rectangle node[above]{$\pm$} +(\columns,-\top) ;
	\draw[fill=cyan] (0,\rows-\top) rectangle node[below]{$+$} +(\columns,-\middle);
	\node at (0.5,0.5) {$\emptyset$};
\end{tikzpicture}
\quad
\begin{tikzpicture}[scale=0.4]
	\def\rows{7}
	\def\columns{1}
	\def\top{3}
	\def\middle{2}
	\def\bracedist{0.15}
	\draw (0,0) rectangle +(\columns,\rows);
	\draw[decorate,decoration={brace,amplitude=8pt}] (\columns+\bracedist,\rows-1) to node[right,xshift=10pt] {$l$} +(0,1-\top);
	%
	%
	\draw[decorate,decoration={brace,amplitude=8pt}] (-\bracedist,\rows-\top-\middle) to node[left,xshift=-10pt] {$k-l-1$} (-\bracedist,\rows-1);
	\draw[] (0,\rows) rectangle node[]{$\bullet$} +(\columns,-1) ;
	\draw[fill=orange] (0,\rows-1) rectangle node[above]{$\pm$} +(\columns,-\top) ;
	\draw[fill=cyan] (0,\rows-\top) rectangle node[below]{$-$} +(\columns,-\middle);
	\node at (0.5,0.5) {$\emptyset$};
\end{tikzpicture}
\quad
\negHWboxdiag[4]=
\ytabcenter{~\\-\\-\\-\\-\\~\\~}
\end{align}

Recall the decomposition from \cref{liver-coideal-thm}. 
\begin{prop}\label{Omega-is-weigh-vec}
	Let $\linearGraph\in \{\linearGraph_r,\linearGraph_r'\}$.
	Let $\beta\in \hwboxdiag(\linearGraph)$ with $|\supp(\beta)\cap \linearGraph^+|=i$. 
	Then $\Omega_{\beta} $ is a highest weight vector in the irreducible submodule $ L_{(1^i),(1^{k-i})}$.
\end{prop}
\begin{proof}
	By \cref{D-wts}, $\Omega_\beta$ is a $\kleinCSA$-weight vector.
	Since $\Omega_\beta$ form an eigenbasis of $\bigwedge^k\vecrep$, we may write $
		Z_l \Omega_{\beta} = \sum a_\gamma \Omega_\gamma,
	$
	where $\gamma$ runs through ${}_{k}\boxdiagramSet$. Take any $h \in \csaCoideal$ By \cref{cartan-commute-both}, $h$ commutes with $\kleinCSA$. This means that $a_\gamma\ne 0 $ only if~ $\Omega_\gamma$ has the same $\kleinCSA$-weight as $\Omega_{\beta}$. In view of \cref{Di-action-Omega}, this means the sum $h \Omega_{\beta} = \sum a_\gamma \Omega_\gamma$ is over $\gamma\in {}_k\boxdiagramSet^{\mathrm{hw}}$.
	Now, by \cref{mainTheorem}, the $\Qint{k-2i}$-eigenspace of $\functorWebB(\dotNotationText)$ inside $\bigwedge^k\vecrep$ is spanned by $\Omega_\gamma$ with $ |\gamma\inv(\pm)|+|\gamma\inv(-)|=i$. 
	On the other hand, by looking at the $\dotNotationText$ eigenvalue, the sum $h \Omega_{\beta} = \sum a_\gamma \Omega_\gamma$ is over $\gamma\in \in {}_k\boxdiagramSet^{\mathrm{hw}}$ having the same number of $+$ and $-$'s as $\beta$. Note that this implies $\gamma=\beta$, hence $\Omega_{\beta}$ is a weight vector of $h$.
	This shows that $\Omega_{\beta} $ is a weight vector for $\csaCoideal$.
	Moreover, by \cref{liver-coideal-thm}, the copy of $L_{(1^i),(1^{k-i})}$ is the $\Qint{k-2i}$-eigenspace of $\functorWebB(\dotNotationText)$ inside $\bigwedge^k\vecrep$. Since $\Omega_{\beta}$ lies in the same eigenspace, we have $\Omega_{\beta}\in L_{(1^i),(1^{k-i})}$. Since the $\{\Omega_\gamma\}_{\gamma\in {}_k\boxdiagramSet}$ form an eigenbasis, it follows from \cref{Di-action-Omega} that $\Omega_\beta$ has the unique maximal $\kleinCSA$-weight inside $L_{(1^i),(1^{k-i})}$.
	Finally, this also implies that $\posCoideal \Omega_\beta=0$. Fix $i>0$. Since $B_i \Omega_\beta\in L_n$ would have strictly larger $\kleinCSA$-weight as $\Omega_\beta$, we must have $B_i\Omega_\beta=0$. Then by \cref{check-hwvec}, $\Omega_\beta$ is a highest weight vector.
\end{proof}
	
We next compute explicitly the $\csaCoideal$-weights of these highest weight vectors.
\subsection{The $\linearGraph_r$ case}
We first compute the $Z_i$ action on some standard monomial wedges $\monWedge_\beta$. Depending on the relative position of $i$ and $\beta$, there are three cases which we deal separately in \cref{Zaction1}---\cref{Zi-act-tailend}. 
Using notations from \cref{sec-linear-graph}, we consider first the case $\source{i}$ is in the orange region and $\target{i}$ is in the blue region in \cref{Zaction-picture}, then
the $Z_i$ action on $\monWedge_\beta$ is given by flipping the symbol in the $\target{i}$ box:
\begin{prop}\label{Zaction1}
	Let $\beta\in \hwboxdiag(\linearGraph_r)$ with $|\beta\inv(\pm)|=l$.
	\begin{enumerate}
		\item For $0\leq i \leq l-1$ we have $Z_i\monWedge_{\beta} =0 $.
		\item If $k-2l\geq 1$ then $Z_l \monWedge_{\beta} = \monWedge_{\flip{\target{l}}{\beta}}$.
		\item If $k=2l$ then $Z_i\monWedge_{\beta} =0 $ for all $i$.
	\end{enumerate}
\end{prop}
\begin{proof}
	For the first statement, we have $B_{\pm i}\monWedge_\beta = B_0 \monWedge_\beta =0$ for $1\leq i\leq l-1$, because there is no room to move the symbols either up or down in the first $l-1$ boxes. 
	We prove the second statement by induction on $l$. The base case $l=0$ 
	is just the $B_0$ action on $\monWedge_\beta$, which is given by \cref{Bi-action-std-wedge}. Now for $l>1$, suppose it holds for $l-1$, then
	we prove the case where we have $+$'s in the blue region in \cref{Zaction-picture}, the other case being similar.
	We have \[ 
		Z_l \monWedge_{\beta} = [B_{l}, Y_l]_{q\inv } \monWedge_{\beta} = B_{l}Z_{l-1} B_{-l} \monWedge_{\beta} =  B_{l}Z_{l-1} \monWedge_{\beta\setminus\foldVertex{\source{l}}\cup\foldVertex{\target{l}}} ,
		\]
	 because $B_{l}\monWedge_{\beta} =0$. Now by induction hypothesis, we have $Z_{l-1} \monWedge_{\beta\setminus\foldVertex{\source{l}}\cup\foldVertex{\target{l}}} = \monWedge_{\beta\setminus\foldVertex{\source{l}}\cup\foldVertex{\target{l}}} $, 
	hence $Z_l \monWedge_{\beta} = B_{l} \monWedge_{\beta^{-{\source{l}}\cup\foldVertex{\target{l}}}} = \monWedge_{\flip{\target{l}}{\beta}}$, as asserted. 
	
	The $i>l$ part of the last statement is trivial because there are no symbols to move in the boxes in the white region in \cref{Zaction-picture}. If $i= l$,
	denote $j=\target{l}$ and $p=\source{i}$, we have 
	\[
		Z_i \monWedge_\beta = B_{i} Z_{i-1} B_{-i} \monWedge_\beta = B_{i} Z_{i-1} (q\inv \monWedge_{\beta\setminus{\foldVertex{p}\cup\foldVertex{j}}} + \monWedge_{\beta \setminus p \cup{j}} ) =
		B_{-i} (q\inv \monWedge_{\beta{\setminus{i}\cup\foldVertex{j}}} + \monWedge_{\beta \setminus{\foldVertex{p}\cup j}} ) =0,
	\]
	because $B_{-i} $ would place two of the same symbols in the same box.
\end{proof}

Next we consider the case where $\source{i},\target{i}$ are in the blue region in \cref{Zaction-picture}.
\begin{prop}\label{Zaction2}
	Suppose $k-2l\geq 2$. 
	Let $\beta\in \hwboxdiag(\linearGraph_r)$ with $|\beta\inv(\pm)|=l$. If $+\in \img \beta$, then for $l <  i < k-l$,
	\begin{align}\label{midnight-formula}
		Z_i \monWedge_{\beta} = (-q)^{l-i+1} (
			-q\inv \monWedge_{\flip{\target{i}}\beta}
			- \monWedge_{\flip{\source{i}}\beta}
		),
		\qquad
		Z_i \monWedge_{\flip{\target{i}}\beta} = q\inv Z_i \monWedge_{\flip{\source{i}}\beta}	
		.
	\end{align}
	Similarly, if $-\in \img \beta$, then we have for $l <  i < k-l$,
	\begin{align}\label{midnight-formula2}
		Z_i \monWedge_{\beta} = (-q)^{l-i+1} (
			-q^{-2} \monWedge_{\flip{\target{i}}\beta}
			-q\inv \monWedge_{\flip{\source{i}}\beta}
		).
	\end{align}
\end{prop}
\begin{proof}
	We prove the first identity of \cref{midnight-formula} by induction on $i$. For the base case $i=l+1$, note that $\source{i}$ is the top box in the blue region in \cref{Zaction-picture}. We compute using \cref{Zaction1} \[
	Z_{i} \monWedge_{\beta} = -q\inv B_{-i}  B_{i} Z_{i-1}\monWedge_{\beta} 
	= -q\inv B_{-i}  B_{i}  \monWedge_{\flip{\source{i}}\beta} 
	= -q\inv B_{-i}  \monWedge_{\beta\setminus{\source{i}\cup\foldVertex{\target{i}}}} =-q\inv \monWedge_{\flip{\target{i}}\beta} - \monWedge_{\flip{\source{i}}\beta}.
	\]
	Now for $i>l+1$, we have \begin{align*}
		Z_i \monWedge_{\beta} &=
		 -q\inv B_{-i} B_i Z_{i-1} \monWedge_{\beta} 
		=  -q\inv  B_{-i} B_i (-q)^{l-i+2} (
			-q\inv \monWedge_{\flip{\target{i-1}}\beta}
			- \monWedge_{\flip{\source{i-1}}\beta})
		\\
		&=  -B_{-i} (-q)^{l-i+1} \monWedge_{\beta\setminus{\source{i}\cup\foldVertex{\target{i}}}}
		= (-q)^{l-i+1} (
			-q\inv \monWedge_{\flip{\target{i}}\beta}
			- \monWedge_{\flip{\source{i}}\beta}
		).
	\end{align*}
	To see the second identity of \cref{midnight-formula}, we again do induction on $i$, the base case $i=l+1$ being \begin{align*}
		Z_i \monWedge_{\flip{\target{i}}\beta} &= 
		(-q\inv B_{-i} B_i Z_{i-1} -q\inv Z_{i-1}B_iB_{-i}) \monWedge_{\flip{\target{i}}\beta} 
		\\
		&= -q\inv B_{-i} B_i \monWedge_{\flip{\source{i}}{(\flip{\target{i}}{\beta})}} -q\inv Z_{i-1}B_i \monWedge_{\beta\setminus{\target{i}\cup\foldVertex{\source{i}}}} 
		= -q\inv Z_{i-1}B_i \monWedge_{\beta\setminus{\target{i}\cup\foldVertex{\source{i}}}} 
		\\
		Z_i \monWedge_{\flip{\source{i}}\beta} &= (-q\inv B_{-i} B_i Z_{i-1} -q\inv Z_{i-1}B_iB_{-i})  \monWedge_{\flip{\source{i}}\beta} 
		\\
		& = -q\inv B_{-i} B_i \monWedge_{\beta} - Z_{i-1}B_i \monWedge_{\beta\setminus{\target{i}\cup\foldVertex{\source{i}}}} 
		= - Z_{i-1}B_i \monWedge_{\beta\setminus{\target{i}\cup\foldVertex{\source{i}}}} 
		.
	\end{align*}
	Here we used the second statement of \cref{Zaction1} in the second equalities.
	Now for $i>l+1$ \begin{align*}
		Z_i \monWedge_{\flip{\target{i}}\beta} &=
		-q\inv B_{-i} B_i Z_{i-1} \monWedge_{\flip{\target{i}}\beta} 
		= -q\inv B_{-i} B_i q\inv Z_{i-1} \monWedge_{\flip{\source{i}}\beta} 
		= q\inv Z_i \monWedge_{\flip{\source{i}}\beta}.
	\end{align*}
	The calculation for \cref{midnight-formula2} is similar, and we omit it.
\end{proof}

Next we consider the case where $\source{i}$ is in the orange region and $\target{i}$ is in the white region in \cref{Zaction-picture}. 
\begin{prop}
	\label{Zi-act-tailend}
	Let $\beta\in \hwboxdiag(\linearGraph_r)$ with $|\beta\inv(\pm)|=l$. If $+\in \img \beta$, then \begin{align}
		Z_{k-l} \monWedge_{\beta} = (-q)^{2l-k} \monWedge_{\flip{\source{k-l}}{\beta} }.
	\end{align}
	If $k-2l=1$ and $\source{k-l}=-$, then we have \begin{align}
		Z_{k-l} \monWedge_{\beta} = -q\inv \monWedge_{\flip{\source{k-l}}{\beta}}.
	\end{align}
\end{prop}
\begin{proof}
	Note that $k-2l$ is the number of boxes in the blue region in \cref{Zaction-picture}.
	We prove the first identity by induction on $k-2l$. For the base case $k-2l=1$, we write $i=k-l$ and use \cref{Zaction1}, \cref{Zaction2} to compute \begin{align*}
		Z_{i} \monWedge_{\beta}& =
		 B_{-i} Z_{i-1} \monWedge_{\beta\setminus{\source{i}\cup \target{i}}} 
		-q\inv B_{-i} B_{i}  \monWedge_{\beta}
		=
		0
		-q\inv B_{-i} B_{i} \monWedge_{\flip{\source{i}}\beta}
		= -q\inv \monWedge_{\flip{\source{i}}\beta}
		.
	\end{align*}
	The second identity follows similarly, and is left to the reader.
	Now assume the first identity holds for $k-2l-1$, we prove the case $k-2l$. Again write $i=k-l$ and compute
	\begin{align*}
		Z_{i} \monWedge_{\beta}& =
		 B_{-i} Z_{i-1} \monWedge_{\beta\setminus{\source{i}\cup \target{i}}} 
		-q\inv B_{-i} B_{i} Z_{i-1} \monWedge_{\beta}
		\\
		&=
		(-q)^{l-i+1} \monWedge_{\flip{\source{i}}\beta}
		-q\inv B_{-i} B_{i}(-q)^{l-i+2} (-q\inv \monWedge_{\flip{\source{i}}\beta} - \monWedge_{\flip{\source{i-1}}\beta}) 
		\\
		&= (-q)^{l-i+1} \monWedge_{\flip{\source{i}}\beta}
		-q\inv B_{-i} B_{i}(-q)^{l-i+2} (-q\inv \monWedge_{\flip{\source{i}}\beta} - \monWedge_{\flip{\source{i-1}}\beta}) 
		= (-q)^{l-i+1} \monWedge_{\flip{\source{i}}\beta}
		.
	\end{align*}
	This finishes the proof.
\end{proof}

\subsection{The $\linearGraph_r'$ case}
Now we compute the $W_i$ action on some standard wedges. Note that the top-aligned condition forces $0\in\supp(\beta) $ for any $\beta \in \hwboxdiag(\linearGraph_r') $ except the $\negHWboxdiag$ in \cref{hw-boxdiag-def}.  

\begin{prop}
	We have $W_{i+\half} \Omega_{\negHWboxdiag}=0 $ for $i\ne k$ and $W_{k+\half} \Omega_{\negHWboxdiag}= (q^{2k-1}-q\inv) \Omega_{\negHWboxdiag}$.
\end{prop}
\begin{proof}
	We prove by induction on $i$, For the base case $i=\half$, note that 
	by definition \cref{OmegaDef}, we have $\Omega_{\negHWboxdiag} = (\basis_{1}-\basis_{\negidx{1}})\wedge w$ for some $w$ with $B_{-\half}w=0$. By \cref{Omega-is-weigh-vec}, we have $B_\half\Omega_{\negHWboxdiag}=0$. By \cref{Bi-action-std-wedge}, we have $B_{-\half}\Omega_{\negHWboxdiag}=0$. It follows that \begin{align*}
		W_\half \Omega_{\negHWboxdiag} & = (B_{-\half} B_\half-{qB_\half B_{-\half}}) \Omega_{\negHWboxdiag} = 0.
	\end{align*}
	Now assume $i>\half$. Denote $W=W_{\half+i-1}$, $X= B_{\half+i}, Y= B_{-\half -i } $.
	As another base case, we need also to compute for $i=1$ \begin{align*}
		W_{\half+1} \Omega_{\ytabcenter{~\\-\\~}} 
		& = W_{\half+1} (\basis_{1}-\basis_{\negidx{1}})  \\
		& = (\cancel{ Y   W X }
		\cancel{-q Y  X W } 
		{- q W XY }
		+ q^2  XW Y 
		+ {\kCoideal_\half XY}
		)
		(\basis_{1}-\basis_{\negidx{1}})
	\end{align*}
	We have $XY (\basis_{1}-\basis_{\negidx{1}})=(\basis_{1}-\basis_{\negidx{1}})$, and $XY (\basis_{1}-\basis_{\negidx{1}})= W (\basis_{2}-\basis_{\negidx{2}})=0$, hence 
	\begin{align*}
		W_{\half+1} \Omega_{\ytabcenter{~\\-\\~}} 
		= (\cancel{-q W}+ (q^2-1) \kCoideal_\half )(\basis_{1}-\basis_{\negidx{1}}) 
		= (q-q\inv) (\basis_{1}-\basis_{\negidx{1}}) .
	\end{align*}
	Now we do the induction step.
	By \cref{Omega-is-weigh-vec} we have $X \Omega_{\negHWboxdiag}=0$.
	By definition \cref{Wi-def}, writing $h= (q^2-1) \dCoideal_0\dCoideal_{i+1}\inv$, we have \begin{align}
		W_{i+\half} \Omega_{\negHWboxdiag} &= (\cancel{ Y   W X }
		\cancel{-q Y  X W } 
		{- q W XY }
		+ q^2  XW Y 
		+ {XY} h
		)\Omega_{\negHWboxdiag}
		,
		\label{W-induction}
	\end{align}
	Denote $\monWedge(\ytabcenter{-\\+})=\monWedge_{\beta}$ where $\beta$ is the box diagram defined by $\beta(i)= -$ and $\beta(i+1) = +$, and $\beta(i)$ not defined for any other $i$. Define also $\monWedge(\ytabcenter{?\\??}) $ similarly for $?,??\in \{+,-,\pm\}$.
	Then we have \[
		\Omega_{\negHWboxdiag} = \begin{cases}
			\Omega_{\negHWboxdiag[i-2]} \wedge (-q \monWedge(\ytabcenter{-\\-}) -\monWedge(\ytabcenter{-\\+}) +q \monWedge(\ytabcenter{+\\-}) + \monWedge(\ytabcenter{+\\+})) \wedge w'
			, & \text{ if } i<k ,\\
			\Omega_{\negHWboxdiag[i-2]} \wedge (\monWedge(\ytabcenter{+\\~}) -\monWedge(\ytabcenter{-\\~})), & \text{ if } i=k,
		\end{cases}
	\]
	for some $w'\in \bigwedge^{\bullet}\vecrep$. 
	By \cref{Bi-action-std-wedge}, we also have 
	\begin{align*}
		&XY  \monWedge(\ytabcenter{+\\-}) = q\inv  \monWedge(\ytabcenter{+\\-}) + \monWedge(\ytabcenter{-\\+}) 
		,&  
		XY \monWedge(\ytabcenter{-\\+})  =   \monWedge(\ytabcenter{+\\-}) + q\monWedge(\ytabcenter{-\\+}) 
		,\\  
		&XY \monWedge(\ytabcenter{-\\ ~ }) = \monWedge(\ytabcenter{-\\ ~ })
		,&  
		XY \monWedge(\ytabcenter{+\\ ~ }) = \monWedge(\ytabcenter{+\\ ~ })
		.
	\end{align*}
	If $i<k $, we have $WY \Omega_{\negHWboxdiag}=  W \Omega_{\negHWboxdiag[i-2]}\wedge w''=0$, and
	 \begin{align*}
		XY\Omega_{\negHWboxdiag} = 
			\Omega_{\negHWboxdiag[i-2]} \wedge (-1+q\cdot q\inv) (\monWedge(\ytabcenter{+\\-})+q \monWedge(\ytabcenter{-\\+}) ) \wedge w' =0 .
	\end{align*}
	This shows that $W_{i+\half} \Omega_{\negHWboxdiag}=0$. If $i=k$, then we have 
	$XY \Omega_{\negHWboxdiag}  = \Omega_{\negHWboxdiag} $, hence
	\begin{align*}
		-q W XY \Omega_{\negHWboxdiag} &= -q W \Omega_{\negHWboxdiag} =0 
		,\\
		q^2 XWY 
		&= q^2 XW \Omega_{\negHWboxdiag[i-2]} \wedge (\monWedge(\ytabcenter{~\\+}) -\monWedge(\ytabcenter{~\\-})) 
		= q^2 X (q^{2i-3}-q\inv )\Omega_{\negHWboxdiag[i-2]} \wedge (\monWedge(\ytabcenter{~\\+}) -\monWedge(\ytabcenter{~\\-}))
		\\
		&= (q^{2i-1} -q) \Omega_{\negHWboxdiag} 
		,\\
		XYh \Omega_{\negHWboxdiag} &=  (q^2-1 ) q\inv \Omega_{\negHWboxdiag}.
	\end{align*}
	This concludes that $ (q^2 XWY + hXY )\Omega_{\negHWboxdiag}  = (q^{2i-1}-q\inv) \Omega_{\negHWboxdiag}$. The case $i>k$ is clear.
\end{proof}

Next we only need to consider $\beta\in \hwboxdiag$ with $\beta(0)=\bullet$.

\begin{lem}\label{W-half-action}
	Let $\beta \in \boxdiagramSet(\linearGraph_r') $ with $\beta(0) = \bullet $ and $ 1\notin \supp(\beta)$. Let $\beta_+,\beta_-,\beta_\pm$ be the box diagrams obtained from $\beta$ by adding $ \beta_+(1)=+,  \beta_- (1)=-, \beta_\pm (1)=\pm 
	.
	$
	Then we have \begin{align*}
		W_\half \Omega_{\beta_+} = 0
		,\qquad
		W_\half \Omega_{\beta_-} = (-q^2-1)\Omega_{\beta_-}
		,\qquad
		W_\half \monWedge_{\beta_\pm} = 0
		,\qquad
		W_\half \monWedge_{\beta }= -2q \monWedge_{\beta}
		,\\
		W_{1+\half} \Omega_{\beta_+} =  (-q^3-q)\Omega_{\beta_+} 
		,\qquad
		W_{1+\half} \Omega_{\beta_-} = 0
		,\qquad
		W_{1+\half}\monWedge_{\beta_\pm} = (-q^3-q)\monWedge_{\beta_\pm} 
		,\qquad
		W_{1+\half} \monWedge_{\beta }= 0
		.
	\end{align*}
\end{lem}
\begin{proof}
	The last two equations is a direct computation using \cref{Bi-action-std-wedge} which we leave to the reader. For the first two one directly checks that \begin{align}
		\label{W0-action-monwedge}
		W_\half \monWedge_{\beta_+} = -q^2 \monWedge_{\beta_+} -q \monWedge_{\beta_-}
		,\qquad
		W_\half \monWedge_{\beta_-} = -q \monWedge_{\beta_+} - \monWedge_{\beta_-},
	\end{align}
	Then the desired equations are easily checked via direct calculation.
\end{proof}

\begin{prop}\label{Waction1}
	Let $\beta\in \hwboxdiag(\linearGraph_r')$ with $|\beta\inv(\pm)|=l$ and $\beta(0) = \bullet$.
	\begin{enumerate}
		\item For $0\leq i \leq l-1$ we have $W_{i+\half}\Omega_{\beta} =0 $.
		\item If $k-2l- 1>0$, then \[
			W_{l+\half} \Omega_{\beta} = \begin{cases}
				(q^{2l} -1)\Omega_{\beta}
				,&\text{ if } \beta({l+1}) = +,\\
				(-q^{2l-2}-1)\Omega_{\beta}
				, &\text{ if } \beta({l+1}) = - .
			\end{cases}
		\]
		\item If $k=2l+1$ then $W_{i+\half}\Omega_\beta =0 $ for all $i$.
	\end{enumerate}
\end{prop}
\begin{proof}
	We prove the first assertion  by induction on $i$, the base case being the third equality in \cref{W-half-action}. Assume $i>0$, denote $W=W_{i-1+\half}$, $X= B_{i+\half}, Y= B_{-i-\half} $. Noted that $[X,Y]=0$ and $X\Omega_{\beta }=0$ by \cref{Omega-is-weigh-vec}. 
	By definition \cref{Wi-def}, we have \begin{align*}
		W_{i+\half} \Omega_{\beta } &= \left(\cancel{ Y   W X }
		\cancel{-q Y  X W } 
		{- q W XY }
		+ q^2  XW Y 
		+ (q^2-1){XY}\dCoideal_0\dCoideal_{i}\inv
		\right)\Omega_{\beta }
		.
	\end{align*}
	Now since $i<l$, by definition \cref{OmegaDef} we have $\Omega_{\beta } = \Omega_{\beta'} \wedge \basis_{\negidx{i}}\wedge \basis_{i} \wedge \basis_{\negidx{i+1}}\wedge \basis_{i+1} \wedge w'$ for some $w,w' \in \bigwedge^\bullet \vecrep$. This means that 
	$Y \Omega_{\beta }=0 $ and we are done.

	For the second part, we use induction on $l$, the base case being \cref{W-half-action}. Assume now $l>0$ and the assertion holds for $l-1$. 
	Now denote $W=W_{l-1+\half}$, $X= B_{l+\half}, Y= B_{-l-\half} $. Noted that $X\Omega_{\beta }=0$ by \cref{Omega-is-weigh-vec}. 
	Similarly to the previous case, we have \[
		W_{l+\half} \monWedge_{\beta } =
		(q^2  XW Y - qWXY +(q^2-1) XY \dCoideal_0\dCoideal_{l}\inv )\monWedge_{\beta }
		.
	\]
	Now $k-2l-1>0$ implies that $l+1\in \supp(\beta)$. Let $ \monWedge(\ytabcenter{-\\\pm} )=\monWedge_\gamma $ where the box diagram $\gamma \in{}_2\boxdiagramSet$ is defined by $\gamma(l)=-,\gamma(l+1)=\pm $. Similarly, we also define
	$\monWedge(\ytabcenter{ ?\\ ?? } )$ for all $?,??\in \{+,-,\pm\}$.
	Then \[
		\Omega_\beta =\begin{cases}
			\Omega_{\beta'} \wedge (q^{2l+2}\monWedge(\ytabcenter{\pm\\+} ) - q^2  \monWedge(\ytabcenter{\pm \\ -}))
			,& \text{ if } \beta(l+1) = - ,\\
			\Omega_{\beta'} \wedge (q^{2l} \monWedge(\ytabcenter{\pm\\+} )+ q^2 \monWedge(\ytabcenter{\pm \\ -}))
			,& \text{ if } \beta(l+1) = + .
		\end{cases}
	\] 
	Here the $q^2 $ factor arises from writing $\monWedge(\ytabcenter{\pm\\~})\wedge \monWedge(\ytabcenter{~\\ -}) $ into standard wedge.
	Furthermore, \cref{Bi-action-std-wedge} gives \[
		Y \monWedge(\ytabcenter{\pm \\ + }) = \monWedge(\ytabcenter{+\\ \pm  })
		,\qquad 
		Y \monWedge(\ytabcenter{\pm \\ - }) =  \monWedge(\ytabcenter{-\\\pm })
		,\qquad 
		X  \monWedge(\ytabcenter{+\\\pm }) = \monWedge(\ytabcenter{\pm \\ + })
		,\qquad 
		X \monWedge(\ytabcenter{-\\\pm }) =   \monWedge(\ytabcenter{\pm \\ - })
		.
	\]
	By induction hypothesis the right-hand side is an eigenvector for $W$ of eigenvalue $\lambda_{l-1}$. Then \begin{align}
		(q^2-1) XY \dCoideal_0\dCoideal_{l}\inv \Omega_{\beta } 
		&= (q^{2}-1)  \Omega_{\beta } 
		\label{hilb}
		,\\
		q^2  XW Y \Omega_{\beta }
		&= \begin{cases}
			q^2  XW \left( 
				\Omega_{\beta'} \wedge (\monWedge(\ytabcenter{+\\\pm} ) - q^{2l+2} \monWedge(\ytabcenter{- \\ \pm}))
			\right)
			,& \text{ if } \beta(l+1) = - ,\\
			q^2  XW \left( 
				\Omega_{\beta'} \wedge( \monWedge(\ytabcenter{+\\\pm} )+ q^{2l} \monWedge(\ytabcenter{- \\ \pm}))
			\right)
			,& \text{ if } \beta(l+1) = + .
		\end{cases} \nonumber
	\end{align}
	Moreover, in either case we have $-q WXY \Omega_{\beta } = -q W\Omega_{\beta }=0 $.
	By induction hypothesis, we have \[
		q^2XWY \Omega_\beta = \begin{cases}
			q^2 (q^{2l-2}-1) \Omega_\beta , & \text{ if } \beta(l+1)=+,\\
			q^2 (-q^{2l}-1) \Omega_\beta, & \text{ if } \beta(l+1)=-.
		\end{cases}
	\]
	Combined with \cref{hilb}, this proves the second assertion.

	Finally, $k=2l+1$ implies $l+1\notin \supp(\beta)$ and furthermore, $\Omega_\beta$ is a scalar multiple of $\monWedge_\beta$. It suffices to show that $W_{l+\half}\monWedge_\beta=0$.
	We first show that if $k-2l- 1>0$, then \begin{align}
		\label{W-act-monwedge1}
			W_{l+\half} \monWedge_{\beta} = \begin{cases}
				q^{2l+2} \monWedge_{\beta} - q^{2l+1} \monWedge_{\flip{\target{l}}{\beta}}
				, &\text{ if } \beta_{l+1} = +,\\
				- q^{2l+1} \monWedge_{\flip{\target{l}}{\beta}}
				- q^{2l} \monWedge_{\beta} 
				, &\text{ if } \beta_{l+1} = - .
			\end{cases}
	\end{align}
	This is easily computed using \cref{OmegaDef} and the second assertion we just proved.
	Now to see that $W_{l+\half}\monWedge_\beta=0$, let $ \monWedge(\ytabcenter{-\\+} )=\monWedge_\gamma$ where the box diagram $\gamma$ is obtained from $\beta$ by setting $\gamma(l)=-,\gamma(l+1)=+$ and $\gamma(i)=\beta(i)$ for all other $i$. Similarly, we also define
	 $\monWedge(\ytabcenter{+\\-} )$, $\monWedge(\ytabcenter{+\\+})$ and $\monWedge(\ytabcenter{-\\-})$.
	Then we have \begin{align}\label{XY-action}
		Y   \monWedge_{\beta } = \monWedge(\ytabcenter{-\\+} ) + q\inv \monWedge(\ytabcenter{+\\-})
		,\qquad
		X \monWedge(\ytabcenter{-\\+} )  = q \monWedge(\ytabcenter{~\\\pm})
		,\qquad
		X \monWedge(\ytabcenter{+\\-}) = \monWedge(\ytabcenter{~\\\pm}).
	\end{align}
	The induction hypothesis gives 
	\begin{align*}
		q^2  XW Y \monWedge_{\beta }
		&=
		q^2 X W (\monWedge(\ytabcenter{-\\+} ) + q\inv \monWedge(\ytabcenter{+\\-} )) 
		\\
		&= q^2 X (u^-_{l-1} \monWedge(\ytabcenter{-\\+} ) + v^-_{l-1} \cancel{\monWedge(\ytabcenter{+\\+}) + q\inv u^+_{l-1} \monWedge(\ytabcenter{-\\-} )}+ q\inv v^+_{l-1}\monWedge(\ytabcenter{+\\-} ))
		\\
		&= q^2 (q u^-_{l-1} +  q\inv u^+_{l-1}) \monWedge_{\beta },
	\end{align*}
	where $ u^-_{l-1} = q^{2l-2}$ and $v^+_{l-1} = q^{2l}$, hence $(q u^-_{l-1} +  q\inv u^+_{l-1}) =0$ and we are done.
\end{proof}

\begin{prop}\label{Waction2}
	Let $\beta\in \hwboxdiag(\linearGraph_r')$ with $|\beta\inv(\pm)|=l$ and $\beta(0) = \bullet$. Assume further that $k-2l-1>0$ and $\img \beta \subset \{\bullet,\pm ,+\}$.
	Then \[
		W_{i+\half } \Omega_\beta = \begin{cases}
			(-q^{2k-2l-1} - q) \Omega_\beta , &\text{ if } i=k-l-1,\\
			  0 , &\text{ if } i> l, \, i\ne k-l-1.
		\end{cases}
	\]
	Similarly, if $k-2l-1>0$ and $\img \beta \subset \{\bullet,\pm ,-\}$,
	then \[
		W_{i+\half } \Omega_\beta = \begin{cases}
			(q^{2k=2l-3} - q) \Omega_\beta , &\text{ if } i=k-l-1,\\
			  0 , &\text{ if } i> l, \, i\ne k-l-1.
		\end{cases}
	\]
\end{prop}
\begin{proof}
	We again argue by induction on $i$, the base case $i=l$ is given by the second statement of \cref{Waction1}. Now assume $i>l$. Denote $W=W_{\half+i-1}$, $X= B_{\half+i}, Y= B_{-\half -i} $, and assume that the desired formula holds for $i=l-1$. 
	Noted that $[X,Y]=0$ and $X\monWedge_{\beta }=0$. Again by \cref{W-induction}, 
	\[
		W_{i+\half} \monWedge_{\beta } =
		(q^2  XW Y - qWXY +(q^2-1) XY \dCoideal_0\dCoideal_{i}\inv )\monWedge_{\beta }
		.
	\]
	Denote $\monWedge(\ytabcenter{-\\+})=\monWedge_{\beta}$ where $\beta$ is the box diagram defined by $\beta(i)= -$ and $\beta(i+1) = +$, and $\beta(i)$ not defined for any other $i$. Define also $\monWedge(\ytabcenter{?\\??}) $ similarly for $?,??\in \{+,-,\pm\}$.
	Then we have \[
		\Omega_{\negHWboxdiag} = \begin{cases}
			\Omega_{\negHWboxdiag[i-2]} \wedge (\monWedge(\ytabcenter{-\\-}) + q^3 \monWedge(\ytabcenter{-\\+}) -  q^4  \monWedge(\ytabcenter{+\\-}) - q^3 \monWedge(\ytabcenter{+\\+})) \wedge w'
			, & \text{ if } l<i<k-l ,\\
			\Omega_{\negHWboxdiag[i-2]} \wedge (\monWedge(\ytabcenter{-\\~}) - q^2  \monWedge(\ytabcenter{+\\~})), & \text{ if } i=k-l,
		\end{cases}
	\]
	for some $w'\in \bigwedge^{\bullet}\vecrep$. By \cref{XY-action}, in the first case we have \[
		Y \Omega_{\negHWboxdiag} =\Omega_{\negHWboxdiag[i-2]} \wedge ( q^4 \monWedge(\ytabcenter{~\\\pm}) -  q^4  \monWedge(\ytabcenter{~\\\pm}) ) \wedge w' =0.
	\]
	This proves the $l<i<k-l-1$ case. Now if $i=k-l-1$, by induction hypothesis we have $W\Omega_{\negHWboxdiag[i-2]} \wedge \monWedge(\ytabcenter{~\\?}) = \lambda_{i-1}\Omega_{\negHWboxdiag[i-2]} \wedge \monWedge(\ytabcenter{~\\?})$. Consequently, 
	\[
		q^2 XWY\Omega_{\negHWboxdiag} =  q^2 XW(\Omega_{\negHWboxdiag[i-2]} \wedge (\monWedge(\ytabcenter{~\\-}) - q^2  \monWedge(\ytabcenter{~\\+})))
		= q^2\lambda_{i-1} \Omega_{\negHWboxdiag} . 
	\]
	Again we have $XY \Omega_{\negHWboxdiag}=\Omega_{\negHWboxdiag} $ and hence $WXY\Omega_{\negHWboxdiag} =0 $ and $(q^2-1)XY \dCoideal_0\dCoideal_i\inv \Omega_{\negHWboxdiag} = (q^3-q) \Omega_{\negHWboxdiag} $. Consequently, the induction hypothesis gives \[
		W_{i+\half}\Omega_{\negHWboxdiag}  = (-q^{2i-2-1+2}- q^3 +q^3-q ) \Omega_{\negHWboxdiag} . 
	\]
	This concludes the proof of the $+\in \img\beta $ case. The other case follows from the exact same argument with a different initial condition and therefore omitted.
\end{proof}

\subsection{Highest weight vectors in $\bigwedge^k\vecrep$}
\begin{thm}\label{hwvecs} 
	The set $\{\Omega_{\beta} \mid \beta \in \hwboxdiag(\linearGraph)\}$ is a complete set of highest weight vectors in $\bigwedge^k\vecrep$. Furthermore, for $\beta\in\hwboxdiag(\linearGraph)$ satisfying $l=|\beta\inv(\pm)|$,
	the $\csaCoideal$-weight $(n_i\mid\zeta_i)$ of $\Omega_{\beta}$ is given as follows. If $\linearGraph=\linearGraph_r$, then 
	\begin{align}\label{hweven}
		n_i = 
		\begin{cases}
			2, & \text{ if } \beta(i) = \pm  ,\\
			1, & \text{ if } \beta(i) = + \text{ or } - ,\\
			0, & \text{ if } i\notin \supp(\beta) ,
		\end{cases}
		\qquad
		\zeta_i =
		\begin{cases}
			1 , & \text{ if } i=l , \beta(\target{i})=+ ,\\
			-1 , & \text{ if } i=l , \beta(\target{i})=- ,\\
			-q\inv , & \text{ if } i=k-l , \beta(\source{i})=+ ,\\
			q\inv , & \text{ if } i=k-l , \beta(\source{i})=- ,\\
			0, & \text{otherwise}.
		\end{cases}
	\end{align}
	And if $\linearGraph=\linearGraph_r'$, then 
	\begin{align}\label{hwodd}
		n_i = 
		\begin{cases}
			2, & \text{ if } \beta(i) = \pm  \text{ or } \bullet,\\
			1, & \text{ if } \beta(i) = + \text{ or } - ,\\
			0, & \text{ if } i\notin \supp(\beta) ,
		\end{cases}
		\quad
		\zeta_i =
		\begin{cases}
			q^{2l}-1 , & \text{ if } i=l ,\, \beta(\target{i})=+ ,\\
			-q^{2l-2}-1 , & \text{ if } i=l ,\, \beta(\target{i})=- ,\\
			-q^{2k-2l-1}-q , & \text{ if } i=k-l , \,\beta(\source{i}) = +,\\
			q^{2k-2l-3}-q , & \text{ if } i=k-l , \,\beta(\source{i}) = -, \\
			q^{2k-1}-q\inv , & \text{ if } i=k , \,\beta=\negHWboxdiag,\\
			0, & \text{otherwise}.
		\end{cases}
	\end{align}
	In particular, the irreducible submodule generated by $\Omega_\beta$ is isomorphic to $\mathbb{L}(n_i\mid\zeta_i)$.
\end{thm}
\begin{proof}
	Let $\beta\in \hwboxdiag(\linearGraph_r)$.
	From \cref{Omega-is-weigh-vec}, $\Omega_\beta$ is a highest weight vector. Moreover, since we have one of them in each irreducible summand of $\bigwedge^k\vecrep $, and by the uniqueness of highest weight vectors in \cref{fd-quotients}, we contend that they form a complete set of highest weight vectors in $\bigwedge^k\vecrep $. It remains to compute their weights.

	First assume $\linearGraph=\linearGraph_r$, we compute the $Z_i$-weights of $\Omega_\beta$ as follows.
	Write $\Omega_{\beta}=\sum a_\gamma \monWedge_\gamma$, we will identify which summands contribute to $\monWedge_\beta$ after applying $Z_i$ to $\Omega_{\beta}$ and compute the corresponding coefficients. 
	First, if $\source{i}=\target{i}\in \{\pm\}$,
	i.e. both are in the orange or white regions in \cref{Zaction-picture}, then all summands $\monWedge_\gamma$ are killed by $Z_i$ because in view of \cref{Bi-action-std-wedge}, there is no room to move the symbols or no symbols to move. 

	Next consider the case $\beta(\source{i}) = \pm, \beta(\target{i}) \in \{+,-\}$. In terms of \cref{Zaction-picture}, $\source{i}$ lies in the orange region and $\target{i}$ lies in the blue region. 
	In this case by definition we have $\Omega_\beta = \monWedge_\beta \pm \monWedge_{\flip{\target{i}}\beta}+ \sum a_\gamma \monWedge_\gamma$, where the sum is over box diagrams $\gamma$ obtained by flipping some signs in $\beta$. 
	By \cref{Zaction1}, we have $Z_i\Omega_\beta = \monWedge_{\flip{\target{i}}\beta} \pm \monWedge_\beta +$ rest, where the rest does not contribute to $\monWedge_\beta$ or $\monWedge_{\flip{\target{i}}\beta}$. From this we conclude that $Z_i\Omega_\beta = \pm \Omega_\beta$, where the sign is the same as $\beta(\target{i})$.

	Now consider the case $\beta(\source{i}) =\beta(\target{i}) \in \{+,-\}$. In terms of \cref{Zaction-picture}, both $\source{i}$ and $\target{i}$ lie in the blue region. 
	By \cref{Zaction2}, in this case we have $Z_i\monWedge_\beta= b\monWedge_{\beta} +\sum_{p\in \linearGraph^+_r} a_p \monWedge_{\flip{p}{\beta}}$ for some $b,a_p\in \funcField$.
	The only summands $\monWedge_\gamma$ that contribute to $\monWedge_\beta$ after applying $Z_i$ are $\monWedge_{\flip{\target{i}}\beta}$ and $\monWedge_{\flip{\source{i}}\beta}$. 
	Unwinding the definition, observe that in this case $\Omega_\beta = c(-q\monWedge_{ \flip{\target{i}}\beta} + \monWedge_{\flip{\source{i}}\beta}) + $ rest, where $c$ is a nonzero scalar. 
	Inserting \cref{Zaction2}, we have $Z_i (-q\monWedge_{\flip{\target{i}}\beta} + \monWedge_{\flip{\source{i}}\beta}) = 0$. This implies that $Z_i\Omega_\beta =0$ in this case.

	It remains to consider the case $\beta(\source{i}) \in \{+,-\}, \target{i}\notin \supp(\beta)$, and in terms of \cref{Zaction-picture}, $\source{i}$ lies in the blue region and $\target{i}$ lies in the white region. We compute the $Z_i$-weight by induction on $k-2l$. Suppose $\beta(\source{i})=+$. For the base case $k-2l=1$ we have \[
		\Omega_{\beta} = \monWedge_{\beta} + \monWedge_{\beta_-}
		\qquad
		\Omega_{\beta_-} =\monWedge_{\beta} + \monWedge_{\beta_-}
		,
		\text{ where }
		\beta = \,{\cdots}\,{\ytabcenter{\pm&+&~}} \, ,\;
		\beta_-=\, \cdots\,{\ytabcenter{\pm&-&~}} \,
		,
	\]
	(the box diagrams are rotated $90^\circ$ anti-clockwise to save space).
	By \cref{Zi-act-tailend}, we have \[
		Z_{i}\Omega_\beta = -q\inv \monWedge_{\flip{\source{i}}\beta} -q\inv \monWedge_\beta
		 = -q\inv \Omega_\beta
		,\qquad
		Z_{i}\Omega_{\flip{\source{i}}\beta} = -q\inv \monWedge_\beta + q\inv \monWedge_{\flip{\source{i}}\beta}
		=q\inv \Omega_{\flip{\source{i}}\beta}
		.
	\]
	Now for $k-2l>1$, write $j=\source{i},s=\target{i}$, then we have $\Omega_\beta= \Omega_{\beta \setminus j}\wedge (\basis_{j} + c \basis_{\foldVertex{j}})$. We compute \begin{align*}
		Z_{i} \Omega_\beta 
		&= B_{-i} Z_{i-1} \Omega_{\beta \setminus j}\wedge (\basis_{s} + c \basis_{\foldVertex{s}}) 
		=  B_{-i} (- q\inv \Omega_{\beta \setminus j}) \wedge (\basis_{s} + c \basis_{\foldVertex{s}})
		= - q\inv \Omega_\beta
		,
	\end{align*}
	where we used the induction hypothesis in the second equality. 
	The case $\beta(\source{i})=-$ is completely analogous and is omitted. Finally, if $\img\beta=\pm$, then $\Omega_\beta$ is a scalar multiple of $\monWedge_\beta$, and it follows from the last statement of \cref{Zaction1} that $Z_i \Omega_\beta=0$ for all $i$. 

	Finally, assume $\linearGraph=\linearGraph_r'$, the $W_i$-weights of $\Omega_\beta$ are determined in \cref{W-half-action}, \cref{Waction1} and \cref{Waction2}. This concludes the proof.
\end{proof}

\begin{eg}
	We list the $Z_i$-weights corresponding to the following box diagrams (rotated $90^\circ$ anti-clockwise to save space) in a table: 
	\begin{center}
	\begin{tabular}{c|c|c|c|c|c}
		 & $\ytabcenter{+&+&{}}$ & $\ytabcenter{-&-&{}}$ & $\ytabcenter{\pm&+&{}}$ & $\ytabcenter{\pm&-&{}}$ & $\ytabcenter{\pm& \pm&{}}$ \\
		 \hline
		$Z_0$ & $1$ & $-1$ & $0$ &  $0$ & $0$ \\
		$Z_1$ &	$0$ & $0$ & $1$ &  $-1$ & $0$ \\
		$Z_2$ &	$-q\inv$ & $q\inv$ & $-q\inv$ &  $q\inv$ & $0$ 
	\end{tabular}
	\end{center}
\end{eg}

Comparing \cref{hwvecs} with \cref{classical-hw}, we conclude that irreducible $\Uq'$-submodules of $\bigwedge^k\vecrep$ are nice quantum analogues of the irreducible $\classicalFixPointSubalg$-modules $\bigwedge^{k-i} W^+ \boxtimes \bigwedge^i W^- \subset \operatorname{Res}^{\gl(\vecrep[\C])}_{\classicalFixPointSubalg}\bigwedge^k\vecrep[\C]$.
However, there is no quantum analogue of the decomposition $U(\classicalFixPointSubalg) = U(\gl(W^+) )\boxtimes U(\gl(W^-))$, so it is nontrivial that we get the same type of decomposition of the exterior power!

\section{The kernel of $\functorWebB$}
Let still $\linearGraph\in\{\linearGraph_r,\linearGraph_r'\}$. We determine the kernel of the functor $\functorWebB:\webCategoryB_{\funcField}\to \FundCoideal$.

\subsection{Specialisation of $\coideal(\linearGraph)$}
We now describe the specialisation from $\coideal(\linearGraph)$ to $U(\classicalFixPointSubalg)$.
Let $\intUq(\linearGraph)\subset \Uq'(\linearGraph)$ be the $\groundring$-subalgebra generated by \[
	\left\langle  B_i ,\,  
	\dCoideal_p^{-1},\, \hat{L}_p\coloneqq \frac{\dCoideal_p-1}{q-1} \,\bigg|\, i\in \graphEdges_{\linearGraph^+}, \,p>0 
	\right\rangle. 
\] Recall the Lie algebra generators $\hat{\matrixElmtFont{d}}_p,\matrixElmtFont{b}_i$ of $\classicalFixPointSubalg$ from \cref{classical-sym-pair}. 
\begin{lemma}
	Let $\linearGraph\in\{\linearGraph_r,\linearGraph_r'\}$.
	There is a surjective $\groundring$-algebra homomorphisms $\xi: \intUq\twoheadrightarrow U(\classicalFixPointSubalg)$ such that $\xi(\hat{L}_p) = \hat{\matrixElmtFont{d}}_p\,(p>0)$ and $\xi(B_i)= \matrixElmtFont{b}_i\,(i\in\graphEdges_{\linearGraph})$. 
\end{lemma}
\begin{proof}
	Recall that for any linear graph $\linearGraph$, the specialisation of $\Uq(\linearGraph)$ can be done as follows:
	Let $U_\groundring(\linearGraph)\subset \Uq(\linearGraph)$ be the $\groundring$-subalgebra generated by \[
		E_i,\, F_i ,\; (i\in \graphEdges_{\linearGraph})
		\qquad
		D_p^{-1},\, L_p\coloneqq \frac{D_p-1}{q-1} ,\; (j\in \graphVertices_{\linearGraph}). 
	\]
	Note that we have for $i\in \graphEdges_{\linearGraph}$ \[
		[K_i;0] \coloneqq
		\frac{K_i-K_i\inv }{q-q\inv } 
		=    
		\frac{L_{\source{i}}D_{\target{i}} \inv - L_{\target{i}} D_{\source{i}}\inv }{1+q\inv}
		+ \frac{D_{\target{i}}\inv - D_{\source{i}}\inv}{q-1} 
		.
	\]
	Then by the PBW basis theorem one easily verifies that $U_\groundring(\linearGraph)$ has a presentation with these generators and relations given by the Serre relations and 
	\[
		[E_i,F_i]
		=    
		\frac{L_{\source{i}}D_{\target{i}} \inv - L_{\target{i}} D_{\source{i}}\inv }{1+q\inv}
		+ \frac{D_{\target{i}}\inv - D_{\source{i}}\inv}{q-1} 
		,\;\; (i\in \graphEdges_{\linearGraph})
		,\quad 
		1-D_p \inv = (q-1) L_jD_p\inv , \;\; (p\in \graphVertices_{\linearGraph}).
	\]
	Therefore, there is a surjective $\groundring$-algebra homomorphism $U_\groundring(\linearGraph)\twoheadrightarrow U(\gl(\linearGraph))$ given by 
	\[	
		D_p \mapsto 1,\quad 
		L_p\mapsto \matrixElmtFont{l}_p
		, \; (j\in \graphVertices_{\linearGraph}) \quad 
		E_i \mapsto \matrixElmtFont{e}_i ,\quad 
		F_i\mapsto \matrixElmtFont{f}_i , \; (i\in \graphEdges_{\linearGraph}). 
	\]
	We refer to \cite[XVI,XVII]{kassel-quagrp-book} for a detailed account of the deformation theory of the universal enveloping algebra.
	Now let $\linearGraph\in\{\linearGraph_r,\linearGraph_r'\}$ and consider $\coideal(\linearGraph)$.
	Since the subalgebra $C\subset \coideal(\linearGraph)$ generated by all $\dCoideal_p\,(p\ne 0 )$ and $B_i\, (i\notin \{0,\pm\half\})$ is isomorphic to $\Uq(\gl_{r+1})$, we get a well-defined $\groundring$-algebra homomorphism $\xi: C \to U(\classicalFixPointSubalg)$ given by \[
		\dCoideal_p \mapsto 1 ,\quad
		\hat{L}_p \mapsto \hat{\matrixElmtFont{d}}_p , \; (p\in \graphVertices_{\linearGraph_r^+})\quad 
		B_i \mapsto \matrixElmtFont{b} _i , \; (i\in \graphEdges_{\linearGraph_r}),
	\]
	where $\hat{\matrixElmtFont{d}}_p, \matrixElmtFont{b} _i $ are as in \cref{gens-g-fixed}.
	Now the defining relations involving $B_0$ and $B_{\pm\half}$ are the Serre-type relations \cref{serre relations 1}--\cref{SR-Gamma-prime}. These relations all have coefficients in $\groundring$ hence can be specialised. Setting $q=1$, these specialised relations in the classical generators $\hat{\matrixElmtFont{d}}_p,\matrixElmtFont{b}_i$  are easily checked in $U(\classicalFixPointSubalg)$. This proves that $\xi$ is well-defined. It is surjective because it hits all generators of $\classicalFixPointSubalg$.
\end{proof}

\subsection{The kernel}
Fix a linear graph $\linearGraph_r$ or $\linearGraph_{r}'$ with the corresponding vector space $\vecrep$. 
Define
\begin{align}
	\finker_{k,\linearGraph} = \prod_{i=n}^{m} (~\dotNotationText-\Qint{k-2i}) \in \webCategoryB, \text{ where } 
	\quad
	\begin{aligned}
		 m&= \begin{cases}
		\min(k,r+1) , & \text{ if } \linearGraph=\linearGraph_r, \\
		\min(k,r) , & \text{ if } \linearGraph=\linearGraph_r',
	\end{cases}	
	\\
	n &=\max(0,k-r-1)
	\end{aligned}
	\label{def-f-kGamma}
\end{align}

Recall that a \emph{tensor ideal} of an $\groundring$-linear monoidal category $\catFont{C}$ is an ideal in the locally unital algebra $\oplus_{a,b\in \catFont{C}} \hom_{\catFont{C}}(a,b)$ that is closed under tensor product with any morphisms in $\catFont{C}$ from both sides. Similarly, let $\catFont{C}$ be an $\groundring$-linear  monoidal category and $\webker $ be a right module category over $\catFont{C}$, then a \emph{submodule} of $\webker $ is an ideal of $\oplus_{a,b\in\webker}\webker(a,b)$ that are closed under tensor products from the right side with morphisms in $\catFont{C}$. 
\begin{definition} 
	Let $\webker\subset \webCategoryB$ be the $\webCategoryA$-submodule generated by the identity on the object $(\dim \vecrep+1)$ together with $\finker_{k,\linearGraph}$ for all $k\leq \dim \vecrep$. 
\end{definition}

Note that each $\webker(\underline{a},\underline{b})$ is free over $\groundring$, because $\groundring$ is a PID and by \cref{basis-typeB-Web}, $\webCategoryB(\underline{a},\underline{b})$ is free over $\groundring$ of finite rank. Furthermore, by \cref{mainTheorem}, the functor $\functorWebB$ factors through $\webker_\funcField = \funcField\otimes_{\groundring}\webker$. We want to show that $\webker_\funcField =\ker \functorWebB$ via a specialisation argument. Let $\webker_\C=\C\otimes_{\groundring} \webker$. 
	
\begin{lem}\label{tensor-ideal}
	The submodule $\ \webker_\C$ is a tensor ideal of $\webCategoryB_\C$. 
\end{lem}
\begin{proof}
	Take $f\colon \underline{a} \to \underline{b} \in \webker_\C$ and any $g\colon \underline{n} \to \underline{d} \in \webCategoryB\otimes_{\groundring} \C$.
	By definition of $\webker$, we have $f\otimes g \in \webker_\C$. It remains to show that $g\otimes f \in \webker_\C$. Let $X_{12}:\underline{a}\otimes \underline{n} \to \underline{n}\otimes \underline{a}$ and $X_{21}:\underline{b}\otimes \underline{d} \to \underline{d}\otimes \underline{b}$ be the braidings between the sources and targets of $f$ and $g$. Then we have $g\otimes f = X_{21} (f\otimes g) X_{12}$, which is in $\webker_\C$ because $f\otimes g$ is. 
\end{proof}

We next use a specialisation argument to show that $\webker$ is exactly the kernel of the functor $\functorWebB$.
Let $\vecrep[\groundring] \subset \vecrep$ be the $\groundring$-submodule spanned by the standard basis elements. It is clearly a free $\groundring$-module. 
Then by \cref{Bi-action-std-wedge}, $\intUq \bigwedge^{\underline{a}} ~\vecrep[\groundring]\subset \bigwedge^{\underline{a}} ~\vecrep[\groundring]$ for all ${\underline{a}} $ and we can make sense of $ \hom_{{\intUq}}(\bigwedge^{\underline{a}} ~\vecrep[\Aring], \bigwedge^{\underline{b}} ~\vecrep[\Aring])$. 
Moreover, in view of \cref{merge1nFormula}---\cref{merge-n1-formula-last} and the $K$-matrix, the functor $\functorWebB$ restricts to an $\groundring$-linear map
\begin{align}
	j: \hom_{\webCategoryB}(\underline{a},\underline{b}) \longrightarrow \hom_{{\intUq}}(\bigwedge^{\underline{a}} ~\vecrep[\Aring], \bigwedge^{\underline{b}} ~\vecrep[\Aring])
\end{align}

By \cref{kar-env-WebB} we may consider $\webCategoryB$ as a full subcategory of $\polyRepCat[{\classicalFixPointSubalg[\linearGraph_\infty]}]$. Recall \cref{intro-prop} we have a decomposition $\vecrep[\C] = W^+\oplus W^-$ into $\pm1$-eigenspaces of $\kmatrix$, this holds also in $\polyRepCat[{\classicalFixPointSubalg[\linearGraph_\infty]}]$. Denote $L_i= \bigwedge^iW^+ \boxtimes \bigwedge^{k-i} W^- \in \polyRepCat[{\classicalFixPointSubalg[\linearGraph_\infty]}]$ for $i\in \N$.
\begin{lemma}\label{J=J'}
Let $J,J'$ be the tensor ideal of $\polyRepCat[{\classicalFixPointSubalg[\linearGraph_\infty]}]$ generated by $\C\otimes_{\Aring}\webker$, respectively by the identity maps on  $L_i$  for $i<n$ or $i>m$, where  $n,m $ are as in \cref{def-f-kGamma}. Then  $J=J'$.
\end{lemma}
\begin{proof} 
	By \cref{mainTheorem} and \cref{intro-prop}, inside $\polyRepCat[{\classicalFixPointSubalg[\linearGraph_\infty]}]$ we have $\dotNotationText=\sum_{0\leq i\leq k} -(k-2i) \id_{L_i}$, this implies $J\subset J'$. Moreover, we have inside $\polyRepCat[{\classicalFixPointSubalg[\linearGraph_\infty]}]$ \[
		\finker_{k,\linearGraph} =  \sum_{i<n  \text{ or } i>m} a_ i \id_{L_i} , \text{ where } a_i= \prod_{n\leq j\leq m}-(2j-2i).
	\]
	Thus, for $i<n  \text{ or } i>m$, we have $a_i\ne a_{i'} $ if $i\ne i'$. Via Lagrange interpolation we have $\id_{L_i} = g_i(\finker_{k,\linearGraph})$ for some polynomial $ g_i$. This implies the reverse inclusion.
\end{proof}

\begin{prop}\label{compare-ranks}
	The following hold for any $\underline{a},\underline{b}$:
	\begin{enumerate}
		\item $ \hom_{{\intUq}}(\bigwedge^{\underline{a}} ~\vecrep[\Aring], \bigwedge^{\underline{b}} ~\vecrep[\Aring])$ is free. 
		\item  $\rank_{\groundring} \img j = \rank_{\groundring} \hom_{{\intUq}}(\bigwedge^{\underline{a}} ~\vecrep[\Aring], \bigwedge^{\underline{b}} ~\vecrep[\Aring])$.
		\item 
		$\C \otimes_{\Aring} \hom_{{\intUq}}(\bigwedge^{\underline{a}} ~\vecrep[\Aring], \bigwedge^{\underline{b}} ~\vecrep[\Aring])
		=
		\hom_{\classicalFixPointSubalg}(\bigwedge^{\underline{a}} ~\vecrep[\C], \bigwedge^{\underline{b}} ~\vecrep[\C])$
		\item The kernel of $\functorWebB_\C$ is precisely $\webker_\C$.
		\item $\rank_{\groundring} \webCategoryB_{\Aring}(\underline{a},\underline{b}) -\rank_{\groundring} \webker_{\Aring}(\underline{a},\underline{b}) = \dim_\C \hom_{\classicalFixPointSubalg}(\bigwedge^{\underline{a}} ~\vecrep[\C], \bigwedge^{\underline{b}} ~\vecrep[\C]) $. 
	\end{enumerate}
\end{prop}
\begin{proof}
	Since $\groundring$ is a PID,
	$ \hom_{{\intUq}}(\bigwedge^{\underline{a}} ~\vecrep[\Aring], \bigwedge^{\underline{b}} ~\vecrep[\Aring])\subset \hom_{{\groundring}}(\bigwedge^{\underline{a}} ~\vecrep[\Aring], \bigwedge^{\underline{b}} ~\vecrep[\Aring])$ is a free $\groundring$-module, and $\img j \subset  \hom_{{\intUq}}(\bigwedge^{\underline{a}} ~\vecrep[\Aring], \bigwedge^{\underline{b}} ~\vecrep[\Aring])$ is a free $\groundring$-module.
	Then by fullness of $\functorWebB$, $j\otimes_{\groundring}\funcField$ is surjective, so $\coker j$ is torsion and therefore, $\rank_{\groundring} \img j = \rank_{\groundring} \hom_{{\intUq}}(\bigwedge^{\underline{a}} ~\vecrep[\Aring], \bigwedge^{\underline{b}} ~\vecrep[\Aring])$.
	Moreover, since all coefficients in \cref{Bi-action-std-wedge} lie in $q^\Z$, the action of $\intUq $ on $ \bigwedge^{\underline{a}} ~\vecrep[\Aring]$ specialises to the action of $U'(\linearGraph)$ on $\bigwedge^{\underline{a}} ~\vecrep[\C]$. This implies that $\C \otimes_{\Aring} \hom_{{\intUq}}(\bigwedge^{\underline{a}} ~\vecrep[\Aring], \bigwedge^{\underline{b}} ~\vecrep[\Aring]) \subset \hom_{\classicalFixPointSubalg}(\bigwedge^{\underline{a}} ~\vecrep[\C], \bigwedge^{\underline{b}} ~\vecrep[\C]) $. On the other hand, in view of \cref{merge1nFormula}---\cref{merge-n1-formula-last} and the $K$-matrix, $\functorWebB_\C$ brings the generating webs of $\webCategoryB_\C$ into $\C \otimes_{\Aring} \hom_{{\intUq}}(\bigwedge^{\underline{a}} ~\vecrep[\Aring], \bigwedge^{\underline{b}} ~\vecrep[\Aring]) $. By \cref{SW-duality-q=1}, the image of these webs generate $ \hom_{\classicalFixPointSubalg}(\bigwedge^{\underline{a}} ~\vecrep[\C], \bigwedge^{\underline{b}} ~\vecrep[\C]) $, so we have the opposite inclusion.

	It remains to determine the kernel of $\functorWebB_\C$. For $r=\infty$, by \cref{classical-webB-action-ss} $\functorWebB_\C$ is faithful. Now for $r<\infty$, denote by $\vecrep[\C](\linearGraph_\infty) , \vecrep[\C](\linearGraph)$ the vector representations associated to $\linearGraph_\infty $ and $\linearGraph\in \{\linearGraph_r,\linearGraph_r'\}$ respectively.
	Consider the projection $\operatorname{pr}:\polyRepCat[{\classicalFixPointSubalg[\linearGraph_\infty]}]\to \polyRepCat$ given by $\vecrep[\C](\linearGraph_\infty) \mapsto \vecrep[\C](\linearGraph)$. We have the following commutative diagram
	\begin{equation}
		\begin{tikzcd}[font=\normalsize]
			\webCategoryB_\C \arrow[d,hook] \arrow[dr,"\functorWebB_\C"] \\
			\polyRepCat[{\classicalFixPointSubalg[\linearGraph_\infty]}]\arrow[r,"\operatorname{pr}"] & \polyRepCat
		\end{tikzcd}
		\label{factorisation-GammaB-C}
	\end{equation}
	Furthermore, let $\vecrep[\C](\linearGraph)= W^+_\linearGraph \oplus W^-_\linearGraph$ be the decomposition into $\pm1$ eigenspaces of $\kmatrix$, then we have $\operatorname{pr} (L_i )= \bigwedge^{k-i}  W^+_\linearGraph \boxtimes \bigwedge^i  W^-_\linearGraph$, where $L_i\in \polyRepCat[{\classicalFixPointSubalg[\linearGraph_\infty]}]$ is as in \cref{J=J'}. Since both $\polyRepCat[{\classicalFixPointSubalg[\linearGraph_\infty]}]$ and $ \polyRepCat$ are semisimple, the kernel of $\operatorname{pr}$ is generated by $\id_{L_i}$ for $k-i>\dim W^+_\linearGraph$ or $i>\dim W^-_\linearGraph$. In fact, this is precisely the tensor ideal $J'=J$ from \cref{J=J'}. To see this, note that \[
		\dim W^+_\linearGraph = r+1
		,\quad 
		\dim W^-_\linearGraph = 
		\begin{cases}
			r+1, & \text{ if } \linearGraph = \linearGraph_r,\\
			r,  & \text{ if } \linearGraph = \linearGraph_r',
		\end{cases}
	\]
	which implies that $\bigwedge^{k-i}  W^+_\linearGraph \boxtimes \bigwedge^i  W^-_\linearGraph = 0$ if and only if $n> i$ or $i> m$ where $m,n$ are as in \cref{def-f-kGamma}. This shows that $\ker \operatorname{pr}=J$ is generated by $\C\otimes_{\groundring}\webker$. Since $\webCategoryB_\C $ is a full subcategory of $\polyRepCat[{\classicalFixPointSubalg[\linearGraph_\infty]}]$, we conclude that the kernel of  $\functorWebB_\C$ is generated by $\webker$.

	For the last equality we consider the maps \[
		\begin{tikzcd}[row sep=small,font=\normalsize]
			 \webker_{\Aring}(\underline{a},\underline{b}) \arrow[r,hook,"i"] &
			 \webCategoryB_{\Aring}(\underline{a},\underline{b}) \arrow[r,"j"] &
			 \hom_{{\intUq}}(\bigwedge^{\underline{a}} ~\vecrep[\Aring], \bigwedge^{\underline{b}} ~\vecrep[\Aring])
			 \\
			\webker_{\Aring}(\underline{a},\underline{b}) \arrow[r,hook,"i'"] & 
			\ker j \arrow[r,"j'",two heads] &
			\ker j/\img i
		\end{tikzcd}
	\]
	Note that $\ker j$ is free over $\Aring$ of finite rank. 
	By the previous statement, $i'\otimes \C$ is bijective, whence $\ker j/\img i$ is torsion and consequently $\rank_{\groundring} \webker_{\Aring}(\underline{a},\underline{b}) = \rank_{\groundring} \ker j$. This implies that \[
		\rank_{\groundring} \webker_{\Aring}(\underline{a},\underline{b}) = \rank_{\groundring} \webCategoryB_{\Aring}(\underline{a},\underline{b})
		-\rank_{\groundring} \img j
		= \rank_{\groundring} \webCategoryB_{\Aring}(\underline{a},\underline{b})
		- \rank_{\groundring} \hom_{{\intUq}}(\bigwedge^{\underline{a}} ~\vecrep[\Aring], \bigwedge^{\underline{b}} ~\vecrep[\Aring])
	\]
	where we used the second statement in the second equality. Combined with the third statement, we get the desired identity. This finishes the proof.
\end{proof}

Recall from \cref{webA-functor} the $\funcField$-linear monoidal equivalence $ \bar\Phi^A:\webCategoryA_{\funcField}/\id_{2\dimV+1}   \cong \Fund$. Note that by definition, $\functorWebB\colon \webCategoryB_\funcField\to \FundCoideal$ is bijective on objects. We now determine the kernel of $\functorWebB$ as a $\webCategoryA_{\funcField}$-submodule. Abbreviate $\webker_\C = \C\otimes_{\groundring}\webker$. 
\begin{thm}
	The functor $\functorWebB $ induces an equivalence of categories $\bar\Phi:\webCategoryB_{\funcField}/\webker_\funcField \cong \FundCoideal$ compatible with $\bar\funcA$ in the sense that the following diagram commutes: \[
		\begin{tikzcd}[font=\normalsize]
			\webCategoryB_{\funcField}/\webker_\funcField \times \webCategoryA_{\funcField}/\id_{\dim \vecrep+1}  
			\arrow[r]\arrow[d,"\bar\Phi^B\times \bar\Phi^A"] 
			 & \webCategoryB_{\funcField}/\webker_\funcField
			  \arrow[d,"\bar\Phi^B"] 
			\\
			\FundCoideal\times \Fund \arrow[r,"\otimes"] & \FundCoideal
		\end{tikzcd}
	\]
\end{thm}
\begin{proof}
	For any objects $(\underline{a},\underline{b})$ with $|\underline{a}|=|\underline{b}|$, 
	by \cref{compare-ranks} 
	 \[
		\rank_{\groundring}
		\hom_{\intUq}(\bigwedge^{\underline{a}} ~\vecrep[\Aring], \bigwedge^{\underline{b}} ~\vecrep[\Aring])
		= 
	 	\dim \hom_{\classicalFixPointSubalg}(\bigwedge^{\underline{a}} ~\vecrep[\C], \bigwedge^{\underline{b}} ~\vecrep[\C])
		=
		\operatorname{rank} \webCategoryB_{\Aring}(\underline{a},\underline{b}) / \webker_{\Aring}(\underline{a},\underline{b})
		.
	\]	
	Since $\rank_{\groundring}
		\hom_{\intUq}(\bigwedge^{\underline{a}} ~\vecrep[\Aring], \bigwedge^{\underline{b}} ~\vecrep[\Aring])
		= \dim \hom_{\coideal} (\bigwedge^{\underline{a}} ~\vecrep, \bigwedge^{\underline{b}} ~\vecrep)$,
	$\webker_{\funcField}$ must be the kernel of the map $\webCategoryB_{\funcField}(\underline{a},\underline{b}) \to  \hom_{\coideal} (\bigwedge^{\underline{a}} ~\vecrep, \bigwedge^{\underline{b}} ~\vecrep)$. 
	We conclude that $\webker_{\funcField}$ is precisely the kernel of $\functorWebB$ and $\webCategoryB_{\funcField}/\webker_\funcField \cong \FundCoideal$. Finally, the commutativity of the diagram follows from \cref{module-functor}.
\end{proof}

\appendix
\section{Computation of $F_{[i,-i]}*F_{[i,-i]}$}

We prove \cref{Ri*Ri-computation} by induction on $i$. The base cases for $\linearGraph_r$ and $\linearGraph_r'$ respectively are:
\[
	R_0*R_0 = F_0 ^2 + \frac{q}{q-q\inv}
	,\qquad
	R_{\half} *R_{\half}  = R_\half^2  + (	q^4-q^2)F_{\half} F_{-\half}\kCoideal_{\half}^{-1} +  \frac{q^3}{q-q\inv } 
	.
\]
For the induction step we first observe the following easy identities:
\begin{align*}
	&	P = [R,X]_q = R*X-qX*R
	,\qquad
	R*X= RX ,\quad R*Y = RY ,\quad X*R = XR ,\quad Y*R= YR 
	,\\
	&[R_i,X]_{q\inv}=[R_i,Y]_q=0
	,\quad 
	R_i * Y  = R_i Y 
	,\quad 
	R_i*X = R_i X  - P \kCoideal_i\inv
	,\quad
	Y*R_i  = YR_i -\kCoideal_i \inv Q
	,\\
	&
	R_i = [Y,P]_q =[Q,X]_q
	= Y*R*X - qX*Y*R  - q R*Y*X+q^2 X*R*Y +R\kCoideal_{i} \inv,
\end{align*}
where $X=F_i , Y=F_{-i}$, $R=R_{i-1}$, 
and $P = F_{[i,1-i]}, Q= F_{[i-1,-i]}$. We then compute
\begin{align*}
	{R_{i}*R_{i}} &= R_i*Y* R*X -q R_i*R*Y*X -qR_i*X*Y*R   + q^2  R_i*X*R*Y
	+  R_i*R*\kCoideal_i\inv
	\\
	&=(R_iY) * R*X -q (R_iR)*Y*X -q(R_iX- P \kCoideal_i\inv)*Y*R   
	\\
	& \quad  + q^2  (R_iX- P \kCoideal_i\inv)*R*Y
	+  R_iR\kCoideal_i\inv
	\\
	&= q (YR_i) * R*X -q (R_iR)*Y*X -(XR_i)*Y*R   + q  (XR_i)*R*Y
	+  R_iR\kCoideal_i\inv
	\\
		& \quad +  P \kCoideal\inv* (qY *R -q^2  R*Y)
		\\
	&= q Y*R_i * R*X   +q \kCoideal_i\inv *Q*R*X
		-q (R_iR)*Y*X 
	-(X*R_i)*Y*R  
	\\
	& \quad + q  X*R_i*R*Y
	+  R_iR\kCoideal_i\inv +  q^2 P * (qY *R  -q^2  R*Y)\kCoideal\inv
	\\
	&= q Y*R_i * R*X   +q \kCoideal_i\inv *Q*R*X
	-q (R_iR)*Y*X 
	-qX*(Y*R_i+ \kCoideal_i\inv Q)*R  
	\\
	& \quad + q  X*R_i*R*Y
	+  R_iR\kCoideal_i\inv +  q^2 P * (qY *R  -q^2  R*Y)\kCoideal\inv
	\\
	&= \underset{\circled{1}}{\underbrace{q Y*R_i * R*X
		-q (R_iR)*Y*X 
	-q X*Y*R_i*R  
	+ q  X*R_i*R*Y
	+  R_iR\kCoideal_i\inv}}
	\\
	& \quad + \underset{\circled{2}}{\underbrace{q^3( P*Y*R - q P*R*Y+ q^{-2} Q*R*X- X*Q*R) \kCoideal_{i}\inv }}
	.
\end{align*} 
Then by \cref{root-vec-skew-der} and \cref{rectangle-commutator-rootvec--Uq}, we have the following easy identities: \begin{align*}
	&\delustig_{i} R_i =(q\inv-q)Q
	,\quad 
	\delustig_{-i} R_i =
	\partial^R_{-i} R = 0
	,\quad 
	\partial^R_{-i} R_i = (q\inv -q)P 
	,\quad
	\partial^R_{-i} Q = (q\inv-q) R
	,\\
	&
	[P,R]_q=0
	,\quad 
	Y*(R_iR)= Y R_i R - \kCoideal_i\inv QR
	,\quad P*Y = PY 
	,\quad 
	Y*P = YP - R\kCoideal_{i}\inv 
	.
\end{align*}
Using these we compute each term separately: 
\begin{align*}
	qY*R_i*R*X 
	&= 
	q (Y R_i R - \kCoideal_i\inv QR)X +  \frac{q}{q-q\inv}\partial^R_{-i} (YR_iR)\kCoideal_i\inv-  \frac{q}{q-q\inv}\partial^R_{-i} (QR)\kCoideal_i^{-2}
	\\
	&= 
	q (Y R_i R )X + (-q QRX   - YPR +\uwave{ \frac{q}{q-q\inv}  R_i R})\kCoideal_i\inv + R^2 \kCoideal_i^{-2}
	,\\
	-q (R_iR)*Y*X 
	&= 
	-q R_iRYX - \frac{q}{q-q\inv}\partial^R_{-i} (R_iRY)\kCoideal_i\inv 
	\\ 
	&= 
	-q R_iRYX  \uwave{-\frac{q}{q-q\inv}R_iR\kCoideal_i\inv } +q^2 PRY\kCoideal_i\inv 
	,\\
	-q X*Y*R_i*R  
	&=  -qX* ( YR_i R- \kCoideal_i\inv QR)
	=
	-q X Y R_i R +q^3XQR  \kCoideal_i\inv- \frac{q\kCoideal_i \partial^L_{-i} (YR_iR)}{q-q\inv}
	\\
	&=-q X Y R_i R +q^3XQR  \kCoideal_i\inv  \uwave{- \frac{q}{q-q\inv}\kCoideal_i R_iR }
	,\\
	q  X*R_i*R*Y
	&=
	q X R_i R Y + \frac{q}{q-q\inv}\kCoideal_i \partial^L_{-i} (R_iRY)
	=
	q X R_i R Y \uwave{+ \frac{q}{q-q\inv}\kCoideal_i R_iR}
	. 
\end{align*}
Together with $[P,Q]= (q-q\inv ) RR_i$, this gives 
\begin{align*}
	\circled{1} & = \uwave{qYR_iRX} -q QR X\kCoideal_{i}\inv - YPR\kCoideal_{i}\inv + R^2 \kCoideal_{i}^{-2}
	\uwave{-q R_iRYX} +q ^2PRY\kCoideal_{i}\inv
	\uwave{-qXYR_iR} 
	\\
	&\qquad + q^3 XQR\kCoideal_{i}\inv +\uwave{qXR_iRY} +  R_iR\kCoideal_i\inv
	\\
	&=R_i^2 +  R^2 \kCoideal_{i}^{-2}
	-2q QP\kCoideal_{i}\inv-(2q^2-1)RR_i \kCoideal_{i}\inv
	=R_i^2 +  R^2 \kCoideal_{i}^{-2}  -2qPQ  \kCoideal_{i}\inv  - RR_i  \kCoideal_{i}\inv
	. 
\end{align*}
Here we used further the following auxiliary identities:
\begin{align*}
	-YPR+q^2PRY & = -q YRP +q^2RPY = -q YRPO +q^2R(YP-R_i) = -qQP-q^2 RR_i
	,	\\
	q^3 XQR- q QRX & = q^2(QX-R_i)R -qQRX = -qQP-q^2R_iR
\end{align*}
Next, one checks easily the following identities:\begin{align*}
	P *Y = PY 
	,\quad
	Y*P = YP - \kCoideal_{i} \inv R
	,\quad 
	Q *X = QX- R\kCoideal_{i}\inv
	\\
	Q  = [Y,R]_q  = Y*R-q*R*Y
	,\qquad
	P*R = q\inv R*R*X- q\inv R*P -qX*R*R.
\end{align*}
We use them to compute the terms of $\circled{2}$:
\begin{align*}
	q^3 P*Y*R  \kCoideal_{i}\inv 
	& =  	q^3(PY)*R  \kCoideal_{i}\inv 
	=	q^2 (YP-R_i)*R  \kCoideal_{i}\inv 
	=     q^2 (Y*P+   \kCoideal_i \inv R-R_i ) *R\kCoideal_i\inv 
	\\
	& =  q Y* R*R*X   \kCoideal_i\inv  \underline{-q Y*R*P\kCoideal_i\inv} -q^3 Y*X*R*R\kCoideal_i\inv 
	\\
	&\qquad+q^2 R*R \kCoideal_i^{-2} -q^2 R_iR\kCoideal_{i}\inv  
	,\\
	-q^4 P*R*Y \kCoideal_{i}\inv 
	& =  
	-q^3 R*R* X* Y \kCoideal_{i}\inv +q^3 R*P*Y \kCoideal_{i}\inv +q^5 X*R*R*Y \kCoideal_{i}\inv 
	\\
	&= 		-q^3 R*R*X*Y \kCoideal_{i}\inv +q^2 R*(Y*P+ R\kCoideal_i\inv -R_i) \kCoideal_i\inv +q^5 X*R*R*Y \kCoideal_{i}\inv 
	\\
	& = 	-q^3 R*R*X*Y \kCoideal_{i}\inv +\underline{q^2 R*Y*P\kCoideal_i\inv} + q^2 R*R\kCoideal_i^{-2}  
	\\
	&\qquad  -q^2 RR_i \kCoideal_i^{-1}  +q^5 X*R*R*Y \kCoideal_{i}\inv 
	\\
	-q^3 X*Q*R \kCoideal_{i}\inv  & = -q^3 (XQ) *R \kCoideal_{i}\inv 
	=  - q ^2(QX-R_i ) *R  \kCoideal_{i}\inv 
	\\
	& =  -q^2(Q*X+ R\kCoideal_{i}\inv ) *R\kCoideal_{i}\inv +q^2 R_iR\kCoideal_{i}\inv   
	\\
	& =  -qQ*(R*X-P)\kCoideal_{i}\inv  - q^2  R*R\kCoideal_{i}^{-2} +q^2 R_iR\kCoideal_{i}\inv 
	.
\end{align*}
Inserting the corresponding terms into the expression of $\circled{1}, \circled{2}$ respectively gives
\begin{align*}
	\circled{2} & =	q Y* R*R*X   \kCoideal_i\inv  \underline{-q Y*R*P\kCoideal_i\inv} -q^3 Y*X*R*R\kCoideal_i\inv 
	+q^2 R*R \kCoideal_i^{-2} -q^2 R_iR\kCoideal_{i}\inv  
	\\
	&\qquad	-q^3 R*R*X*Y \kCoideal_{i}\inv +\underline{q^2 R*Y*P\kCoideal_i\inv} + q^2 R*R\kCoideal_i^{-2}  
		-q^2 RR_i \kCoideal_i^{-1}  +q^5 X*R*R*Y \kCoideal_{i}\inv 
	\\
	&\qquad+\cancel{q Q*R*X\kCoideal_{i}\inv} -qQ*(\cancel{R*X}-P)\kCoideal_{i}\inv  - q^2  R*R\kCoideal_{i}^{-2} +q^2 R_iR\kCoideal_{i}\inv
	\\
	& =	G
	+q^2 R*R \kCoideal_i^{-2} 
	-q^2 R_iR\kCoideal_{i}\inv  
		\cancel{- q Q*P \kCoideal_{i}\inv }  + \cancel{qQ*P\kCoideal_{i}\inv}  
	=G 
	+q^2 R*R \kCoideal_i^{-2} 
	-q^2 R_iR\kCoideal_{i}\inv  ,
\end{align*}
where denoting $T=R*R$, \begin{align*}
		G &= q Y* T*X   \kCoideal_i\inv   
		-q^3 Y*X*T\kCoideal_i\inv 
		-q^3 T*X*Y \kCoideal_{i}\inv
		+q^5 X*T*Y \kCoideal_{i}\inv 
			.
\end{align*}
Now by induction hypothesis, denoting $u_j =  F_{[i-1,j] } k_j F_{[-j,1-i]} , \, c=c_{i-1}$, we compute
\begin{align*}
	qY*T*X  \kCoideal_i\inv   & = qY*( (R^2 + \sum u_j + c)X)  \kCoideal_i\inv  
	\\
	&= qYTX   \kCoideal_i\inv    +  \frac{q^{-1}}{q-q\inv}   \kCoideal_i\inv R^2  \kCoideal_i\inv   
		+ \frac{1}{q-q\inv}  \kCoideal_i\inv \sum u_j   \kCoideal_i\inv     + c\frac{q}{q-q\inv}  \kCoideal_i^{-2} 
	,\\
	-q^3Y*X*T\kCoideal_i\inv  & = -q^3 YXT   \kCoideal_i\inv- \frac{q^3}{q-q\inv }  T \kCoideal_i^{-2} 
	,\\
	-q^3 T*X*Y \kCoideal_{i}\inv & = -q^3 TXY \kCoideal_{i}\inv - \frac{q^3}{q-q\inv }  T 
	,\\
		q^5 X*T*Y \kCoideal_{i}\inv  & = q^5 X*((R^2+ \sum u_j + c)Y) \kCoideal_{i}\inv 
		\\
		&=  q^5 XTY \kCoideal_{i}\inv   + \frac{q^3}{q-q\inv} R^2 
	+ \frac{q^3}{q-q\inv } \sum u_j  + \frac{q^5}{q-q\inv} c 
		.
\end{align*}
Inserting these terms into the expression for $G$ gives 
\begin{align*}
	G & = \underset{\circled{3}}{\underbrace{(qYTX -q^3 YXT-q^3TXY +q^5 XTY)  \kCoideal_{i}\inv}} +  \frac{cq}{q-q\inv} \kCoideal_{i}^{-2} + \frac{q\inv-q^3}{q-q\inv} R^2\kCoideal_{i}^{-2}
	\\
	& \qquad + \frac{1}{q-q\inv}  \kCoideal_i\inv \sum u_j   \kCoideal_i\inv   - \frac{q^3}{q-q\inv }  (\sum u_j +c)\kCoideal_i^{-2} - \frac{q^3}{q-q\inv }  \sum u_j 
	+ \frac{q^4}{q-q\inv } \sum u_j  +\frac{q^5-q^3}{q-q\inv} c 
	\\
	& = \circled{3} - c q^2  \kCoideal_{i}^{-2}  - (1+q^2) R^2\kCoideal_{i}^{-2} + \frac{q-q^3}{q-q\inv} \sum u_j \kCoideal_{i}^{-2} + q^4c
	\\
	& = \circled{3} - c q^2  \kCoideal_{i}^{-2}  - (1+q^2) R^2\kCoideal_{i}^{-2} -q^2 \sum u_j \kCoideal_{i}^{-2}  + q^4c
	.
\end{align*}
Since $[X,Y]=0 $ possibly except in the base case, we have
\begin{align*}
	\circled3 & = (\underset{\circled{4}}{\underbrace{qYR^2X -q^3 YXR^2-q^3R^2XY +q^5 XR^2Y}})  \kCoideal_{i}\inv
	+ \sum (\underset{\circled{5}}{\underbrace{qYu_jX -q^3 YXu_j-q^3u_jXY +q^5 Xu_jY}})  \kCoideal_{i}\inv
\\
&\qquad + c(q+q^5-2q^3)XY  \kCoideal_{i}\inv
.
\end{align*} 
Using $[P,R]_q=[Q,R]_{q\inv}=0$ and $Q=[Y,R]_q, \, P=[R,X]_q$, we compute \begin{align*}
	qYR^2X & = q(qRY+Q) R X = q^2R(qRY+Q)X + qQRX =q^3 R^2 YX+(q^2+1) RQX 
	\\
	& = q^3 R^2 YX+(q^2+1) RR_i +(q^3+q) RXQ
	,\\
	-q^3 YXR^2 & = -q^3XYR^2 = -q^3X(qRY+Q)R  = -q^4XRYR -q^3XQR 
	\\
	& = -q^5 XR^2Y -(q^4+q^2) XRQ,
\end{align*}
\begin{align}
	\circled{4}
	&= \cancel{q^3 R^2 YX}+(q^2+1) RR_i +(q^3+q) RXQ
	\cancel{-q^5 XR^2Y }\nonumber
	\\
	&\qquad -(q^4+q^2) XRQ
		\cancel{-q^3 R^2 XY} +\cancel{q^5 XR^2Y }
		= (q^2+1) RR_i +(q^3+q)PQ.
		\label{circled4}
\end{align}
It remains to compute $\circled{5}$. Put $x_j=F_{[i-1,j]}$, $y_j=F_{[-j,1-i]}$, so that
$u_j=x_jk_jy_j$. Denote further $P_j = F_{[i,j]}  ,\, Q_ j =F_{[-j,-i]}$.
By \cref{rectangle-commutator-rootvec--Uq},
\[
x_jX=qXx_j+P_j,\qquad Yy_j=qy_jY+Q_j
,\qquad [X,Y] =0
, \qquad [X,Q_j]=0 \; (j<i-1).
\]
We then compute 
\begin{align*}
	qYu_jX
	&=q^2x_jk_j(Yy_j)X
	=q^2x_jk_j(qy_jY+Q_j)X
	=q^4 (x_j X ) k_jy_jY + q^3 (x_j X ) k_j Q_j 
	\\
	&=q^4P_j k_jy_j Y +q^3 P_jk_jQ_j 
	+\uwave{ q^5 Xu_jY }+ q^4 Xx_jk_jQ_j
	,
	\\
	-q^3YXu_j
	&=-q^4 Xx_jk_j(Yy_j)
	=-q^4 Xx_jk_j Q_j \uwave{ -q^5 Xu_jY }
	,
	\\
	-q^3 u_j XY &=-q^4 x_jX k_jy_j Y
	= -q^4 Pk_jy_jY -q^5 Xu_jY
	.
\end{align*}
Whence $\circled{5} = \sum q^3 P_j k_j Q_j$ and 
\begin{align}
		\circled{3} \kCoideal_{i}\inv = c' XY  \kCoideal_{i}\inv  + (q^2+1) RR_i  \kCoideal_{i}\inv  +(q^3+q)PQ  \kCoideal_{i}\inv  + \sum q^3 P_j k_j Q_j    \kCoideal_{i}\inv 
		\label{circled3}
\end{align}
Combining \cref{circled3} gives \begin{align*}
	R_i*R_i & = 
	R_i^2 +  \cancel{R^2 \kCoideal_{i}^{-2}}  -2qPQ  \kCoideal_{i}\inv  \cancel{-(1+q^2) RR_i  \kCoideal_{i}\inv }  + q^2(\cancel{R^2}+\sum u_j + c)\kCoideal_{i}^{-2}
	\\
	& \quad + c' XY  \kCoideal_{i}\inv  + \cancel{(q^2+1) RR_i  \kCoideal_{i}\inv  }+(q^3+q)PQ   \kCoideal_{i}\inv   \kCoideal_{i}\inv + \sum q^3 P_j k_j Q_j   \kCoideal_{i}\inv 
	- c q^2  \kCoideal_{i}^{-2} 
	\\
	& \quad  \cancel{- (1+q^2) R^2\kCoideal_{i}^{-2} }-q^2 \sum u_j \kCoideal_{i}^{-2} + q^4c
	\\ 
	& = 
	R_i^2 -2qPQ  \kCoideal_{i}\inv   + \cancel{q^2(\sum u_j + c)\kCoideal_{i}^{-2}} + c' XY  \kCoideal_{i}\inv  +(q^3+q)PQ   \kCoideal_{i}\inv + \sum q^3 P_j k_j Q_j   \kCoideal_{i}\inv 
	\\
	& \quad 
	\cancel{- c q^2  \kCoideal_{i}^{-2}  -q^2 \sum u_j \kCoideal_{i}^{-2} } + q^4c
	\\ 
	& = 
	R_i^2  + c' XY  \kCoideal_{i}\inv  +(q^3-q)PQ   \kCoideal_{i}\inv + \sum q^3 P_j k_j Q_j   \kCoideal_{i}\inv + q^4c
	.
\end{align*}
This finishes the induction and concludes the proof.

\bibliography{references}
\end{document}